%% file: g3fin_master.tex
\documentclass[a4paper]{article}

\usepackage{amsmath, amscd,  amssymb, latexsym, amsthm, xspace, color}
\usepackage{ifpdf}
\usepackage[rgb,cmyk,dvipsnames]{xcolor}
\usepackage{tikz}
\usetikzlibrary{matrix,arrows}
\usetikzlibrary{positioning,arrows}
\usetikzlibrary{decorations.markings}
\usepackage{breqn}
\usepackage{paralist}
\usepackage{graphicx}
\usepackage{subfigure}
\usepackage[all]{xy}
\usepackage{psfrag}
\usepackage{epsfig,subfigure}
\usepackage{caption} 
\usepackage{epstopdf} 

\usepackage{paralist}
\usepackage[boxed]{algorithm2e}

\ifpdf
\else
\usepackage{epsfig}
\fi

\numberwithin{equation}{section}

\newcommand{\hyp}{{\rm hyp}}
\newcommand{\odd}{{\rm odd}}
\newcommand{\red}{{\rm red}}

\newcommand{\spec}{{\rm Spec}\,}
\newcommand{\spe}[1]{\widetilde{{#1}}}
\newcommand{\cu}{{\mathcal{C}_c}}
\newcommand{\subgrpunion}[2]{{#1}^{[#2]}}
\newcommand{\sgu}[2]{{#1}^{[#2]}} 

\newcommand{\IA}{{\bf A}}
\newcommand{\IC}{{\bf C}}

\newcommand{\IG}{{\bf G}}
\newcommand{\IH}{{\bf H}}
\newcommand{\IN}{{\bf N}}
\newcommand{\IP}{{\bf P}}
\newcommand{\IQ}{\mathbf{Q}}
\newcommand{\IQbar}{\overline{\bf Q}}

\newcommand{\IR}{{\bf R}}
\newcommand{\IZ}{{\bf Z}}

\newcommand{\oa}[1]{{#1}^{\rm oa}}
\newcommand{\Qta}[1]{{#1}^{\IQ,\rm ta}}
\newcommand{\ssm}{\smallsetminus}

\newcommand{\supp}[1]{{\rm Supp}({#1})}

\newcommand{\heightS}{h}
\newcommand{\height}[1]{\heightS({#1})}

\newcommand{\cx}{{\IC}} 
\newcommand{\half}{\IH}

\newcommand{\PP}{\IP}

\newcommand{\ratls}{\IQ} 
\newcommand{\reals}{\IR} 
\newcommand{\proj}{\IP}
\newcommand{\zed}{\IZ}

\newcommand{\Poincare}{Poincar\'e\xspace}

\newcommand{\Teichmuller}{Teich\-m\"uller\xspace}

\newcommand{\Neron}{N\'eron\xspace}

\newcommand{\APTCs}{algebraically primitive \Teichmuller curves\xspace}
\newcommand{\APTC}{algebraically primitive \Teichmuller curve\xspace}
\newcommand{\APVS}{algebraically primitive Veech surface\xspace}

\newcommand{\MM}{\mathbb{M}}
\newcommand{\LL}{\mathbb{L}}
\newcommand{\UU}{\mathbb{U}}

\newcommand{\barmoduli}[1][g]{{\overline{\mathcal M}}_{#1}}
\newcommand{\bdry}{\partial}

\newcommand{\GLtwoRplus}{{\rm GL}_2^+(\reals)}
\newcommand{\Gm}{{\bf G}}

\renewcommand{\tilde}{\widetilde}
\renewcommand{\Im}{\IM}

\newcommand{\RM}[1][\mathcal{O}]{{\cal RM}_{#1}}

\newcommand{\isom}{\cong}

\newcommand{\moduli}[1][g]{{\mathcal M}_{#1}}

\newcommand{\omoduli}[1][g]{{\Omega\mathcal M}_{#1}}

\newcommand{\pomoduli}[1][g]{{\proj\Omega\mathcal M}_{#1}}

\newcommand{\qtq}[1]{\quad\text{#1}\quad}

\newcommand{\SL}{{\mathrm{SL}}}
\newcommand{\SLtwoR}{ {\mathrm{SL}_2 (\reals)}}

\newcommand{\Sp}{\mathrm{Sp}}

\newcommand{\E}[1][\mathcal{O}]{\mathcal{E}_{#1}}

\newcommand{\bcX}{\overline{\mathcal{X}}}

\newcommand{\br}{{\boldsymbol{r}}}

\newcommand{\bt}{{\boldsymbol{t}}}
\newcommand{\bba}{\boldsymbol{a}}

\newcommand{\be}{{\boldsymbol{e}}}

\newcommand{\bm}{\boldsymbol{m}}
\newcommand{\bq}{\boldsymbol{q}}

\newcommand{\bz}{{\boldsymbol{z}}}

\DeclareMathOperator{\IM}{Im}

\DeclareMathOperator{\Aff}{Aff}

\DeclareMathOperator{\Aut}{Aut}
\DeclareMathOperator{\CR}{CR}

\DeclareMathOperator{\Jac}{Jac}
\DeclareMathOperator{\Pic}{Pic}
\DeclareMathOperator{\Ker}{Ker}

\DeclareMathOperator{\modulus}{mod}

\DeclareMathOperator{\Resi}{Res}

\DeclareMathOperator{\Tr}{Tr}
\DeclareMathOperator{\Thick}{Thick}

\DeclareMathOperator{\rank}{rank}

\DeclareMathOperator{\NST}{NST}

\newcommand{\tir}{{\tilde{r}}}

\DeclareMathOperator{\Div}{Div}

\def\be{\begin{equation}}   \def\ee{\end{equation}}     \def\bes{\begin{equation*}}    \def\ees{\end{equation*}}
\def\ba{\be\begin{aligned}} \def\ea{\end{aligned}\ee}   \def\bas{\bes\begin{aligned}}  \def\eas{\end{aligned}\ees}

\newtheorem{theorem}{Theorem}[section] 
\newtheorem{prop}[theorem]{Proposition} 
\newtheorem{cor}[theorem]{Corollary}
\newtheorem{lemma}[theorem]{Lemma}

\newtheorem{conj}[theorem]{Conjecture}

\theoremstyle{definition}
\newtheorem{definition}[theorem]{Definition}
\newtheorem*{rem}{Remark}
\newtheorem{example}[theorem]{Example}

\makeatletter
\def\blfootnote{\xdef\@thefnmark{}\@footnotetext}

\renewcommand{\l@section}{\@dottedtocline{0}{1.5em}{2.3em}}
\renewcommand{\l@subsection}{\@dottedtocline{1}{3.8em}{3.2em}}
\renewcommand{\l@subsubsection}{\@dottedtocline{2}{7.0em}{4.1em}}

\makeatother

\newcommand{\codim}{{\rm codim}}

\newcommand{\cC}{\mathcal{C}}
\newcommand{\cD}{\mathcal{D}}

\newcommand{\cH}{\mathcal{H}}
\newcommand{\cI}{\mathcal{I}}
\newcommand{\cK}{\mathcal{K}}
\newcommand{\cL}{\mathcal{L}}
\newcommand{\cO}{\mathcal{O}}

\newcommand{\cS}{\mathcal{S}}
\newcommand{\cT}{\mathcal{T}}
\newcommand{\cW}{\mathcal{W}}
\newcommand{\cX}{\mathcal{X}}
\newcommand{\cY}{\mathcal{Y}}
\newcommand{\cV}{\mathcal{V}}
\newcommand{\cZ}{\mathcal{Z}}

\newcommand{\barcY}{\overline{\mathcal{Y}}}

\newcommand{\ol}{\overline}
\newcommand{\rel}{{\ol{\cX}/\ol{C}}}

\begin{document}

\bibliographystyle{halpha}

\title{\Teichmuller curves in genus three and \\ just likely intersections in $\IG_m^n\times\IG_a^n$.}
\author{Matt Bainbridge, Philipp Habegger and Martin M\"oller}
\maketitle
\setcounter{tocdepth}{0}

\begin{abstract}
  We prove that the moduli space of compact genus three Riemann surfaces contains
  only finitely many algebraically primitive \Teichmuller curves.
  For the stratum $\Omega\moduli[3](4)$, consisting of holomorphic one-forms with a single zero, our
  approach to finiteness uses the Harder-Narasimhan filtration of the Hodge bundle over a
  \Teichmuller curve to obtain new information on the locations of the zeros of eigenforms.  By
  passing to the boundary of moduli space, this gives explicit constraints on the cusps of
  \Teichmuller curves in terms of cross-ratios of six points on $\proj^1$.
  \par
  These constraints are akin to those that appear in Zilber and Pink's conjectures on unlikely
  intersections in diophantine geometry.  However, in our case one is lead naturally to the
  intersection of a surface with a family of codimension two algebraic subgroups of
  $\IG_m^n\times\IG_a^n$ (rather than the more standard $\IG_m^n$). The ambient algebraic group
  lies outside the scope of Zilber's Conjecture but we are nonetheless able to prove a sufficiently
  strong height bound.

  For the generic stratum $\Omega\moduli[3](1,1,1,1)$, we obtain global torsion order bounds through
  a computer search for subtori of a codimension-two subvariety of $\Gm^9$.  These
  torsion bounds together with new bounds for the moduli of horizontal cylinders in terms of torsion
  orders yields finiteness in this stratum.  The intermediate strata are handled with a mix of these
  techniques.
\end{abstract}
\tableofcontents
\newpage

\input{introduction}
\input{htboundthm}
\input{background}

\input{HNfilt}
\input{toruscontainment}

\input{minimalstrata}
\input{torsion_and_moduli}

\input{general_finiteness}

\input{principalstrata}
\input{otherstrata}

\bibliography{my}
\end{document}

%% file: introduction.tex
\section{Introduction}
\label{sec:introduction}


A closed Riemann surface $X$ of genus $g$ equipped with a nonzero holomorphic
quadratic differential $q$ determines an isometrically immersed hyperbolic
plane $\half\to \moduli[g]$ in the moduli space of genus $g$ Riemann
surfaces.  Occasionally this may cover an isometrically immersed
algebraic curve $C = \half / \Gamma \to \moduli[g]$.  Such a curve is
called an \emph{\Teichmuller curve}, and the pair $(X, q)$ is called a \emph{Veech surface}.
\par
The trace field of $\half/\Gamma$ is the number field $F = \ratls(\Tr(\gamma) : \gamma
\in \Gamma)$.  A \Teichmuller curve is said to be \emph{arithmetic} if
$F= \ratls$.  It is said to be \emph{algebraically primitive} if the
generating quadratic differential $q$ is the square of a holomorphic 
one-form $\omega$ and
the degree of $F$ attains its maximum, namely $[F : \ratls] = g$.   
In this case, we call the pair
$(X, \omega)$ an \emph{algebraically primitive Veech surface}.


While arithmetic \Teichmuller curves are dense in every $\moduli$,
algebraically primitive \Teichmuller curves seem to be much more rare.
There are infinitely many examples of algebraically primitive
\Teichmuller curves in $\moduli[2]$, discovered independently by Calta \cite{calta}
and McMullen \cite{mcmullenbild}, and it remains an open
problem whether there are infinitely many such curves for any larger
genus.  The aim of this paper is the following partial solution to
this problem.
\par
\begin{theorem} \label{thm:intromainfin}
There are only finitely many algebraically primitive \Teichmuller
curves in $\moduli[3]$.
\end{theorem}
\par
The methods used here do not use any dynamics of the \Teichmuller geodesic flow, with the exception
of the hyperelliptic locus in the stratum $\omoduli[3](2,2)^\odd$ (consisting of forms with two
double zeros which are fixed by the hyperelliptic involution).  We emphasize that our proofs of
Theorem~\ref{thm:intromainfin} are {\em effective}, in the sense that a reader keeping careful track
of constants at every step should arrive at an explicit bound for the number of \APTCs in any
stratum (except $\omoduli[3](2,2)^\odd$).  Good effective bounds would allow one to finish
classifying \Teichmuller curves in these strata with a computer search.  Unfortunately, the bounds
produced by our methods are so large that this is not feasible.
\par
In parallel to our work, Matheus and Wright showed \cite{MatWri13} that for every fixed genus $g$
which is an odd prime, there are only finitely many \APTCs generated by Veech surfaces with a single
zero. These results rely on recent results of Eskin, Mirzakhani, and Mohammadi (\cite{esmi},
\cite{esmimo}) on $\SL_2(\reals)$ orbit-closures in strata of holomorphic one-forms. In particular, these
methods are not effective.  We appeal to \cite{MatWri13} to obtain finiteness in the hyperelliptic
locus of the stratum $\omoduli[3](2,2)^\odd$, as none of our methods could handle this case. We give
a summary of the known results on the classification of \Teichmuller curves at the end of the
introduction.
\par
One essential ingredient of the proof of Theorem~\ref{thm:intromainfin} is a height bound for the
cusps of these \Teichmuller curves.  We obtain these bounds by applying methods used to attack
conjectures on unlikely intersection in the multiplicative group $\IG_m^n$ (whose complex points are
just $(\IC^*)^n$).  In our case, we are lead to study similar unlikely intersection problems in the
group $\IG_m^n\times\IG_a^n$.

The remaining techniques depend on the stratum the \Teichmuller curve
surface lies in.  In the case of few zeros the main new ingredient is an application of the
Harder-Narasimhan filtration of the Hodge bundle. In the case of many zeros we prove global torsion
order bounds and we use conformal geometry to derive bounds for ratios of moduli.  We now describe
these techniques in more detail.

\paragraph{Harder-Narasimhan filtrations.}

Consider an \APVS $(X, \omega)$ with trace field $F$.  One of the fundamental constraints on
$(X,\omega)$, established in \cite{moeller06}, is that the Jacobian of $X$ has real multiplication by an
order in $F$ with $\omega$ an eigenform.  This real multiplication in fact distinguishes $g$
eigenforms $\omega_1 = \omega, \omega_2, \ldots, \omega_g$ (up to constant multiple). These other
$g-1$ eigenforms are in general very mysterious from the point of view of the flat geometry of $(X,
\omega)$; however, in \S\ref{sec:HNfilt}, for \Teichmuller curves in certain strata we obtain some
information on the locations of the zeros of the other eigenforms.

More precisely, we denote by $\omoduli(n_1, \ldots, n_k)$ the moduli space of genus $g$ surfaces $X$
equipped with a holomorphic one-form $\omega$ having $k$ zeros of order given by the $n_i$.  The
minimal stratum $\omoduli(2g-2)$ has as one connected component the hyperelliptic component
$\omoduli(2g-2)^{\rm hyp}$, consisting entirely of hyperelliptic curves.  Here is one example of the
type of control we obtain on the zeros of the other eigenforms.  Similar statements are proved in
\S\ref{sec:HNfilt} for all genus three strata except for the generic stratum $\omoduli[3](1,1,1,1)$.

\begin{theorem}
  \label{thm:hyperelliptic_zeros}
  Suppose $(X, \omega)$ generates an \APTC $C$ in $\omoduli(2g-2)^{\rm hyp}$, with $p\in X$ the
  unique zero of $\omega$ of order $2g-2$.  Then the eigenforms $\omega_i$, listed in an appropriate
  order, have a zero of order $2g-2i$ at $p$.
\end{theorem}

The basic idea of the proof is to consider a canonical filtration of the Hodge bundle over $C$.
Every vector bundle over a projective variety has a canonical filtration, the Harder-Narasimhan
filtration.  For \Teichmuller curves in the strata under consideration, these filtrations were
computed by Yu and Zuo \cite{yuz1} in terms of the zero divisor of the family of one-forms
generating the \Teichmuller curve in the canonical family of curves over $C$ of.  Alternatively the
decomposition of the Hodge bundle into eigenform bundles yields a second filtration.  Uniqueness of
the Harder-Narasimhan filtration implies that these two filtrations are in fact the same, and
comparing them yields Theorem~\ref{thm:hyperelliptic_zeros}.

\paragraph{Finiteness in $\omoduli[3](4)$.}

The real multiplication condition, together with Theorem~\ref{thm:hyperelliptic_zeros} gives very
strong constraints on \APTCs in $\omoduli[3](4)^{\rm hyp}$.  Unfortunately, these conditions are
difficult to apply, as determining when the Jacobian of a given Riemann surface has real
multiplication and understanding its eigenforms is generally very hard.

We bypass this difficulty by studying the cusps of \Teichmuller curves.  Veech \cite{veech89}
established that every \Teichmuller curve has at least one cusp.  Exiting a cusp, the family of
Riemann surfaces degenerates to a noded Riemann surface equipped with a meromorphic one-form (a
\emph{stable form}).  By algebraic primitivity, this stable form has geometric genus zero, and
because we are in the minimal stratum, it is in fact irreducible.  More concretely, we may regard it
as $\proj^1$ with three pairs of  distinct points $(x_i, y_i)$, $i=1,2,3$ each identified to form a node. In the hyperelliptic stratum we have moreover $y_i = -x_i$.

The advantage of passing to the boundary is that our constraints become completely explicit.  More
precisely, consider the cross-ratio
\begin{equation*}
  R_i = [x_{i+1}, y_{i+1}, x_{i+2}, y_{i+2}],
\end{equation*}
with indices taken modulo three.  We established in \cite{BaMo12} that the real multiplication
condition is equivalent to
\begin{equation}
  \label{eq:constraint1}
  R_1^{b_1}R_2^{b_2}R_3^{b_3} = 1
\end{equation}
for some nonzero integers $b_i$.  In \S\ref{sec:background}, we show that the condition of
Theorem~\ref{thm:hyperelliptic_zeros} on the zeros of the  second eigenform $\omega_2$ is equivalent
to
\begin{equation}
  \label{eq:constraint2}
  b_1 x_1 + b_2 x_2 + b_3 x_3=0,
\end{equation}
where the $b_i$ are the same integers as above.  A similar constraint is established in the other
component $\omoduli[3](4)^{\rm odd}$.

By a theorem of \cite{BaMo12}, to prove finiteness of \Teichmuller curves in these strata, it is
enough to establish finiteness of cusps, thus these equations reduce the problem to an explicit
problem in number theory.

\paragraph{Unlikely intersections.}

The cross-ratios $R_i$ can be regarded as a diagonal embedding of a two-dimensional variety ${\cal
  Y}$ in the algebraic group $\IG_m^3 \times \IG_a^3$.  Allowing all possible coefficients, equations
\eqref{eq:constraint1} and \eqref{eq:constraint2} can then be interpreted as an intersection of
the surface ${\cal Y}$ with a countable collection of codimension-two subgroups.
Our problem is then most naturally regarded in the context of unlikely intersections in diophantine
geometry.  Unlikely intersections refer to vast conjectures due to Zilber \cite{Zilber} and Pink
\cite{Pink} and partially motivated by a theorem of Bombieri, Masser, and
Zannier \cite{BMZ99}.  
This last group of authors
considered the problem of intersecting a curve $\cX\subset\IG_m^n$
with the infinite union ${\cal H}\subset\IG_m^n$ of all
proper algebraic subgroups of $\IG_m^n$ and showed -- as long as $\cX$
itself is not contained in
the translate of  a
proper subgroup -- that $\cX\cap {\cal H}$ is a set of bounded height.
By height we mean the absolute logarithmic Weil height which we recall
further down in \S\ref{sec:onheights}.
A non-empty intersection of a curve $\cX$ with a
 subgroup of codimension one is more appropriately 
called
  ``just likely'', in contrast to
what is studied by Zilber and Pink's conjectures, where
the subgroups must have codimension at least two and the intersections
are deemed ``unlikely''.  
Under this more stringent condition and when imposing an appropriate
condition on $\cX$, one  expects finiteness instead of 
 merely bounded height.
However,  boundedness of height in the ``just likely'' situation
is often a gateway to proving  finiteness in the ``unlikely'' case. 


Zilber and Pink's conjectures are open in general. But
several cases that incorporate classical results
 such as the Mordell or Manin-Mumford Conjectures are known. We provide a partial overview 
of state of this field in \S\ref{sec:zilber}. 
\par 
 One  aspect 
 that sets our work apart from previous results 
 is that it mixes  the additive group of a field
$\IG_a$, which is unipotent, with the multiplicative group $\IG_m$,
which is not. The latter appears  in the literature
\cite{BMZ99,Zilber} on these conjectures
but  the former seems to lie  outside 
the general framework of the Zilber-Pink Conjectures.
The following theorem is proved in \S\ref{sec:htbound}.
There we also provide all the necessary definitions used in the
theorem's formulation.
\par
\begin{theorem}
\label{thm:htlogbound}
Let $\cY\subset\IG_m^n\times\IG_a^n$ be an irreducible closed surface
and let $\Qta{\cY}$ denote the 
complement in $\cY$ of the union of all rational semi-torsion
anomalous subvarieties of $\cY$ (for a definition see 
\S\ref{sec:inters-with-algebr}). There is a constant $c$
with the following property. If
$P=(x_1,\ldots,x_n,y_1,\ldots,y_n)\in \Qta{\cY}$ such that there is 
$(b_1,\ldots,b_n)\in\IZ^n\ssm\{0\}$ with
\begin{equation}
\label{eq:coupledeqn}
x_1^{b_1}\cdots x_n^{b_n}=1 \quad\text{and}\quad
  b_1y_1+\cdots + b_n y_n = 0, 
\end{equation}
then $\height{P}\le c \log(2[\IQ(P):\IQ])$. 
\end{theorem}

The equations (\ref{eq:coupledeqn}) define an algebraic subgroup of
$\IG_m^n\times\IG_a^n$ of codimension two. 
The fact that the coefficients on the additive side are coupled to the
exponents on the multiplicative side is essential for the proof and
for the application to Theorem~\ref{thm:intromainfin}. 

We emphasize that this height bound only applies to points in the subset
$\Qta{\cY}\subset {\cal Y}$ which we define precisely in \S\ref{sec:inters-with-algebr}.  
This set 
arises  by removing from $\cY$ certain subvarieties that have
anomalously large intersection with certain translates of algebraic subgroups.
It
bears similarities to Bombieri, Masser, and Zannier's 
\emph{open anomalous locus} $\oa{\cX}$  \cite{BMZGeometric} of a subvariety of algebraic torus
$\cX\subset \IG_m^n$. Indeed, the second named author \cite{BHC} proved a
height bound on points in $\oa{\cX}$ that are contained in an
algebraic subgroup of dimension at most $n-\dim \cX$.
However, $\Qta{\cY}$  can have a delicate structure.
We will see in Example \ref{ex:nonopen} that it need not in
general be Zariski open in $\cY$; its complement in $\cY$ can be a
countable infinite union of curves.  
  In general, it is difficult to determine the open anomalous locus.
The analogous problem  for a plane in $\IG_m^n$ is already a difficult problem which
was  solved by Bombieri, Masser, and Zannier
\cite{BMZPlanes}.
For our application,  $\Qta{\cY}$ will cause additional difficulties.  
  In \S\ref{sec:hypin22odd}, we compute $\Qta{\cY}$
  for the two cases arising from the two components of $\omoduli[3](4)$.
This is done via by amalgamating a theoretical analysis 
with the use of computer algebra software \cite{sage}.

\paragraph{A torus containment algorithm.}

At two places we rely on computer-assistance to establish the non-existence of tori in a given
subvariety of $\IG_m^n$.  In \S\ref{sec:torusalgo} we provide an algorithm that deals with that
problem effectively. The algorithm is designed so as to check only for tori whose character group is
contained in a specified subgroup of $\IZ^n$.  For the application in
\S\ref{sec:pres-tors-princ} we can restrict to such a situation, and only with this
restriction is the run-time of the algorithm reasonable.

\paragraph{Multiple zeros.}

The proof of finiteness for strata with multiple zeros is quite different and starts with the
torsion condition of \cite{moeller}. This states that if $(X, \omega)$ is an \APVS, and $p,q$ are distinct
zeros of $\omega$, then the divisor $p-q$ represents a torsion point of $\Jac(X)$.  As for the real
multiplication condition, this torsion condition may be interpreted explicitly at the boundary, and
we couple this with  other conditions to obtain finiteness of cusps.  We say that $(X, \omega)$ has
\emph{torsion dividing $N$} if the order of $p-q$ divides $N$ for every two zeros $p$ and $q$.
\par
A fundamental difficulty is that the limiting stable forms arising from \Teichmuller curves in these
strata may have thrice-punctured sphere components (pairs of pants), and none of our conditions give
any control on the one-form restricted to these components.  

For our approach to work, we need to know that controlling all irreducible components of limiting
stable curves, except for pants components, is enough to conclude finiteness of \Teichmuller
curves.  To formalize this,  we say that a collection of \Teichmuller curves is \emph{pantsless-finite} if the
collection of all non-pants irreducible components of limiting stable forms is finite.  In
\S\ref{sec:torsion-moduli} and \S\ref{sec:genfinthm}, we prove:

\begin{theorem}
  \label{thm:pantsless-finite}
  In any stratum of holomorphic one-forms, any pantsless-finite collection of \APTCs is in fact finite.
\end{theorem}
In
\S\ref{sec:pres-tors-princ} and \S\ref{sec:other-strata-genus}, we then show that for each
remaining stratum in genus three, the set of \APTCs is pantsless-finite.


\paragraph{Torsion and moduli: the Abel metric.} 

Its a simple observation that a pantsless-finite collection of \APTCs has uniform bounds on its
torsion orders (though proving these uniform bounds is in fact the difficult step in establishing
pantsless-finiteness, as we discuss below).  A Veech surface $(X, \omega)$ has a canonical flat
metric $|\omega|$.  In \S\ref{sec:torsion-moduli}, we study how this flat metric is controlled by
these torsion order bounds.  More precisely, $(X, \omega)$ has many  \emph{periodic directions} in which
the surface decomposes as a union of parallel flat cylinders whose moduli have rational ratios, and
the complement of these cylinders is a collection of parallel geodesic segments joining the zeros of
$\omega$ (called the \emph{spine} of this periodic direction).
The main ingredient in the proof of
Theorem~\ref{thm:pantsless-finite} is new bounds on the moduli of these cylinders in terms of the
torsion orders.

For a graph $\Gamma$ we define its {\em blocks} to be maximal subgraphs which cannot be disconnected
by removing any single vertex.  Every graph has a canonical decomposition into blocks, with any two
adjacent blocks meeting in a single vertex.  This notion applies in particular to the dual graph of
a periodic direction whose edges are cylinders and vertices are connected components of the spine
(equivalently, the dual graph of a corresponding stable curve over a cusp of the \Teichmuller
curve).  This induces a partition of the cylinders in any given periodic direction into blocks.  In
\S\ref{sec:torsion-moduli}, we show that bounds for torsion orders yield bounds for ratios of moduli
within any block.
\par
\begin{theorem}[Theorem~\ref{thm:moduli_bound}] \label{thm:intro_modulibound}
  Let $(X, \omega)$ be
  an algebraically primitive Veech surface, with torsion dividing $N$.  Then for any block of
  cylinders $C_1, \ldots, C_n$ of some periodic direction of $(X, \omega)$, the number of
  possibilities for the projectivized tuple $(\modulus(C_1): \ldots :\modulus(C_n))$ is bounded in
  terms of $N$ and $n$.
\end{theorem}

We emphasize that this theorem is only useful for strata with multiple zeros.  For strata with a
single zero, torsion orders are bounded trivially, but each block consists of one cylinder, so
Theorem~\ref{thm:intro_modulibound} gives no information.
\par
This theorem generalizes \cite[Theorem~2.4]{Mo08}, which establishes a special case of this bound
using properties of N\'eron models of the stable curves over the cusps of the \Teichmuller curve.
Our proof here uses instead conformal geometry and a new metric on $(X, \omega)$ coming from the
torsion condition.  The torsion condition determines via Abel's theorem a meromorphic function $f$
whose divisor $(f)$ is supported on the zeros of $\omega$ (in fact there are many possible functions
$f$ we can use here).  Pulling back the flat metric $dz/z$ on $\proj^1$ via this map, we obtain a
metric on $(X, \omega)$ with infinite cylinders at the zeros of $\omega$, which we call the
\emph{Abel metric}.

We show that for $(X, \omega)$ close to the boundary of moduli space, the finite cylinders of $(X, \omega)$
have moduli close to corresponding cylinders in the Abel metric.  Using the torsion condition, we
obtain constraints on the moduli of these cylinders in the Abel metric, which then give
corresponding constraints on the original cylinders of $(X, \omega)$.  Combining these constraints
for all of the possible choices of the initial function $f$, we obtain Theorem~\ref{thm:intro_modulibound}.

\paragraph{Geometry of cylinder widths.}

To complete the proof of Theorem~\ref{thm:pantsless-finite}, we need to be able to compare moduli of
cylinders in a periodic direction which do not belong to the same block.  An irreducible component
of the limiting stable curve which joins two blocks must have at least four punctures, so the
hypothesis of pantsless-finiteness means that we can control the geometry of a component connecting
two blocks.  In particular, for two cylinders which lie in different blocks but whose boundaries
share a common connected component of the spine, the ratio of the circumference of these cylinders
is an algebraic number $\lambda$ which can be assumed to be known (in particular, all of its Galois
conjugates are bounded).  We would like to use this control on ratios of circumferences to bound the
ratios of moduli of cylinders in these two blocks.  Ordinarily knowing only the widths or moduli of
cylinders gives no information about their moduli or widths, since their heights can be arbitrary.
In \S\ref{sec:genfinthm}, we see that on an \APVS the widths and moduli are intimately connected.

More precisely, consider a periodic direction of an \APVS $(X, \omega)$.  The widths of cylinders
are a collection of algebraic numbers (\emph{a priori} only defined up to scale, but in
\S\ref{sec:genfinthm}, we define a nearly canonical way to normalize them).  Considering their
different Galois conjugates, we can regard these widths as a collection of vectors $\{v_1, \ldots,
v_n\}\subset \reals^g$.  We group these vectors according to the blocks of cylinders defined above,
and call $B_i\subset\reals^g$ the span of the $i$th block.

\begin{theorem}
  \label{thm:orthogonal_blocks}
  The $B_i$ are pairwise orthogonal subspaces of $\reals^g$.  In each $B_i$, the set of width vectors
  $v_j$ it contains are determined by the moduli of the corresponding cylinders up to a similarity of $B_i$.  The norm of the width vectors in $B_i$ is inversely
  proportional to the moduli of the corresponding cylinders.
\end{theorem}

Combining Theorem~\ref{thm:intro_modulibound} and Theorem~\ref{thm:orthogonal_blocks} (for
simplicity in the case of a periodic direction with a single block), we see that 
given a bound on the torsion orders of $(X, \omega)$, there are finitely many possibilities for the
collection of width vectors up to similarity of $\reals^g$.

Theorem~\ref{thm:intro_modulibound} and Theorem~\ref{thm:orthogonal_blocks} imply that for any
pantsless-finite collection of \Teichmuller curves, there are uniform bounds for the ratios of
widths and moduli of cylinders in any periodic direction, as we can use the width-ratio $\lambda$
and Theorem~\ref{thm:orthogonal_blocks} to control the scale of the moduli of adjacent blocks of
cylinders.  From this, we conclude using the Smillie-Weiss ``no small triangles'' condition from
\cite{smillieweiss10} that our collection of \Teichmuller curves is in fact finite.

\paragraph{Bounding torsion orders in the generic stratum.}

In \S\ref{sec:pres-tors-princ}, we establish finiteness of \APTCs in the generic stratum
$\omoduli[3](1,1,1,1)$.  The key quantity governing the geometry of \Teichmuller curves in the
generic stratum is the torsion orders introduced above.  Given an \emph{a priori} bound for torsion
orders in this stratum,  we show using height estimates that the collection of \APTCs in this
stratum is pantsless-finite in the sense of \S\ref{sec:genfinthm}.  From
Theorem~\ref{thm:pantsless-finite}, we obtain:

\begin{theorem}
  \label{thm:torsion_implies_finite}
  Suppose that there is a uniform bound for the torsion orders of \APTCs in the generic stratum
  $\omoduli[g](1^{2g-2})$.  Then there are only finitely many \APTCs in this stratum.
\end{theorem}

It remains to establish these torsion bounds.  Consider an irreducible component of a cusp of a
\Teichmuller curve in this stratum which contains more than one zero of the stable form.  Marking
the zeros and poles of this stable form, we can regard it as a point in $\moduli[0,n]$ for some
$n$.  A choice of two zeros and two poles determines a cross-ratio morphism $\moduli[0,n]\to \Gm$.
Choosing an appropriate collection of cross-ratios, we obtain a morphism $\moduli[0,n]\to\Gm^N$, and
the torsion condition may be interpreted as saying that our stable form maps to a torsion point of
this torus.  For example in the (most difficult) case of an irreducible stable form, we obtain an
embedding $\moduli[0,10] \to \Gm^9$.

The problem is then to show that a subvariety of a torus meets only finitely many torsion points.  
We appeal to Laurent's theorem \cite{laurent}
which says that a subvariety of $\Gm^n$ contains only finitely many torsion points unless it
contains a torsion-translate of a subtorus, so we are lead to study translates of tori in the
seven-dimensional variety $\moduli[0,10]\subset\Gm^9$.

We then use the torus-containment algorithm from \S\ref{sec:torusalgo} to study tori in this
variety.  As the algorithm is much too slow to try to classify all torus-translates in a
nine-dimensional torus,  in \S\ref{sec:proof4zeros}, we give a significant reduction that says that
only subtori of one of three three-dimensional tori need to be considered.  Applying the algorithm,
we see that there are no subtori of $\moduli[0,10]$ which could lead to \Teichmuller curves with
arbitrarily large torsion orders.  Thus we obtain:

\begin{theorem}
  \label{thm:bounded_torsion}
  There is a uniform bound for the torsion orders of \APTCs in $\omoduli[3](1,1,1,1)$.
\end{theorem}

The final step using the torus-containment algorithm is heavily computer-aided and was only
completed in genus three.  Finding uniform torsion order bounds is all that remains to
establish finiteness of \APTCs for the generic stratum in arbitrary genus.

Combining these two theorems yields Theorem~\ref{thm:intromainfin} in the case of the generic
stratum in genus three.

For the remaining strata with two or three zeros, we establish finiteness in
\S\ref{sec:other-strata-genus} using a mix of these techniques.

\paragraph{Classification of \Teichmuller curves: State of the art.}

\Teichmuller curves with trace field $\ratls$ all arise as branched covers of tori by
\cite{gutkinjudge} and are dense in every stratum.  More generally, any \Teichmuller curve gives
rise to many \Teichmuller curves in higher genera by passing to branched covers.  \Teichmuller
curves that do not arise from this branched covering construction are called \emph{primitive}.  The
classification of primitive \Teichmuller curves is an important guiding problem.

The first examples of primitive \Teichmuller curves were constructed by Veech \cite{veech89}.  The
constructions in \cite{bouwmoell} subsumes his examples as well as those of Ward and provides all currently
known examples of \APTCs for $g \geq 6$.  Calta \cite{calta} and McMullen \cite{mcmullenbild}
constructed an infinite sequence of primitive \Teichmuller curves in $\moduli[2]$, the first
infinite sequence in any fixed genus.  McMullen completed the classification of primitive
\Teichmuller curves in genus two \cite{mcmullentor} and showed via a Prym construction
\cite{mcmullenprym} that in genus three and four there are an infinite number of \Teichmuller curves
which are primitive but not algebraically primitive.
\par
The torsion and real multiplication conditions are strong restrictions on the existence of \APTCs.
After \cite{mcmullentor}, the torsion condition was used in \cite{Mo08} to show finiteness of \APTCs
in the hyperelliptic strata where $\omega$ has two zeros or order $g-1$.   The real multiplication
and torsion conditions  were applied in \cite{BaMo12} to
prove finiteness in the case of the stratum $\omoduli[3](3,1)$  and to give an
effective algorithm showing the non-existence of \Teichmuller curves for given discriminants.  This
used a classification of the stable forms in the boundary of the eigenform locus, and introduced the
cross-ratio equation~\eqref{eq:constraint1} which is essential in the present proof of finiteness in
the minimal strata.
\par
Independent of this work, Matheus and Wright showed in \cite{MatWri13} that for every fixed genus
$g$ which is an odd prime, there are only finitely many \APTC generated by Veech surfaces with a single
zero.  Moreover, Nguyen and
Wright showed \cite{ManhWri13} that there are only finitely many primitive \Teichmuller curves in
genus $g=3$ generated by Veech surfaces with a single zero in the hyperelliptic stratum.  Their methods are very different from ours, relying on the \Teichmuller geodesic flow and recent work
of Eskin-Mirzakhani \cite{esmi} and Eskin-Mirzakhani-Mohammadi \cite{esmimo} establishing a an
analogue of Ratner's theorem for the $\SLtwoR$ action on strata of one-forms.  
\par
\paragraph{Acknowledgments}
The second named author was partially supported by the National Science Foundation under
agreement No.~DMS-1128155. Any opinions, findings and conclusions or
recommendations expressed in this material are those of the authors
and do not necessarily reflect the views of the National Science
Foundation. 
He thanks the Institute for Advanced Study in Princeton  for providing a stimulating and
encouraging
environment. 
The third named author is supported by ERC-StG~257137.
The authors warmly thank the MPIM Bonn, where this project began and a large part of this paper
has been written, as well as the MFO Oberwolfach, where this paper was finalized.


%% file: htboundthm.tex
\section{A theorem on height bounds} \label{sec:htbound}

In this section, we discuss some necessary background in diophantine geometry and establish the
``unlikely intersections'' result, Theorem~\ref{thm:htlogbound}.

\subsection{Zilber's Conjecture on Intersections with Tori}
\label{sec:zilber}

Zilber's Conjecture  on Intersections with Tori \cite{Zilber} governs
the locus where a subvariety of $\IG_m^n$  meets
algebraic subgroups of sufficiently low dimension. 
Let us state a variant of the conjecture found in \cite{Zilber}.

\begin{conj}
\label{conj:CIT}
Let $\cY$ be an irreducible  subvariety of $\IG_m^n$ defined over
$\IC$.
Let us suppose that the union
\begin{equation}
\label{eq:CITunion}
  \bigcup_{\substack{H\subset \IG_m^n \\ \dim \cY + \dim H \le n-1}} \cY\cap H \quad\text{is Zariski dense
    in}\quad
\cY
\end{equation}
where $H$ runs over algebraic subgroups with the prescribed
restriction on the dimension. 
Then $\cY$ is contained in a proper algebraic subgroup of $\IG_m^n$. 
\end{conj}
\par
 Zilber's Conjecture is stated more generally for
semi-abelian varieties and Pink \cite{Pink} has a version for mixed
Shimura varieties. 
\par
The algebraic subgroups of $\IG_m^n$ can be characterized easily, they
are in natural bijection with subgroups of $\IZ^n$, cf. Chapter 3.2 \cite{BG}. 
\par
The heuristics behind this conjecture are 
supported by the following basic observation.
Two 
subvarieties  of $\IG_m^n$ in general
position whose dimensions add up to something less than 
the dimension of the ambient group variety 
are unlikely to intersect; 
however, non-empty intersections are certainly possible.
Unless we are in the trivial case
$\cY = \IG_m^n$, the union (\ref{eq:CITunion}) is over a countable
infinite set of algebraic subgroups. The content of the conjecture is
just that any non-empty intersections that arise are contained in a  sufficiently
sparse subset of $\cY$ unless $\cY$ is itself inside a proper algebraic
subgroup of $\IG_m^n$.
\par
 Although the conjecture above  is open,
many partial results are known. We will now briefly mention several ones. 
\par
If $\cY$ is a hypersurface, i.e.\ $\dim \cY = n-1$, then the algebraic subgroups
in question are finite. So the union (\ref{eq:CITunion}) is precisely
the set of points on $\cY$ whose coordinates are roots of unity.
Describing the distribution of points of finite order on subvarieties
of $\cY$ is a special case of the classical Manin-Mumford Conjecture.
In general, the Manin-Mumford Conjecture states that a subvariety of a
semi-abelian variety can only contain a Zariski dense set of torsion
points if it is an irreducible component of an algebraic subgroup. 
The first proof in this generality is due to Hindry. In the important
case of  abelian
varieties the Manin-Mumford Conjecture was proved earlier by  Raynaud. 
Laurent's Theorem \cite{laurent} contains the Manin-Mumford Conjecture
for $\IG_m^n$.

Conjecture \ref{conj:CIT}  is known also if $\dim \cY = n-2$ due to work of
Bombieri, Masser, and Zannier \cite{BMZGeometric}. 

In low dimension,
 Maurin \cite{Maurin} proved the conjecture for curves defined over
 $\IQbar$. 
 Bombieri, Masser, and Zannier \cite{BMZUnlikely} later generalized this to
 curves defined over $\cx$.  

A promising line of attack of Conjecture \ref{conj:CIT} is via the theory of heights, 
which we will review in the next section. 
It is this approach that  motivates our Theorem \ref{thm:htlogbound}.
In some circumstances it is
 possible to prove instances of the conjecture by first
studying the larger union over algebraic subgroups that satisfy the weaker
 dimension inequality
$\dim \cY + \dim H \le n$.
It is no longer appropriate to call non-empty intersections $\cY\cap H$
unlikely and
one cannot expect 
 \begin{equation*}
   \bigcup_{\substack{H\subset \IG_m^n \\ \dim \cY + \dim H \le n}}\cY
   \cap H
 \end{equation*}
to be  non-Zariski dense in $\cY$.  We say that such a non-empty intersection $\cY\cap H$ 
is \emph{just likely}. 
It is sometimes possible to show that
the absolute logarithmic Weil height is bounded from above
on this union.
In fact, we will use such a bound which we state more
precisely below in Theorem \ref{le:BZbound}. 

To ease notation we abbreviate
\begin{equation*}
  \subgrpunion{(\IG_m^n)}{m} = 
  \bigcup_{\substack{H\subset \IG_m^n \\ \codim H \ge m}}  H.
\end{equation*}

Two caveats are in
order. First, in order to use the height in a meaningful way 
 we need to work with   subvarieties $\cY$ defined
over $\IQbar$, the field of algebraic numbers. Therefore, results on
height bounds usually contain an additional hypothesis on the
field of definition. Second, it is  in general false that the 
elements of 
\begin{equation*}
\subgrpunion{(\IG_m^n)}{\dim \cY} \cap \cY
\end{equation*}
have uniformly bounded height.  Indeed, it is possible 
that  $\cY$  has positive dimensional intersection with an algebraic
subgroup of dimension $\dim \cY$. As the height is not
bounded on a positive dimensional subvariety  of $\IG_m^n$ we must
avoid such intersections. 



We will see in moment  that there are more delicate obstructions to
boundedness of height. One must remove more as was pointed out by 
Bombieri, Masser, and Zannier \cite{BMZ99}. 
They proved the
following height-theoretic result for an
irreducible algebraic curve $\cC$ defined over $\IQbar$
and contained in $\IG_m^n$. 
A coset of $\IG_m^n$ will mean the translate of an algebraic subgroup
of $\IG_m^n$. 
If $\cC$ is not contained in a 
proper coset, then a point in $\cC$ that is contained in a proper
algebraic subgroup has height bounded in terms of $\cC$ only. 
They also proved a converse in the second remark after their Theorem
1. If $\cC$ is contained in a proper coset, 
then $ \cC\cap \subgrpunion{(\IG_m^n)}{1}$ does not have bounded
height. 
Observe that $\cC \cap \subgrpunion{(\IG_m^n)}{1}$ is always infinite. 


The second named author later in \cite{BHC}
proved a qualitative refinement of these height bounds
for the intersection of 
 a general subvariety $\cY\subset\IG_m^n$
with algebraic subgroups of complementary dimension.
An irreducible closed subvariety $\cZ\subset \cY$ is called
\label{def:anomalous}{\em anomalous} if there exists a coset $\cK\subset\IG_m^n$ with $\cZ\subset
\cK$
and
\begin{equation*}
  \dim \cZ \ge \max\{1,\dim \cY + \dim \cK -n+ 1\}.
\end{equation*}
Bombieri, Masser, and Zannier \cite{BMZGeometric} showed that the
(possibly infinite)
union of all anomalous subvarieties is Zariski closed in $\cY$. 
We write $\oa{\cY}$ for its complement in $\cY$.
The aforementioned result states that $\oa{\cY}$ is Zariski open in $\cY$. 

In the case of  a curve $\cC$, we have
$\oa{\cC}=\cC$ if and only if $\cC$ is not
contained in a proper coset. Otherwise we have $\oa{\cC}=\emptyset$.
Thus Bombieri, Masser, and Zannier's original height bound for curves \cite{BMZ99}
states that 
$\oa{\cC}\cap\subgrpunion{(\IG_m^n)}{1}$ has bounded height. 

The following bound is the  main theorem of \cite{BHC}.
\begin{theorem} \label{le:BZbound}
Let $\cY\subset\IG_m^n$ be an irreducible closed subvariety defined
over $\IQbar$. 
There exists $B\in\IR$ depending only on $\cY$  such that any
 point in  $\oa{\cY}\cap 
\subgrpunion{(\IG_m^n)}{\dim \cY}$
has absolute logarithmic Weil height bounded by $B$. 
\end{theorem}
\par

After introducing more notation we will cite a quantitative version of
Theorem
\ref{le:BZbound} for curves in Theorem \ref{thm:heightbound}.

The height bound can be used to recover some cases of Zilber's
Conjecture. 
The second named author later made this result completely explicit 
 \cite{habegger:effBHC}. The height bound $B$ is thus effective.

Before proceeding to the main result of this section we make, as promised, a brief detour to define
the height function mentioned above and several others.

\subsection{On Heights}
\label{sec:onheights}

We refer to the Chapter 1.5 \cite{BG}
or Parts B.1 and B.2  \cite{HiSi00}
for proofs of many basic properties of the absolute logarithmic Weil 
height that we discuss in this section.
 
Every non-trivial absolute value  $|\mathord{\cdot}|_v$  on a number field $K$
is equivalent to one of the following type.
If $|\cdot|_v$ is Archimedean, then there exists a field embedding
$\sigma\colon K\rightarrow\IC$, uniquely defined up to complex conjugation, such
that
$|x|_v = |\sigma(x)|$ for all $x\in K$, where $|\cdot|$ is the standard
complex absolute value. In this case we call $v$ infinite and write
$v|\infty$.
Depending on whether $\sigma(K)\subset \IR$ or not we define
the local degree of $v$ as $d_v = 1$ or $d_v=2$.
If $|\cdot|_v$ is non-Archimedean, then its restriction to $\IQ$ is
the $p$-adic absolute value for some rational prime $p$. 
For fixed $p$, the set of extensions of the $p$-adic valuation to $K$ is in bijection with the set 
of prime ideals in the ring of algebraic integers in $K$ that contain
the prime ideal $p\IZ$. In this case we call $v$ finite and write 
$v\nmid\infty$ or $v\mid p$. The local degree here is 
$d_v=[K_v:\IQ_p]$ where $K_v$ is a completion of $K$ with respect to
$v$. 
We write $V_K$ for the set of all absolute values $|\cdot|_v$ on $K$ as 
described above.
This set is sometimes called the set of places of $K$. 

 We note that if $x\in K\ssm\{0\}$, then 
$|x|_v =1$ for all but finitely many $v\in V_K$. 
The choice of  local degrees $d_v$ facilitates the product formula
\begin{equation}
\label{eq:productformula}
  \prod_{v\in V_K}|x|_v^{d_v} = 1. 
\end{equation}
Now we are ready to defined the absolute logarithmic Weil height, or
height for short,  of 
a tuple $x=(x_1,\ldots,x_n)\in K^n$ as
\begin{equation}
\label{eq:defineheight}
\height{x} = \frac{1}{[K:\IQ]}\sum_{v\in V_K}
d_v \log \max\{1,|x_1|_v,\ldots,|x_n|_v\} \ge 0.
\end{equation}
The normalization constants $d_v/[K:\IQ]$ guarantee that
 $\height{x}$ does not change when  replacing
$K$ by another number field containing all $x_i$. So we obtain a
well-defined function $\heightS\colon \IQbar^n\rightarrow[0,\infty)$.   
\par
Northcott's Theorem, Theorem 1.6.8 \cite{BG}, states that a subset of $\IQbar^n$
whose elements have uniformly bounded height and degree over $\IQ$ is
finite. This basic result is an important tool for proving finiteness
results in diophantine geometry. We will apply it in the proof of 
Theorem \ref{thm:intromainfin}. 
\par
In the special case $n=1$  the following estimates  will
prove useful. 
If $x,y\in\IQbar$ then both inequalities
\begin{equation*}
  \height{xy}\le \height{x}+\height{y}\quad\text{and}\quad
\height{x+y}\le \height{x}+\height{y}+\log 2,
\end{equation*}
follow from corresponding local inequalities 
applied to the definition
(\ref{eq:defineheight}).
The height, taking no negative values, does not 
restrict to a group homomorphism $\IQbar\ssm\{0\}\rightarrow\IR$.  
However, the  definition and the product formula yield
 homogenity
\begin{equation*}
  \height{x^k} = |k|\height{x}
\end{equation*}
for any integer $k$ if $x\not=0$.

It is sometimes useful to work with the  height of  algebraic
points in projective space. If $x = [x_0:\cdots:x_n]\in\IP^n$
is such a point with  representatives $x_0,\ldots,x_n$ in $K$, 
we set
\begin{equation*}
  \height{x} =  \frac{1}{[K:\IQ]}\sum_{v\in V_K}
d_v \log \max\{|x_0|_v,\ldots,|x_n|_v\}.
\end{equation*}
The  product formula (\ref{eq:productformula}) guarantees
that $\height{x}$ does not depend on the choice of projective
coordinates of $x$. 

If $f$ is a non-zero polynomial in algebraic coefficients, we set
$\height{f}$ to be the height of the point in projective space whose
coordinates are the non-zero coefficients of $f$.

We remark that different sources in the literature may employ
different norms at the Archimedean places of $K$. For example, instead
of taking the $\ell^\infty$-norm one can take the $\ell^2$-norm at the infinite
places. This leads to another height function $h_2(\cdot)$ on the algebraic points of $\IP^n$ which differs from
$\height{\cdot}$ by a bounded function.

We will make use of a result  of Silverman to control the behavior
of the height function under rational maps between varieties. 

\begin{theorem}
\label{thm:silverman}
  Let $\cX\subset\IA^m$ and $\cY\subset\IA^n$ be irreducible
  quasi-affine
 varieties defined over $\IQbar$
  with $\dim \cX = \dim \cY$. Suppose that $\varphi\colon\cX\rightarrow\cY$
  is a dominant morphism. There exist constants $c_1>0$ and $c_2$
  that depend only on $\cX,\cY$ and a Zariski open and dense subset
  $U\subset \cX$ such that
  \begin{equation*}
    \height{\varphi(P)} \ge c_1\height{P}- c_2
\quad\text{for all}\quad P\in U. 
  \end{equation*}
Moreover, this estimate holds true with 
$$U=U_0 = \{P\in \cX : \,\, \text{$P$ is isolated in
}\varphi^{-1}(\varphi(P))\},$$ 
which is Zariski open in $\cX$. 
\end{theorem}
\begin{proof}
The first statement follows from Silverman's Theorem 1 \cite{Silverman}. 

The openness of $U_0$ from the last statement  follows
from Exercise II.3.22(d) \cite{hartshorne}. 
 By restricting
to the irreducible components in the complement of  the open set provided by
Silverman's Theorem  we
may use 
 Noetherian induction to prove
the height inequality on $U_0$ with possibly worse constants.
\end{proof}
\par
A reverse inequality, i.e. 
  \begin{equation} \label{eq:htupperbd}
 \height{\varphi(P)} \leq c_1^{-1}\height{P} + c_2
  \end{equation}
for any $P\in\cX$ 
holds with possibly different constants. It requires neither $\varphi$ being dominant or $\dim \cX = \dim \cY$, and
is more elementary, see e.g.\ \cite[Theorem B.2.5]{HiSi00}. 
\par
It is also possible to assign a height to an irreducible closed subvariety $\cY$ of
$\IP^n$  defined over $\IQbar$. 
The basic idea is to consider the Chow form of $\cY$, which is
well-defined up-to scalar multiplication, as a point in some
 projective space. The height of this point 
with then be the height $\height{\cY}$ of $\cY$.
In this setting it is common to use a norm at the  Archimedean place
which is related to the Mahler measure of a polynomial.  The details
of this definition are presented in Philippon's paper
\cite{Philippon95}. 
\par
With this normalization, the height of a singleton
$\{P\}$ with $P$ an algebraic point of $\IP^n$ is the
height of $P$ with the $\ell^2$-norm at the Archimedean places. 
Beware that the height of a projective variety is by no means an
invariant of its
isomorphism class. It depends heavily on the embedding
$\cY\subset\IP^n$. 
\par
Zhang's inequalities \cite{ZhangArVar} relate the height of $\cY\subset
\IP^n$, its degree, and the  points of small height on $\cY$. In order
to state them, we require the essential minimum
\begin{equation*}
  \mu^{\rm ess}(\cY) = \inf\Bigl\{x\ge 0 : 
\{P\in \cY :  h_2(P)\le x\} \text{ is Zariski dense in $\cY$} \Bigr\}
\end{equation*}
of $\cY$. The set in the infimum is non-empty and so
$\mu^{\rm ess}(\cY)<+\infty$. 
In connection with the Bogomolov Conjecture
Zhang proved 
\begin{equation}
  \label{eq:zhangineq}
\mu^{\rm ess}(\cY) \le \frac{h(\cY)}{\deg \cY} \le (1+\dim \cY)\mu^{\rm ess}(\cY).
\end{equation}
The second inequality can be used to bound $h(\cY)$ from above if one can
exhibit a Zariski dense set of points on $\cY$ whose height is bounded
from above by a fixed value. 

The morphism
$(x_1,\ldots,x_n)\mapsto[1:x_1:\cdots:x_n]$ allows us to consider
 $\IG_m^n$ and $\IA^n$ as open subvarieties of $\IP^n$.
The height of an irreducible closed subvariety of $\IG_m^n$
or $\IA^n$ 
defined over $\IQbar$ 
is the height of its Zariski closure in $\IP^n$. 

We recall that
$\deg{\cY}$ is the cardinality of  the intersection  of $\cY$ with a linear
subvariety
of $\IP^n$ in general position  with codimension $\dim\cY$. 
By taking the Zariski closure in $\IP^n$ as in the previous paragraph
we may speak of the degree of any irreducible closed subvariety of
$\IG_m^n$ or $\IA^n$. 
\par
If $\cX$ is a second irreducible closed subvariety of $\IP^n$, then
B\'ezout's Theorem states 
\begin{equation*}
  \sum_{\cZ}\deg \cZ \le (\deg\cX)(\deg\cY)
\end{equation*}
where $\cZ$  runs over all irreducible components $\cZ$ of
$\cX\cap\cY$. For a proof we refer to Example 8.4.6 \cite{Fulton}.
\par 
We come to the  arithmetic counterpart of this classical result. 
According to Arakelov theory, $\height{\cY}$ is the arithmetic
counterpart of the geometric degree $\deg{\cY}$. 
\par
\begin{theorem}[Arithmetic B\'ezout Theorem]
\label{thm:ABT}
There exists a positive and effective constant $c>0$ that depends only
on $n$ and satisfies the following property.
  Let $\cX$ and $\cY$ be irreducible closed subvarieties
of $\IP^n$, both defined over $\IQbar$, then
\begin{equation*}
  \sum_{\cZ}\height{\cZ} \le \deg(\cX)\height{\cY} +
  \deg(\cY)\height{\cX} + c \deg(\cX)\deg(\cY)
\end{equation*}
where $\cZ$ runs over all irreducible components of $\cX\cap \cY$. 
\end{theorem}
\begin{proof}
  For a proof we refer to Philippon's Theorem 3 \cite{Philippon95}. 
\end{proof}
\par
 Not surprisingly, the height of a hypersurface is
closely related to the height of a defining equation. 
For our purposes it suffices to have the following estimate. 
\begin{prop}
\label{prop:heighthyper}
There exists a positive and effective constant $c>0$ that depends only
on $n$ and satisfies the following property.
  Let $f\in \IQbar[X_0,\ldots,X_n]$ be a homogeneous, irreducible polynomial and
  suppose that $\cY$ is its zero set in $\IP^n$. Then 
$\height{\cY}\le \height{f} + c \deg{f}$. 
\end{prop}
\begin{proof}
  See page 347 of Philippon's paper \cite{Philippon95} for a more precise statement. 
\end{proof}
\par
We will freely apply Zhang's inequalities and the Arithmetic B\'ezout
Theorem
to subvarieties $\IG_m^n$ and
$\IA^n$, always keeping in mind the open immersions
 $\IG_m^n\rightarrow\IP^n$ and
 $\IA^n\rightarrow\IP^n$.

\subsection{A Weak Height Bound}

In this section, we will formulate and prove a height
bound  which  is reminiscent of the result on just likely
intersections in Theorem \ref{le:BZbound}.
But instead of working in the ambient group $\IG_m^n$, we work instead in
$\IG_m^n\times\IG_a^n$. We will also restrict to  surfaces.
The results of this section  will be applied in the
proof of Theorem \ref{thm:intromainfin}.
Our new height bound will only take a certain class of algebraic
subgroups into account. 
It  will also  no longer be uniform, as it will depend logarithmically on the
degree over $\IQ$ of the point in question. 
However, the points in  our application  are known to have bounded
degree over the rationals. Therefore, their height and degree are
bounded from above. Northcott's Theorem will imply that the number of
points
under consideration is finite. 
\par
Let us consider an irreducible, quasi-affine
surface  $\cY \subset \IA^k$ defined over $\IQbar$
with two collections of functions.
For $1\le i\le n$ let $R_i\colon  \IA^k \dashrightarrow \IG_m$ and let $\ell_i\colon  \IA^k \dashrightarrow \IG_a$
be rational maps, defined on Zariski open and dense subsets of
$\IA^k$. 
We also suppose that their  restrictions to $\cY$ (denoted by the same letter)
are regular.
The main theorem of this section is a height bound for
points on $\cY$  that satisfy both
a multiplicative relation among the $R_i$ and a linear relation among
the $\ell_i$, with the same coefficients. We write $R\colon  \cY \to \IG_m^n$
and $\ell\colon  \cY \to \IG_a^n$ for the product maps.

In the theorem below, we will suppose that $R\colon \cY\to\IG_m^n$ has finite
fibers. Then  $\cS = \overline{R(\cY)}\subset\IG_m^n$ is a surface.
Here and below $\overline{\,\cdot\,}$ 
refers to closure with respect to 
 the Zariski topology. 

\begin{theorem}
  \label{thm:heightmain}
Let us keep the assumptions introduced before. 
There is an effective constant $c>0$ depending only on $\cY$, the $\ell_i$, and the $R_i$
with the following
property. 
Suppose $y\in \cY$ is such that $R(y)\in \oa{\cS}$. 
If there is $(b_1,\ldots,b_n)\in\IZ^n\ssm\{0\}$ 
\begin{enumerate}
\item [(i)] with $b_1\ell_1+\cdots + b_n\ell_n\not=0$ in the function
  field of  $\cY$, 
\item[(ii)] such that $y$ is contained in an irreducible curve $\cC_b$
cut out on $\cY$ by $$b_1\ell_1+\cdots + b_n\ell_n=0$$ 
with $\oa{\overline{R(\cC_b)}}=\overline{R(\cC_b)}$,
\item[(iii)] and
\begin{equation}
\label{eq:relation}
  R_1(y)^{b_1}\cdots R_n(y)^{b_n} = 1,
\end{equation}
\end{enumerate}
 then
\begin{equation*}
  \height{y} \le c\log(2[\IQ(y):\IQ]). 
\end{equation*}
\end{theorem}
\par
Recall that the condition on $\overline{R(\cC_b)}$ in (ii)
stipulates that the said curve is not contained in a proper coset of $\IG_m^n$. 
\par
At the end of this section we will provide another formulation for 
 this theorem  which is more in line with known results towards Zilber's
 Conjecture. 
The formulation at hand was chosen with our application to 
\Teichmuller curves in mind.    
\par
The theorem is effective in the sense that one can explicitly express
$c$  in terms of $\cY$.
\par
The proof splits up into two cases.
\begin{enumerate}
\item  In the first case we forget about the additive relation in
  (\ref{eq:relation}) but assume that there is an additional
  multiplicative relation. This will lead to a  bound 
  for the height that is independent of $[\IQ(y):\IQ]$. 
\item Second, we assume that there is precisely one multiplicative
  relation up-to scalars. This time we need the additive equation in
  (\ref{eq:relation}) and we will obtain a height bound that depends
  on $[\IQ(y):\IQ]$. 
\end{enumerate}
\par

 We remark that Theorem \ref{thm:intromainfin} uses Theorem \ref{thm:heightmain}
applied to $n=3$. The latter relies on height bound in Theorem \ref{le:BZbound}.
In this low dimension, $\cS\ssm \oa{\cS}$ 
coincides with the union of all positive dimensional cosets contained
completely
in $\cS$.
It turns out that the result of Bombieri-Zannier,
cf. Appendix of \cite{Schinzel}, can be used instead of Theorem
\ref{le:BZbound}. 
One could also use Theorem 1 \cite{Ha08} in the case $n=3$,  $s=2$,
and $m=1$ to obtain a completely explicit height bound
while avoiding the crude bound of \cite{habegger:effBHC}. 
For general $n$ it seems that Theorem \ref{le:BZbound} is indispensable.

\begin{lemma} 
\label{lem:htcompare}
There exist effective constants $c_1,c_2$ depending only on $\cY$ 
with $c_1>0$
such that if $y\in \cY$ then 
\be  \label{eq:htfinitemap}
h(y) \leq c_1 h(R(y)) + c_2.
\ee
\end{lemma}
\par
\begin{proof}
This statement follows  from Theorem \ref{thm:silverman} as
$R$ has finite fibers on $\cY$. 
One checks that readily that Silverman's second proof is effective.
\end{proof}
\par
We use $|\cdot|$ to denote the $\ell^\infty$-norm on any power of $\IR$. 
Let us recall the following basic result called Dirichlet's Theorem on
Simultaneous Approximation. 

\begin{lemma}
\label{lem:dirichlet}
  Let $\theta\in\IR^n$ and suppose
$Q > 1$ is an integer. There exist 
$q\in\IZ$ and $p\in\IZ^n$ with $1\le q < Q^n$ and
$|q\theta - p| \le 1/Q$. 
\end{lemma}
\begin{proof}
  See Theorem 1A in Chapter II, \cite{Schmidt:LNM785}. 
\end{proof}
If $r=(r_1,\ldots,r_n)\in\IG_m^n$ is any point
and $b=(b_1,\ldots,b_n)$ then we abbreviate
$r_1^{b_1}\cdots r_n^{b_n}$ by $r^b$. 
\par
\begin{lemma}
\label{lem:findb}
There is an effective constant $c>0$ depending only on $n$
 with the following property. 
Let $d\ge 1$ and
suppose $r \in\subgrpunion{(\IG_m^n)}{1}$ is algebraic
with $[\IQ(r):\IQ]\le d$. 
There exists $b\in\IZ^n$ with 
$|b|\le c  d^{2n}\max\{1,\height{r}\}^n$ such that 
$r^b$ is a root of unity. 
\end{lemma}
\begin{proof}
Let $Q > 1$ be a sufficiently large integer to be fixed later on. 
Since $r$ is contained in a proper algebraic
subgroup of $\IG_m^n$ there is
 $b'\in\IZ^n\ssm\{0\}$
with 
$r^{b'}=1$.

By Dirichlet's Theorem, Lemma \ref{lem:dirichlet}, there exists
 $b\in\IZ^n$ and an integer
$q$ with $1\le q<  Q^n$ such that
$|q b'/|b'|-b|\le Q^{-1}$.
We remark that $b\not=0$  since $Q>1$. Moreover,
 $|b|\le q+Q^{-1}<Q^n+1$ by the triangle inequality. Hence $|b|\le
Q^n$ since $|b|$ and $Q^n$ are integers. 

With $\delta =|b'| b - qb' \in
\IZ^n$ we have
\begin{equation*}
r^{|b'|b}= r^{\delta + qb'} = r^{\delta}.
\end{equation*}
The height  estimates mentioned above yield
$$|b'|\height{z} \le
|\delta|(\height{r_1}+\cdots+\height{r_n})
\le n|\delta|\height{r},$$
where  $z=r^b$.
We divide by $|b'|$ and find $\height{z}\le n Q^{-1}\height{r}$. 

We note  that $z\in \IQ(r)\ssm\{0\}$ and recall $[\IQ(r):\IQ]\le d$. 
By  Dobrowolski's Theorem \cite{Dobrowolski},  which is effective, we have
either $\height{z}=0$ or  $\height{z}\ge c' d^{-2}$ for some absolute
constant $c'\in (0,1]$. Observe that we do not need the full strength
of Dobrowolski's bound.
The choice
 $Q=[2nd^2 \max\{1,\height{r}\}/c']$  forces $z$ to be  a root of
unity. The lemma follows with $c=(2n/c')^n$.
\end{proof}

We are now almost ready to prove our main result. It relies on the
following explicit height bound. 

\begin{theorem}[{\cite{Ha08}}]
\label{thm:heightbound}
  Suppose $\mathcal C\subset\IG_m^n$ is an irreducible algebraic curve defined
  over $\IQbar$ that is not contained in a coset of $\IG_m^n$. Any point in $\mathcal
  C\cap\subgrpunion{(\IG_m^n)}{1}$ has height at most
  \begin{equation*}
    c(\deg{\cC})^{n-1} (\deg{\cC} + \height{\cC})
  \end{equation*}
where $c>0$ is effective and depends only on $n$. 
\end{theorem}
\par
\begin{proof}[Proof of Theorem \ref{thm:heightmain}]
Suppose $y \in \cY$ is as in the hypothesis and 
$d=[\IQ(y):\IQ]$. In particular,  (\ref{eq:relation}) holds
for some $b\in \IZ^n\ssm\{0\}$. As discussed in the introduction, we
split up into two cases.
\par
In the first case, suppose the point $r=R(y)$ satisfies two independent
multiplicative relations. Then Theorem
\ref{le:BZbound} applies because $r\in \oa{\cS}$ by hypothesis. 
Since $R$ has finite fibers  Lemma \ref{lem:htcompare}   
implies that the height of $y$ is bounded from above solely in terms of $\cY$.
This is stronger than the conclusion of the theorem. 
\par 
In the second case, we will assume that the coordinates of
 $r$ satisfy precisely one
multiplicative relation up-to scalar multiple.
Here we shall
 make use of the additive relation in \eqref{eq:relation}. By
 assumption,
the group
\begin{equation*}
\{a\in\IZ^n : r^a=1\}
\end{equation*}
is free abelian  of rank $1$.
It certainly contains 
$b$ from the multiplicative relation in \eqref{eq:relation}. 
However, it also contains a positive multiple of a vector
 $b'\in\IZ^n\ssm\{0\}$  coming from Lemma~\ref{lem:findb}. 
Thus $b$ and $b'$ are linearly dependent
and hence the additive relation \eqref{eq:relation} holds with
$b$ replaced by 
 $b' = (b'_1,\ldots,b'_n)$. 
By hypothesis (i)
 our point $y$ lies on
an irreducible curve $\cC\subset \cY$ on which
\begin{equation}
\label{eq:bell}
 b'_1 \ell_1+\cdots + b'_n\ell_n
\end{equation}
vanishes identically with $\oa{\cC}=\cC$. 
\par
Recall that the  curve $\cC$ is an irreducible component of the zero set
of (\ref{eq:bell}) on $\cY$. Each $\ell_i$ can be
expressed by a quotient of polynomials mappings. From this point of
view,
$b'_1\ell_1+\cdots+b'_n\ell_n$ is a quotient of polynomials whose
degrees
are bounded by a quantity that is independent of
$(b'_1,\ldots,b'_n)$. So B\'ezout's Theorem implies that the degree of
the Zariski
closure of $\cC$ in $\IA^k$ is bounded from above in terms of $\cY$
only. 
We observe that $\overline{R(\cC)}$ is an irreducible curve. 
As $\deg\overline{R(\cC)}$ equals the generic number of
intersection points of $\overline {R(\cC)}$ with a hyperplane, we
conclude, again using B\'ezout's Theorem, that
\begin{equation}
  \label{eq:dRCbound}
\deg \overline{R(\cC)}\ll 1,
\end{equation}
 where here and below 
 $\ll$ signifies Vinogradov's notation with a constant that
 depends only on $\cY$, the $\ell_i$, and the $R_i$. These constants
 are effective. 
\par 
We also require a bound for the height of the curve
$\overline{R(\cC)}$. This we can deduce with the help of Zhang's inequalities
(\ref{eq:zhangineq}). Indeed, the numerator of (\ref{eq:bell}) is a
polynomial whose height is $\ll \log(2|b'|)$ by elementary height
inequalities.
Any irreducible component of its zero set has height $\ll \log(2|b'|)$ by
Proposition \ref{prop:heighthyper} and degree $\ll 1$. The Arithmetic B\'ezout Theorem
implies $h(\overline \cC)\ll \log(2|b'|)$. Using the degree
bound we deduced above and the first inequality in (\ref{eq:zhangineq}) we conclude
that $\overline \cC$ contains a Zariski dense set of points $P$ with
$h(P)\le h_2(P)\ll \log(2|b'|)$. The height bound (\ref{eq:htupperbd}) just below
Silverman's result yields $h(R(P))\ll \log(2|b'|)$. So the second bound in
(\ref{eq:zhangineq}) and $\deg \overline R(\cC)\ll 1$ give
\begin{equation}
\label{eq:hRCbound}
 h(\overline{R(\cC)})\ll \log|b'|.
\end{equation}
\par 
Now $r=R(y)\in \overline{R(\cC)}$ and $r\in \subgrpunion{(\IG_m^{n})}{1}$ by the
 original multiplicative relation \eqref{eq:relation}.
We insert (\ref{eq:dRCbound}) and (\ref{eq:hRCbound}) into
Theorem~\ref{thm:heightbound} and 
use the upper bound for $|b'|$ to find
\begin{equation*}
  \height{r}\ll \log(2d \max\{1,\height{r}\}).
\end{equation*}
\par
 Linear beats logarithmic, so $\height{r}\ll \log(2d)$. 
Finally,
we use Lemma \ref{lem:htcompare} again to deduce
$\height{y}\ll \log(2d)$. This completes the proof.
\end{proof}

\subsection{Intersecting with algebraic subgroups of $\IG_m^n\times\IG_a^n$}\label{sec:inters-with-algebr}


The unipotent group $\IG_a^n$ is not covered by Conjecture
\ref{conj:CIT} or Zilber's more general formulation for
 semi-abelian varieties. Indeed, a verbatim
translation of the statement of Conjecture \ref{conj:CIT}
 to $\IG_a^n$ fails  badly. Any point of $\IG_a^n$ is
 contained in a line passing through the origin, and is
thus in a $1$-dimensional
 algebraic subgroup.

Motivated by Theorems \ref{le:BZbound} and \ref{thm:heightmain}
we will deduce a height bound for points on a surface
inside $\IG_m^n\times\IG_a^n$ which are contained
in a \emph{restricted} class of algebraic subgroups of codimension $2$. Our aim
is
to formulate a result
that is  comparable to the more
 well-known case of the algebraic torus. 
The reader whose main interest lies in proof of Theorem \ref{thm:intromainfin}
may safely skip this section.  

Any algebraic subgroup of $\IG_m^n\times\IG_a^n$ splits into the
product
of an algebraic subgroup of $\IG_m^n$ and of $\IG_a^n$. 
We call the translate of an
algebraic subgroup of $\IG_m^n\times\IG_a^n$ by a point in
$\IG_m^n\times\{0\}$
a semi-torsion coset.
We call it  rational if it is
the translate of an algebraic subgroup of $\IG_m^n\times\IG_a^n$ 
defined over $\IQ$ by any point of $\IG_m^n\times\{0\}$. 
A rational semi-torsion coset need not be defined over $\IQ$,
but its associated algebraic subgroup of $\IG_a^n$ is defined by
linear equations with rational coefficients. 

Let $\cY$ be an irreducible subvariety of $\IG_m^n\times\IG_a^n$ defined
over $\IC$. 
We
single out an exceptional class of 
 subvarieties of $\cY$ 
related to  Bombieri, Masser, and Zannier's anomalous subvarieties
\cite{BMZGeometric}.

We say that an irreducible closed subvariety $\cZ$ of $\cY$ is  \emph{rational
semi-torsion anomalous} if 
it is contained in a rational semi-torsion coset
$\cK\subset\IG_m^n\times\IG_a^n$ with 
\begin{equation}
\label{eq:anomalousineq}
  \dim \cZ \ge \max\{1,\dim \cY + \dim \cK - 2n + 1\}. 
\end{equation}

We let $\Qta{\cY}$ denote the complement in $\cY$ of the union of all
rational semi-torsion anomalous subvarieties of $\cY$. 
\label{def:rstas}

 Bombieri, Masser, and Zannier's $\oa{\cY}$ for
$\cY\subset\IG_m^n$ is always Zariski open. In the example below we
 show that this
 is not necessarily the case for $\Qta{\cY}$ if
 $\cY\subset\IG_m^n\times\IG_m^a$ is  a surface. 
 
 \begin{example}
\label{ex:nonopen}
   Let us consider the case $n=2$ and let $\cY$ be the irreducible
   surface given by
     \begin{equation}
\label{eq:defcV}
\begin{aligned}
     x_1y_1 + (x_1+1)y_2 &=0,\\
     x_1y_1 + x_2 y_2   &=1
\end{aligned}
\end{equation}
where $(x_1,x_2,y_1,y_2)\in \IG_m^2\times\IG_a^2$.
Observe that the projection of $\cY$ to $\IG_m^2$ is dominant. 
 We will try to understand some features of $\Qta{\cY}$. 

Let $\cK$ be a rational semi-torsion anomalous subvariety of
$\cY$. Then a certain number of additive and multiplicative relations
hold on the coordinates of $\cK$ and the dimension of $\cK$ cannot be
below the threshold determined by (\ref{eq:anomalousineq}). 

Suppose first that a relation $b_1y_1+b_2y_2=0$ holds on $\cK$
where $(b_1,b_2)\in\IZ^2\ssm\{0\}$. The second equality in
(\ref{eq:defcV}) yields $(y_1,y_2)\not=0$, and the first one yields
$0=x_1b_2-(x_1+1)b_1$,  so the projection of $\cK$ to $\IG_m^2$ maps to
one of countably many algebraic curves. In particular, $\cK$ is a
curve and there are at most countably many possibilities for $\cK$. 

Second, let us assume that no linear relation as above holds on $\cK$.
Then a certain number of multiplicative relations
$x_1^{b_1}x_2^{b_2}=\lambda$ hold on   $\cK$.
We cannot have $\cK=\cY$, as $\cY$ has dense image in $\IG_m^2$. So
 there must be two 
multiplicative relations  with independent exponent vectors
 for  $\cK$ to be anomalous.
In particular, $x_1$ and $x_2$ are constant on $\cK$. But for a fixed
choice of $(x_1,x_2)$ the two linear equations (\ref{eq:defcV})  
are linear in $(y_1,y_2)$ and 
have at most one solution in these unknowns. This contradicts $\dim \cK\ge
1$. 

Now we know  that any rational semi-torsion anomalous
subvariety of  $\cY$  is a curve and that their cardinality is  at most countable. 

Finally, let us exhibit such curves. For a given $\xi\in\IQ\ssm\{0\}$ the
equation $x_1=\xi$ cuts out an irreducible curve in
$\cY$.
This equation and the one obtained by substituting $\xi$ for $x_1$ in
first line of (\ref{eq:defcV})
establishes that this curve is rational semi-torsion
anomalous. 

Thus $\cY\ssm \Qta{\cY}$ is a countable, infinite union of
curves. In particular, $\Qta{\cY}$ is not Zariski open in $\cY$; it is
also not open 
with respect to the Euclidean topology. 
 \end{example}

Above we introduced the notation $x^b$ for a point
$x\in\IG_m^n$ and $b\in\IZ^n$. If $y=(y_1,\ldots,y_n)\in \IG_a^n$ and 
$b=(b_1,\ldots,b_n)$ we set
\begin{equation*}
  \langle y,b\rangle  = y_1b_1+\cdots + y_nb_n. 
\end{equation*}

An algebraic subgroup  $G\subset \IG_m^n\times\IG_a^n$ is called
\emph{coupled} if there exists a subgroup $\Lambda\subset\IZ^n$  with
\begin{equation*}
G=\left\{(x,y)\in\IG_m^n\times\IG_a^n : x^b = 1\quad\text{and}
\quad \langle y,b\rangle = 0 
\quad\text{for all}\quad b\in \Lambda \right\}. 
\end{equation*}
The dimension of $G$ is $2(n-\rank \Lambda)$. 

We define $\sgu{(\IG_m^n\times \IG_a^n)}{s}$ to be the union of all
coupled algebraic subgroups of $\IG_m^n\times\IG_a^n$ whose 
codimension is at least $s$. 

Using this notation we have the following variant of Theorem \ref{thm:heightmain}. 

\begin{theorem}
  Let $\cY \subset \IG_m^n\times\IG_a^n$ be an irreducible,  closed
algebraic  surface defined over $\IQbar$. 
There exists a constant $c>0$ with the following property.
If $P\in \Qta{\cY} \cap \sgu{(\IG_m^n\times\IG_a^n)}{2}$, then 
\begin{equation*}
 \height{P}\le c\log(2[\IQ(P):\IQ]).
\end{equation*}
\end{theorem}
\begin{proof}
The current theorem resembles  Theorem \ref{thm:heightmain} but it is
not a direct consequence.
However, we will invoke Theorem \ref{thm:heightmain} below. 
  Indeed, we take $\cY$ as a quasi-affine subvariety of
  $\IG_m^n\times\IG_a^n\subset \IA^{2n}$.
  The rational maps $R$ and  $\ell$  are the two projections
$\IG_m^n\times\IG_a^n\rightarrow \IG_m^n$ and
$\IG_m^n\times\IG_a^n\rightarrow \IG_a^n$, respectively.

Say $P=(x,y)\in \Qta{\cY} \cap\sgu{(\IG_m^n\times\IG_a^n)}{2}$.
 We begin a study of various cases. 

Suppose first that $\overline{R(\cY)}\subset\IG_m^n$ is a point. The existence of $P$ implies
that
  $\cY$ meets a proper coupled subgroup.
So  $R(\cY)$ is in a proper algebraic
subgroup of $\IG_m^n$
which means that $\cY$ is in the product of this subgroup with $\IG_a^n$.
 In this case $\Qta{\cY}$ is empty, a contradiction.

So we have $\dim \overline{R(\cY)} \ge 1$.
Let us assume that $P$ is not isolated in its fiber of
$R|_{\cY}$. Here we can argue much as in the proof of
Theorem \ref{thm:heightmain}.
The point $P$ is contained in some irreducible component $\cD_x$ of %
$R|_{\cY}^{-1}(x)$ with $\dim \cD_x \ge 1$. But $\dim \cD_x = 1$ since %
$R|_{\cY}$ is non-constant.  
Observe that $\cD_x$ is an irreducible component in the %
intersection of $\cY$ and the
rational semi-torsion coset $\{x\}\times\IG_a^n$. %
We now split-up into 2 subcases. 

Suppose first that  $\overline{R(\cY)}$ has dimension $1$.
Then $\overline{R(\cY)}$ cannot be contained in a proper coset of $\IG_m^n$ as 
 $\Qta{\cY}\not=\emptyset$. 
So
$\overline{R(\cY)}
\cap\subgrpunion{(\IG_m^n)}{1}$ has bounded height by Theorem 1
of Bombieri, Masser, and Zannier \cite{BMZ99}.  
As $x$ lies in this intersection we have  %
$\height{x}\ll 1$; here and below the constant implied in Vinogradov's %
notation  depends only on $\cY$.

In the second subcase we suppose that $\overline{R(\cY)}$ has dimension $2$. 
The set of points in $\cY$ that are contained in a positive
dimensional fiber of $R|_{\cY}$ 
 is a Zariski closed proper subset of $\cY$, cf. 
Exercise II.3.22(d) \cite{hartshorne} already used above. 
Hence $\cD_x$ is a member of a finite set of curves depending only on %
$\cY$. 
So $x$, being the image of $\cD_x$ under $R$, is %
member of a finite set depending only on $\cY$. 
In particular, $\height{x}\ll 1$ holds trivially. %

In both subcases we have $\height{x}\ll 1$. 
The Arithmetic B\'ezout Theorem yields the height bound
\begin{equation}
  \label{eq:degheightbound}
 h(\cD_x)\ll 1\quad\text{and}\quad \deg{\cD_x}\ll 1 %
\end{equation}
the degree bound follows from the classical B\'ezout Theorem. 

Let us abbreviate $d = [\IQ(P):\IQ]$. 
The coordinates of $x$ are multiplicatively %
dependent.
But there cannot be $2$ independent relations as  $\cD_x$ %
would otherwise be contained in rational semi-torsion anomalous subvariety of
$\IG_m^n\times\IG_a^n$.
Lemma~\ref{lem:findb} 
and $h(x)\ll 1$  implies $\langle y,b\rangle = 0$ %
 for some 
$b\in\IZ^n$ with $|b|\ll d^{2n}$.

The vanishing locus of the linear form
\begin{equation*}
y\mapsto \langle y, b\rangle %
\end{equation*}
determines a linear subvariety of 
$\IG_m^n\times\IG_a^n$ with height $\ll \log (2|b|)$. 
The point $P=(x,y)$ is isolated in its intersection with $\cD_x$ %
as $P\in \Qta{\cY}$. 
The Arithmetic B\'ezout Theorem and (\ref{eq:degheightbound}) yield
$\height{P}\ll \log(2|b|)$. 
We combine this bound with the upper bound for $|b|$ to establish the
theorem 
if $P$ is not isolated in the corresponding fiber of $R|_{\cY}$.  

From now on  we assume that $P$ is isolated in $R|_{\cY}^{-1}(x)$. %
The set of all such points of $\cY$ is a Zariski open subset $\cY'$ of
$\cY$. The restriction $R|_{\cY'}$ has finite fibers and hence the hypothesis
leading up to Theorem~\ref{thm:heightmain} is fulfilled for
$\cY'$ where the $\ell_i$ run over the $n$ the projection morphisms to $\IG_a$. 
We write $\cS$ for the Zariski closure of $R(\cY')$; this is an
irreducible surface. 

Say $b\in\IZ^n\ssm\{0\}$ with 
 $x^b=1$ and $\langle y,b\rangle = 0$. 
The  conditions (i), (ii), and (iii) 
in Theorem~\ref{thm:heightmain}
are met; for the first two we need $P\in \Qta{\cY}$.  
If $x \in \oa{\cS}$ holds, then  %
the height bound from the said theorem
 completes the proof. 

So it remains to treat the case $x\not\in \oa{\cS}$.  %
By definition there is a coset $\cK\subset\IG_m^n$ 
and an irreducible component $\cZ$ of $\cS\cap \cK$ containing $x$ %
with 
\begin{equation}
  \label{eq:dimcZlb}
\dim \cZ \ge \max\{1,3+\dim \cK - n\}.
\end{equation}
This inequality implies $\dim \cK \le n-1$ because $\dim \cZ\le 2$. 

Observe that $\dim \cZ = 1$. Indeed,  otherwise $\cZ=\cS$ 
would be contained in $\cK$.  Then
$\cY$ would be contained in the rational semi-torsion coset
$\cK\times\IG_a^{n}$ which would contradict $P\in \Qta{\cY}$.

Since $\cZ$ is a curve we find 
\begin{equation}
\label{eq:dimcKub}
  \dim \cK \le n-2 
\end{equation}
from (\ref{eq:dimcZlb}).

Of course $P=(x,y)\in R|_{\cY}^{-1}(\cZ)$. 
Let $\cZ'$ be
 an irreducible component  $R|_{\cY}^{-1}(\cZ)$ containing
$P$ with largest 
 dimension. Now  $\cZ'$ is in the rational semi-torsion
coset $\cK\times\IG_a^n$ already used above. 
If 
 $\cZ'$
has positive dimension, then 
$\dim \cZ' \ge 2 + \dim  \cK\times\IG_a^n-2n+1$
because of (\ref{eq:dimcKub}). 
But then $\cZ'$ is a rational semi-torsion anomalous subvariety of
$\cY$. This is again a
  contradiction to  $P\in\Qta{\cY}$.

We conclude that $\cZ' = \{P\}$.
This is an awkward
situation as one would expect that the pre-image of a curve under 
the dominant morphism $R|_{\cY}\colon \cY\rightarrow \cS$ between surfaces 
to be again a curve.
 So we can hope to extract 
useful information. 
We are in characteristic $0$, so by Lemma III.10.5 \cite{hartshorne} there is a 
Zariski open and non-empty set $U\subset \cY$ such that
$R|_U\colon U\rightarrow \cS$ is a smooth morphism.
This restriction is in particular open.
It has the property that the preimage of any irreducible curve
 in $R(U)$ is a finite union of irreducible curves. 
We claim that $P$ does not lie in $U$. Indeed, otherwise 
$P$ would be an isolated point of a fiber of $R|_U$. 
This contradicts smoothness of $R|_U$
 as $R(U)\cap\cZ$ is an irreducible curve containing $R(P)$. 

The complement $\cY\ssm U$ has dimension at most $1$ and does not
depend on $P$. It contains $P$ by the previous paragraph. 
After omitting the finitely many isolated points in $\cY \ssm U$ 
we may suppose  that $P$ is  in a 
curve $\cC \subset \cY\ssm U$. Thus $\cC$  arises from a finite set depending
only on $\cY$.  

The restriction $R|_{\cC}\colon \cC\rightarrow\IG_m^n$ is non-constant
because we already reduced to the case where $P$ is isolated in the
fiber of $R|_{\cY}$. So $\overline{R(\cC)}$, the Zariski closure 
of $R(\cC)$ in $\IG_m^n$, is a curve.
By Theorem~\ref{thm:silverman} we have
\begin{equation}
\label{eq:quasiequiv}
  \height{P}\ll \max\{1,\height{x}\}, %
\end{equation}
where the constant implicit in $\ll$ depends only on $\cC$
and thus only on $\cY$.
 
If $\overline{R(\cC)}$  is not contained in a proper coset, then 
$\height{x}\ll 1$ by Theorem~\ref{thm:heightbound} or by Bombieri, %
Masser, and Zannier's original height bound \cite{BMZ99}. So
(\ref{eq:quasiequiv})
yields $\height{P}\ll 1$ and this is better than what the theorem claims. 


But what if $R(\cC)$ is contained in a proper coset of $\IG_m^n$?
As we have already pointed out, there is no hope 
that $\overline {R(\cC)}\cap\subgrpunion{(\IG_m^n)}{1}$
has bounded height.
But we know
that $\overline{R(\cC)}$ is not contained in a coset of codimension at least two
since $P\in \Qta{\cY}$. The projection of $\cC$
 to a suitable choice of  $n-1$ coordinates of
$\IG_m^{n-1}$ is a curve that is 
 not  in a proper coset. So if the coordinates of $x$ %
 happen to
 satisfy two independent multiplicative relations, then these $n-1$ coordinates
 will be multiplicatively dependent and thus have bounded height by
 Theorem~1 in \cite{BMZ99}. Using  Theorem~\ref{thm:silverman}, applied
 now to the projection, we can
 bound the height of the remaining coordinates. So
 $\height{x}\ll 1$ and even $\height{P}\ll 1$ by %
 (\ref{eq:quasiequiv}). 
Therefore,  we may assume that the coordinates of  $x$ satisfies %
only one multiplicative relation, up to scalars. From here we proceed in
a similar fashion as we have done several times before.
We use Lemma~\ref{lem:findb} to deduce that $b$ is linearly dependent
to some $ b'\in\IZ^n\ssm\{0\}$ with
$|b'|\ll  d^{2n}\max\{1,\height{x}\}^n$. %
We certainly have 
$\langle y,b\rangle = \langle y,b'\rangle =0$. %
The morphism $(x',y')\mapsto \langle y',b'\rangle$ does not vanish
identically on $\cC$ because $R(\cC)$ is already assumed to lie in a
proper coset and since  $P\in\Qta{\cY}$. So $P$ is an isolated point
 of $\cC\cap \{(x',y') : \langle y',b'\rangle=0\}$. 
A final application of the Arithmetic B\'ezout Theorem 
and (\ref{eq:quasiequiv}) yield
\begin{equation}
\label{eq:lasthb}
 \height{P}\ll \log(2|b'|)    \ll \log(2 d \max \{1,\height{P}\}). 
\end{equation}
The inequality (\ref{eq:lasthb}) marks the final subcase
in this proof and so the theorem is established. 
\end{proof}



%% file: background.tex
\section{Background on \Teichmuller curves} \label{sec:background}

We recall here some necessary background on flat surfaces and the stratification of
$\omoduli[g]$.  
We also recall the cross-ratio equation for the cusps of algebraically primitive
\Teichmuller curves and the defining equation for the exponents appearing in this equation.
These two equations were the motivation for the height bound theorem (Theorem~\ref{thm:heightmain}) in the
previous section.

\paragraph{Flat surfaces.}

A \emph{flat surface} is a pair $(X, \omega)$, where $X$ is a closed Riemann surface and $\omega$ a
nonzero holomorphic one-form on $X$.  The one-form $\omega$ gives $X$ a flat metric $|\omega|$ which
has cone points of cone angle $2\pi(n+1)$ at zeros of order $n$ of $\omega$.  The metric $|\omega|$
is obtained by pulling back the metric $|dz|$ on $\cx$ by local charts $\phi\colon U \to \cx$
defined by integrating $\omega$, defining an atlas on $X \setminus Z(\omega)$ (where $Z(\omega)$ is
the set of zeros of $\omega$) whose transition functions are translations of $\cx$.  There is an
action of $\SLtwoR$ on the moduli space of genus $g$ flat surfaces $\omoduli[g]$, defined by
postcomposing these charts with the standard linear action of $\SLtwoR$ on $\cx=\reals^2$.

Let $\Aff^+(X, \omega)$ be the group of locally affine homeomorphisms of $(X, \omega)$, and let
$D\colon\Aff^+(X, \omega)\to\SLtwoR$ the homomorphism sending an affine map to its derivative.  The
\emph{Veech group} of $(X, \omega)$ is the image $D\Aff^+(X, \omega)$, denoted by $\SL(X, \omega)$.
The group $\SL(X, \omega)$ is always discrete in $\SLtwoR$.  If it is a lattice, we call $(X,
\omega)$ a \emph{Veech surface}.  The $\SLtwoR$ orbit of a Veech surface is closed in $\omoduli[g]$
and called a \Teichmuller curve (we also refer to the projection of this orbit to $\pomoduli[g]$ or
$\moduli[g]$ as a \Teichmuller curve).

A \emph{saddle connection} on a flat surface $(X, \omega)$ is an embedded geodesic segment connecting 
two zeros of $\omega$. The foliation $\mathcal{F}_\theta$ of slope $\theta$ is said to be \emph{periodic} if every leaf of
$\mathcal{F}_\theta$ is either closed (i.e.\ a circle) or a saddle connection. In this case, we call $\theta$ a periodic direction.  A periodic direction $\theta$ yields a decomposition of $(X,\omega)$ into finitely many maximal cylinders foliated by closed geodesics of slope $\theta$.  We refer to the length
of the waist curve of the cylinder $C$ as its \emph{width} $w(C)$. 
The ratio of height over width is called the \emph{modulus} $m(C) = h(C)/w(C)$.
The complement of these cylinders is a finite collection of saddle connections.

We constantly use
Veech's dichotomy \cite{veech89} stating that if $(X, \omega)$ is 
a Veech surface with either a closed geodesic or a saddle connection of
slope $\theta$, then the foliation $\mathcal{F}_\theta$ is periodic
and the moduli of the cylinders in the direction $\theta$ are commensurable.
\par
Given a Veech surface $(X, \omega)$ generating a \Teichmuller curve $C\subset\proj\Omega\moduli$,
there is a natural bijection between the cusps of $C$ and the periodic directions on $(X, \omega)$,
up to the action of $\SL(X, \omega)$.  The cusp associated to a periodic direction $\theta$ is the
limit of the \Teichmuller geodesic given by applying to $(X, \omega)$ the one-parameter subgroup of
$\SLtwoR$ contracting the direction $\theta$ and expanding the
perpendicular direction.  The stable form in $\proj\Omega\barmoduli$ which is the limit of this cusp
is obtained by cutting each cylinder of slope $\theta$ along a closed geodesic and gluing a
half-infinite cylinder to each resulting boundary component (see \cite{masur75}).  These infinite
cylinders are the poles of the resulting stable form, and the two poles resulting from a single
infinite cylinder are glued to form a node.
\par
The \emph{spine} of $(X, \omega)$ is the union of all horizontal saddle connection, whose complement
is a union of cylinders.  The \emph{dual graph} $\Gamma$ is the graph whose vertices correspond to components
of the spine, and edges correspond to complementary cylinders.  Equivalently, $\Gamma$ is the dual
graph of the stable curve associated to this periodic direction.  Its vertices correspond to
irreducible components, and its edges correspond to nodes.
\par

\paragraph{\bf Strata of $\omoduli[g]$.} The moduli space of flat surfaces
$\omoduli$ is stratified according to the number and multiplicities
of zeros of $\omega$. The stratum $\omoduli[g](1^{2g-2})$ parameterizing flat
surfaces with only simple zeros is open and dense. It is
called the {\em principal stratum}.
\par
Connected components of the strata were classified in \cite{kz03}.  In genus three all the strata
but $\omoduli[3](4)$ and $\omoduli[3](2,2)$ are connected.  These two strata have two connected
components, distinguished by the {\em parity of the spin structure} $h^0(X,{\rm div}(\omega)/2)$.
We write $\omoduli[3](4)^{\rm hyp}$ and $\omoduli[3](4)^{\rm odd}$ for the hyperelliptic and odd
components respectively.  The case of even spin parity coincides in genus three with the
hyperelliptic components, the component in which all the curves are hyperelliptic. Note that also
$\omoduli[3](2,2)^\odd$ contains hyperelliptic curves, forming a divisor in this stratum. In this
case, the hyperelliptic involution fixes the zeros, while for a flat surface in
$\omoduli[3](2,2)^\hyp$ the hyperelliptic involution swaps the two zeros.

\paragraph{\bf The family over a \Teichmuller curve.}
Let $f\colon \cX \to C$ be the universal family over a  \Teichmuller curve
generated by a Veech surface $(X,\omega)$ (or possibly the pullback
to an unramified covering of $C$). 
We also denote by $f\colon \ol{\cX} \to \ol{C}$
the corresponding extension to a family of stable curves.
Again passing to an unramified covering of $C$ we may suppose
that the zeros of $\omega$ define sections  $s_j\colon \ol{C} \to \ol{\cX}$
and we let $D_j = s_j(\ol{C})$ denote their images.
In a fiber $X$ of $f$ we write $z_j$ for the zeros of $\omega$, 
that is is for the intersection of $X$ with $D_j$. We also 
write just $D$ for the section and $z$ for the zero, if $k=1$.
\par
\paragraph{\bf The eigenform locus and its degeneration.} Let $\RM \subset \moduli[g]$
denote the locus of Riemann surfaces that admit real multiplication
by an order $\cO$ in a totally real number field $F$ with $[F:\ratls]=g$. In the bundle of one-forms
over $\RM$ there is the locus of {\em eigenforms} $\E \subset \omoduli[g]$ 
consisting of pairs $(X,\omega)$, where $[X] \in \RM$ and where $\omega$
is an eigenform for real multiplication. The intersection
of the closure of $\RM$ in $\barmoduli$  with the boundary was described
in \cite{BaMo12}.   We gave a necessary and sufficient condition for a stable form to lie in the
boundary of $\RM$ in genus three where there is no Schottky
problem involved.  We summarize these results here.
\par
The irreducible stable curves in the boundary of $\RM$ are {\em trinodal curves}, 
rational curves with three pairs of points $x_i$ and $y_i$, $i=1,2,3$
identified. Coordinates on the boundary component of trinodal curves
in $\barmoduli[3]$ are given by the cross-ratios 
defined by
\begin{equation} \label{CR6}
R_{jk}=[x_j,y_{j}, x_{k}, y_k].
\end{equation}
where for $z_1, \ldots z_4\in\cx$,  $$[z_1,z_2, z_3, z_4]=\frac{(z_1-z_3)(z_2-z_4)}{(z_1-z_4)(z_2-z_3)}.$$
We often use complementary index notation, writing $R_k$ for $R_{ij}$
where $\{i,j,k\}=\{1,2,3\}$.
\par
In general, cusps of Hilbert modular varieties are determined by the ideal
class of a nonzero module $\cI$ for the order $\cO$ and an extension class $E$. 
In \cite{BaMo12} we called a triple $(r_1,r_2,r_3) \in F$ {\em admissible}, if $\{r_1,r_2,r_3\}$
is a basis for such an ideal $\cI$ and if $\left\{\frac{N(r_1)}{r_1},\frac{N(r_2)}{r_2},
\frac{N(r_3)}{r_3}\right\}$ is $\ratls^+$-linear dependent.
\par
The boundary components of $\RM$ intersected with the locus of trinodal curves
are in bijection with projectivized admissible triples up to permutation and sign change.
On the component given by an admissible triple $(r_1,r_2,r_3)$, the closure 
of $\RM$ is cut out by the cross-ratio equation
\begin{equation} \label{eq:CREq}
R_1^{a_1} R_2^{a_2} R_3^{a_3} = \zeta_E,
\end{equation}
where $\zeta_E$ is a root of unity determined by the extension class $E$ 
(see \cite{BaMo12} for the precise definition, which  is irrelevant in this
paper) and where the exponents $a_i$ are defined as follows.
\par
Let $(s_1,s_2,s_3)$ be dual to $(r_1,r_2,r_3)$ with respect to the trace pairing on $F$,
and let $(b_1,b_2,b_3) \in \zed^3$ be such that (indices read mod $3$)
$$ \sum_{i=1}^1 b_i s_{i+1}s_{i+2} = 0,$$
as stated in \cite[Proof of Theorem~8.5]{BaMo12}.
The existence of such $b_i$ and the fact that $b_i \neq 0$ for all $i$ 
is a consequence of admissibility. We let
$(a_1,a_2,a_3)$ be a tuple proportional 
to $(b_1,b_2,b_3)$ that is relatively prime.  This only determines the $a_i$ up to sign.  A more
precise description of the $a_i$ including the sign can be found in \cite{BaMo12}, but that choice
of sign is not relevant in this paper.   The defining condition
for the cross-ratio exponents can equivalently be stated as
\begin{equation} \label{eq:defCREXP}
\sum_{i=1}^3 b_i/s_i = 0.
\end{equation} 
\par

\paragraph{Applications of algebraic primitivity.}
By \cite{moeller06} an algebraically primitive \Teichmuller curve lies
in the real multiplication locus $\RM$ for some order ${\cal O}$ in the trace field $F$. Consequently, 
cusps of Teichm\"uller curves, if they correspond to 
an irreducible stable curve, have an associated admissible triple $(r_1,r_2,r_3)$, which is a basis
of $F$ over $\ratls$, and the corresponding stable curve
satisfies the cross-ratio equation \eqref{eq:CREq}, with exponents defined
by \eqref{eq:defCREXP}. We constantly use the fact from \cite{BaMo12} that 
 this triple is, up to a suitable rescaling and permutation of indices, 
the triple of widths of cylinders in a periodic direction that has this
associated cusp. 


%% file: HNfilt.tex
\section{Harder-Narasimhan filtrations} \label{sec:HNfilt}

In this section, we recall the notion of a Harder-Narasimhan filtration of a vector bundle, and
following Yu-Zho \cite{yuz1,yuz2}, we describe the Harder-Narasimhan filtration of the Hodge bundle
for most strata in genus three.  We apply this in
Proposition~\ref{prop:HNFisEigsplit} to obtain information on the zeros of the other eigenforms of
Veech surfaces.

Let $\cV$ be a vector bundle on a compact curve $\ol{C}$. The {\em slope} of a 
vector bundle is defined as $\mu(\cV) = \deg(\cV)/\rank(\cV)$. A bundle
is called {\em semistable} if it contains no subbundle of strictly larger slope.
A filtration
$$ 0 = \cV_0 \subset \cV_1 \subset \cV_2 \cdots \subset \cV_g = \cV$$
is called a {\em Harder-Narasimhan filtration} if the successive
quotients $\cV_{i}/\cV_{i-1}$ are semi-stable and the slopes are strictly
decreasing, i.e.\ 
$$\mu_i := \cV_{i}/\cV_{i-1}  > \mu_{i+1} := \cV_{i+1}/\cV_i.$$
The Harder-Narasimhan filtration is the unique filtration with these
properties (see e.g.\ \cite{huylehn}).
\par
\par
We now study the Harder-Narasimhan filtration in the case of a family $f\colon \cX \to C$ over
a \Teichmuller curve in any nonprincipal stratum in genus threem with the
Hodge bundle $\cV = f_* {{\omega_{\overline{{\mathcal X}}/\overline{{\mathcal{C}}}}}}$ 
over $\overline{C}$. Here $\omega_{\cX /\mathcal{C}}$ denotes the
relative dualizing sheaf of $f$, whose sections are fibrewise stable differentials. 
\par
\begin{prop}[\cite{yuz1}] \label{eq:HNfiltg3}
For $C$ a \Teichmuller curve in any nonprincipal stratum in genus three, the Harder-Narasimhan filtrations of $f_* {{\omega_{\overline{{\mathcal X}}/\overline{{\mathcal{C}}}}}}$  are
given by the direct image sheaves
$$\begin{array}{llll}
\omoduli[3](4)^\hyp:\quad & f_* {{\omega_{\overline{{\mathcal X}}/\overline{{\mathcal{C}}}}}}(-4D_1) & \subset f_*{{\omega_{\overline{{\mathcal X}}/\overline{{\mathcal{C}}}}}}(-2D_1) &  \subset f_* {{\omega_{\overline{{\mathcal X}}/\overline{{\mathcal{C}}}}}}. \\
\omoduli[3](4)^\odd:\quad & f_* {{\omega_{\overline{{\mathcal X}}/\overline{{\mathcal{C}}}}}}(-4D_1) &\subset f_*{{\omega_{\overline{{\mathcal X}}/\overline{{\mathcal{C}}}}}}(-D_1)  & \subset f_* {{\omega_{\overline{{\mathcal X}}/\overline{{\mathcal{C}}}}}}. \\
\omoduli[3](3,1):\quad & f_* {{\omega_{\overline{{\mathcal X}}/\overline{{\mathcal{C}}}}}}(-3D_1 -D_2)& \subset f_*{{\omega_{\overline{{\mathcal X}}/\overline{{\mathcal{C}}}}}}(-D_1) & \subset f_* {{\omega_{\overline{{\mathcal X}}/\overline{{\mathcal{C}}}}}}. \\
\omoduli[3](2,2)^\hyp:\quad & f_* {{\omega_{\overline{{\mathcal X}}/\overline{{\mathcal{C}}}}}}(-2D_1- 2D_2) &\subset f_*{{\omega_{\overline{{\mathcal X}}/\overline{{\mathcal{C}}}}}}(-D_1) & \subset f_* {{\omega_{\overline{{\mathcal X}}/\overline{{\mathcal{C}}}}}}. \\
\omoduli[3](2,2)^\odd \quad &  f_*{{\omega_{\overline{{\mathcal X}}/\overline{{\mathcal{C}}}}}}(-2D_1 -2D_2) & \quad  \quad \subset & 
\phantom{\subset} f_* {{\omega_{\overline{{\mathcal X}}/\overline{{\mathcal{C}}}}}}. \\
\omoduli[3](2,1,1):\quad & f_* {{\omega_{\overline{{\mathcal X}}/\overline{{\mathcal{C}}}}}}(-2D_1-D_2-D_3) &\subset f_*{{\omega_{\overline{{\mathcal X}}/\overline{{\mathcal{C}}}}}}(-D_1) & \subset f_* {{\omega_{\overline{{\mathcal X}}/\overline{{\mathcal{C}}}}}}. \\
\end{array}
$$
For a Teichmuller curre that is  a family of hyperelliptic curves in the 
stratum  $\omoduli[3](2,2)^\odd$ the 
bundle $f_* {{\omega_{\overline{{\mathcal X}}/\overline{{\mathcal{C}}}}}}/ f_* {{\omega_{\overline{{\mathcal X}}/\overline{{\mathcal{C}}}}}}(-2D_1 - 2D_2)$ 
is the direct sum of the two line bundles  
$f_* {{\omega_{\overline{{\mathcal X}}/\overline{{\mathcal{C}}}}}}(-2D_i)/ f_* {{\omega_{\overline{{\mathcal X}}/\overline{{\mathcal{C}}}}}}(-2D_1 - 2D_2) \cong f_* {{\omega_{\overline{{\mathcal X}}/\overline{{\mathcal{C}}}}}} (-4D_i)$, $i=1,2$,
of the same slope. That is, 
$$ f_* {{\omega_{\overline{{\mathcal X}}/\overline{{\mathcal{C}}}}}} = f_* {{\omega_{\overline{{\mathcal X}}/\overline{{\mathcal{C}}}}}}(-2D_1-2D_2) \oplus f_* {{\omega_{\overline{{\mathcal X}}/\overline{{\mathcal{C}}}}}} (-4D_1) \oplus f_* {{\omega_{\overline{{\mathcal X}}/\overline{{\mathcal{C}}}}}}(-4D_2).$$
\end{prop}
\par
In each case the bottom term $\cV_1$ is the maximal Higgs subbundle, its
fibres are the generating one-form of the \Teichmuller curve and
$$\deg \cV_1 = \tfrac12 \deg \Omega^1_{\ol{C}}(\partial C),$$
where 
$\partial C = \overline{C} \setminus C$.  In the
strata $\omoduli[3](2,1,1)$ and $\omoduli[3](3,1)$ the filtration is split as well, 
i.e.\  a direct sum of line bundles, see \cite[Section~5.3]{yuz1}.
\par
The first case is a special case of the following the more general statement
for the strata $\omoduli[g](2g-2)^\hyp$.
\par
\begin{prop} \label{eq:HNhypminimal} For a \Teichmuller curve generated
by a Veech surface in  the stratum $\omoduli[g](2g-2)^\hyp$ the 
bundle $\cV_j = f_* {{\omega_{\overline{{\mathcal X}}/\overline{{\mathcal{C}}}}}}(-(2g-2j)D)$
is a vector bundle of rank $j$. The quotient line bundles $\cV_j / \cV_{j-1}$ have degree 
\begin{equation}
  \label{eq:9}
  \deg(\cV_j / \cV_{j-1}) = \frac{2g+1-2j}{2g-1}\deg(\cL).
\end{equation}
In particular, the filtration by the $\cV_j$ is the HN-filtration 
of $f_* \omega_\rel$.
\end{prop}
\par
\begin{proof}
  We reproduce the argument of \cite{yuz1} for convenience. For all fibres $X$ of $f$ the dimensions
  of the following cohomology spaces are the same, namely for $j$ odd $h^0(X,\cO_X(jz))=(j+1)/2$ and
  for $j$ even $h^0(X,\cO_X(jz))=(j+2)/2$.  Consequently, the direct image sheaves
  $f_{*}\cO_{\cX}(jD)$ and $R^1f_{*}\cO_{\cX}(jD)$ are vector bundles.
  \par
  Suppose first that $j$ is odd.  Since every section of $f_{*}\cO_{\cX}(jD)$ is also a section of
  $f_{*}\cO_{\cX}((j-1)D)$, these bundles are then isomorphic.

  The long exact sequence associated to
  $$ 0 \to \cO_\cX ((j-1)D) \to \cO_\cX (jD) \to \cO_{D} (jD) \to 0$$
  is
  $$ 0 \to f_{*}\cO_{\cX}((j-1)D) \to f_{*}\cO_{\cX}(jD) \to f_{*}\cO_{jD}(jD) $$
  $$ \to R^1 f_{*}\cO_{\cX}((j-1)D) \to R^1f_{*}\cO_{\cX}(jD) \to 0. $$
  \par
  Since $f|_D$ is an isomorphism, the middle term is a line bundle.  Its degree is
  $$\deg f_{*}\cO_{jD}(jD) = j D^2 = \frac{j}{2g-1}\deg \cL$$
  by \cite[Lemma~4.11]{MoeST}.  Here $\cL=\cV_1 \subset f_* {{\omega_{\overline{{\mathcal X}}/\overline{{\mathcal{C}}}}}}$ is the (``maximal Higgs'') line bundle
  whose fibers are the generating one-forms of the \Teichmuller curve.

  Now supposes $j$ is even.  By Serre duality and the same argument as in the odd case, the last map
  of the long exact sequence is an isomorphism.  Hence the degree of $f_{*}\cO_{\cX}((j-1)D)$ and
  $f_{*}\cO_{\cX}(jD)$ differs by $\deg  \cO_{D} (jD)$. 

  To obtain \eqref{eq:9} from $\deg \cO_{D} (jD)$, note that note that
  \begin{equation*}
    f_* {{\omega_{\overline{{\mathcal X}}/\overline{{\mathcal{C}}}}}}((2j- 2g+ 2)D) = \cL
    \otimes f_{*}\cO_{\cX}(2jD).
  \end{equation*}

  The last statement is always true for a filtration, whose successive quotients are line bundles
  with strictly decreasing degrees.
\end{proof}
\par
\begin{proof}[Proof of Proposition~\ref{eq:HNfiltg3}] In all the
cases the direct images are vector bundles, since the dimensions 
of their fibers are constant by Riemann-Roch and by definition of
the parity of the spin structure. The degrees of the successive
quotients can be computed by the same method as in the last section.
The values appear in \cite[Table~1]{yuz1} (rescaled dividing 
by $\deg \cV_1$) and the degrees are decreasing. Consequently the
filtation is the Harder-Narasimhan filtration.
\par
We provide full details in the case $\omoduli[3](2,2)^\odd$ and prove
the last statement of the proposition. Note that by the parity of
the spin structure implies $f_* {{\omega_{\overline{{\mathcal X}}/\overline{{\mathcal{C}}}}}} (-2D_1 -2D_2) = f_* {{\omega_{\overline{{\mathcal X}}/\overline{{\mathcal{C}}}}}} (-D_1-D_2)$.
By \cite{yuz1}, if
$$h^0(\cO_X(\sum_{i=1}^n d_iz_i)) =  h^0(\cO_X(\sum_{i=1}^n (d_i-a_i)z_i)) + \sum_{i=1}^n  a_i$$ 
then
$$f_* \cO_X(\sum_{i=1}^n a_iD_i)/f_*\cO_X (\sum_{i=1}^n (a_i-d_i)D_i) \cong 
\oplus_{i=1}^n f_* \cO_{D_i}(a_iD_i).$$ 
In this stratum we apply this to the case $a_1=a_2=2$ and $d_1 = d_2=1$. 
Since the self-intersection number of $D_i$ only depends on the
order of the corresponding zero $z_i$ we  obtain that the rank two second 
step of the filtration is  a direct sum of two line bundles of the same 
slope. 
\par
In the hyperelliptic case the bundles $f_* {{\omega_{\overline{{\mathcal X}}/\overline{{\mathcal{C}}}}}}(-2D_i)$ are of rank
two for $i=1,2$, each of the being the direct sum of 
$f_*{{\omega_{\overline{{\mathcal X}}/\overline{{\mathcal{C}}}}}}(-4D_i)$ and $f_* {{\omega_{\overline{{\mathcal X}}/\overline{{\mathcal{C}}}}}}(-2D_1-2D_2)$. By Riemann-Roch and the
preceding remark on the self-intersection number of $D_i$ we conclude
$\deg f_*{{\omega_{\overline{{\mathcal X}}/\overline{{\mathcal{C}}}}}}(4D_1) = \deg f_* {{\omega_{\overline{{\mathcal X}}/\overline{{\mathcal{C}}}}}}(4D_2)$ and all the claims follows.
\end{proof}

\paragraph{\bf Consequences for the algebraically primitive case.}
Suppose from now on that $(X,\omega)$ is algebraically primitive.
Then
$$ f_* {{\omega_{\overline{{\mathcal X}}/\overline{{\mathcal{C}}}}}} = \oplus_{j=1}^g \cL_j,$$
where the $\cL_j$ are the bundles generated by the eigenforms
for real multiplication by the trace field $F$. We
may enumerate them such that $\deg(\cL_j)$ is non-increasing
with $j$. We let $\omega^{(j)} \in H^0(X,\Omega^1_X)$ be a one-form
on the generating Veech surface $X$ that spans $\cL_j$. More generally, 
we write $\omega_c^{(j)}$ for a generator of $\cL_j$ in the fiber over $c \in \ol{C}$.
For $g=3$, we also label the bundles and their generators with field embeddings
$\omega = \omega^{(1)}$, $\omega^\sigma = \omega^{(2)}$ and $\omega^\tau = \omega^{(3)}$.
\par
\begin{prop} \label{prop:HNFisEigsplit} 
In  each of the strata $\omoduli[3](4)^\hyp$, $\omoduli[3](4)^\odd$, 
 $\omoduli[3](3,1)$, $\omoduli[3](2,1,1)$, and $\omoduli[3](2,2)^\hyp$ 
the second step of the HN-filtration consists of the first two eigenform bundles, i.e.\ 
$$\cL_1 \oplus \cL_2 = \cV_2 \subset \cV_3 = f_* {{\omega_{\overline{{\mathcal X}}/\overline{{\mathcal{C}}}}}}.$$
Consequently, $\cL_2$ is generated for all $c \in \ol{C}$ by a one-form $\omega_c^{(2)}$
with a zero at $D_1$. This zero is necessarily a double zero in the case  $\omoduli[3](4)^\hyp$.
\end{prop}
\par
This is based on the following principle. 
\par
\begin{lemma} \label{le:HNFlinebundles} 
If $\cV = \oplus_{j=1}^g \cL_j$ is a direct sum of line bundles, ordered with
non-increasing degrees,  and if the 
successive quotients $\cV_j / \cV_{j-1}$ of the Harder-Narasimhan filtration 
are line bundles, then $ \cV_i = \oplus_{j=1}^i \cL_j.$
\end{lemma}
\par
\begin{proof}
Since $\cV_1$ is the unique line subbundle of $\cV$ of maximal degree, 
only one of the projection maps $\cV_1 \to \cV \to \cL_j$
is non-zero, by the non-increasing ordering necessarily for $j=1$.
By maximality of the degree, this map is an isomorphism. We may
consequently consider $\cV / \cV_1$ and proceed inductively.
\end{proof}
\par

\begin{proof}[Proof of Proposition~\ref{prop:HNFisEigsplit}] This follows from
Lemma~\ref{le:HNFlinebundles} and Proposition~\ref{eq:HNfiltg3}
\end{proof}
\par
\begin{proof}[Proof of Theorem~\ref{thm:hyperelliptic_zeros}] This is a direct consequence
of Lemma~\ref{le:HNFlinebundles} together with Proposition~\ref{eq:HNhypminimal}.
\end{proof}

\paragraph{The fixed part in genus three.}

As a final application of the Harder-Narasimhan filtration, we discuss the fixed part of the family
of Jacobians over a \Teichmuller curve, which is one source of zero Lyapunov exponents of the
Kontsevich-Zorich cocycle.

Let $h: \Jac(\cX/C) \to C$ be the family of Jacobian varieties over a \Teichmuller curve.
The family $h$ is said to have a {\em fixed part} of dimension $d$, if there is an abelian
variety $A$ of dimension $d$ and an inclusion $A \times C \to \Jac(\cX/C)$ of the constant
family with fiber $A$ into the family of Jacobians. It was shown in \cite{MoeST}, that the
only \Teichmuller curve in genus three with a fixed part of rank two is the 
family $y^4 = x(x-1)(x-t)$, generated by a square tiled surface in the stratum $\omoduli[3](1,1,1,1)$. 
Studying the fixed part of \Teichmuller curves is motivated from dynamics, since
a \Teichmuller curve with a Forni-subspace of rank $2d$ has a fixed part of dimension $d$,
see \cite{Au12} for definitions and background.
\par
\begin{prop} There does not exist a \Teichmuller curve $C$ generated by a genus three Veech surface in
a stratum other than $\omoduli[3](1,1,1,1)$ with a positive-dimensional fixed part.
\end{prop}
\par
\begin{proof} By the preceding remark we may restrict to a one-dimensional fixed part.
The variation of Hodge structures  over $C$ decomposes into rank-two summands $R^1f_* \cx = \LL \oplus \UU \oplus \MM$, 
where $\LL$ is maximal Higgs, $\UU$ is the unitary summand stemming from the fixed part
and $\MM$ is the rest. The $(1,0)$-pieces of these summands form a decomposition
of $f_* {{\omega_{\overline{{\mathcal X}}/\overline{{\mathcal{C}}}}}}$ into line bundles with $\deg(\MM^{1,0})> 0 = \deg(\UU^{(1,0)})$. For all but
the stratum $\omoduli[3](2,2)^\odd$ the claim follows from Lemma~\ref{le:HNFlinebundles}
and the fact that the lowest degree quotient of the filtration in Proposition~\ref{eq:HNfiltg3}
does not have degree zero, as calculated in \cite[Table~1]{yuz1}.
\par
In the remaining case, the inequality of degrees in the non-maximal Higgs part implies that the
Harder-Narasimhan filtration has three terms, 
contradicting Proposition~\ref{eq:HNfiltg3}, which says that the filtration has only two terms.
\end{proof}


%% file: toruscontainment.tex
\section{Tori contained in subvarieties of $\IG_m^n$} \label{sec:torusalgo}

At two occasions in this paper, for the principal stratum $\omoduli[3](1,1,1,1)$
and for the stratum $\omoduli[3](4)^\odd$,  we are facing the task
of enumerating torus translates in an algebraic subvariety $Y$ of $\IG_m^n$.
In this section we describe an algorithm that we have
implemented in both cases to check that all the tori
contained in the subvariety are irrelevant for the
finiteness statements we are aiming for.

When we apply this algorithm in \S\ref{sec:pres-tors-princ}, we will actually need only to find
torus-translates contained in $Y$ which are parallel to a subtorus of a fixed torus $T\subset
\IG_m^9$ of large codimension.  This is a useful reduction, since the running time of the
algorithm increases exponentially in $n$ and is only useful in practice for very small dimensions.

To this end, given a rank $r$ subgroup $M\subset\ratls^n$, determining a subtorus
$T_M\subset \IG_m^n$, let $V_M\subset \IG_m^n$ denote the subvariety of $\bba\in \IG_m^n$ such
that $\bba  T_M \subset Y$.  We wish to enumerate those codimension-one subspaces $N\subset M$
such that the $V_N$ potentially strictly contains $V_M$.  Applying this procedure inductively then
produces a list of subspaces $N\subset M$ and varieties $V_N$ which account for all of the
torus-translates contained in $Y$.  

Consider a $r$-by-$n$ matrix $E$ whose rows vectors span $M$.  A coefficient vector $\bba\in
\IG_m^n$ determines a parametrization $f_{E,\bba}\colon \IG_m^r \to \bba T_M$ given by
\begin{equation*}
  f_{E, \bba}(\bt) = (a_1 \bt^{E_1}, \ldots, a_s \bt^{E_s}),
\end{equation*}
where the $E_i$ are the column vectors of $E$.
Suppose first that $Y$ is defined by the single polynomial $h(\bz) = \sum
b_I \bz^I$  in the variables $z_1, \ldots, z_n$ with coefficients in a number field.  
The vanishing of the composition
\begin{equation*}
  h \circ f_{E, \bba}(\bt) = \sum_I b_I \bba^I \bt^{E\cdot I}
\end{equation*}
is then equivalent to the vanishing of the system of rational functions 
\begin{equation*} \label{eq:def_of_Pij}
  p_J(\bba) = \sum_{E\cdot I = J} b_I a^I
\end{equation*}
obtained by partitioning the coefficients according to the images of the exponent vectors
$I\in\supp{h}$ under $E$, or equivalently according to their images under the orthogonal projection
to $M$.  We take the numerator of each $p_J$, yielding an ideal in $\IQbar[a_1, \ldots
a_s]$\footnote{Although we are only interested in solutions in $\IG_m^n$, it is computationally
  convenient to work over the polynomial ring and later discard components contained in the coordinate
  hyperplanes}.

Now suppose $Y$ is defined by an ideal $I = (h_1, \ldots, h_s)$.  Applying this construction to
each $h_i$, the collection of resulting polynomials defines an ideal $I_M\subset\IQbar[a_1,
\ldots, a_s]$ which cuts out the desired variety $V_M$.  We call $I_M$ the \emph{coefficient ideal}
of $I$.  

Now, for a generic subspace $N\subset M$, the partition of each $\supp{h_i}$ induced by orthogonal
projection to $N$ will be unchanged, so $V_N = V_M$.  If $v,w\in\supp{h_i}$ are identified by
projection to $N$, but not by projection to $M$, then $N$ must be the orthogonal complement of
$p_M(v-w)$ in $M$.  Enumerating all $N$ arising in this way then yields all $N$ for which $V_N$
potentially strictly contains $V_M$.

Applying this algorithm inductively to the list of subspaces obtained in each dimension, we obtain a
list of subspaces $N\subset M$ together with varieties $V_N$, which together account for all
torus-translates contained in $Y$.

The complete algorithm builds in this way the possible subspaces $N$ 
with $1 \leq {\rm rank}(N) \leq r$ from the difference of the projections 
of monomials onto any subspace of larger rank that is already in the
list of candidate subspaces. In the next step it assembles the ideals defining
$V_{N}$, as summarized in Algorithm~\ref{cap:algo}.

\par
\begin{figure}
\begin{algorithm}[H]
\SetKwData{Subspaces}{Subspaces}
\SetKwData{ProjectionList}{ProjectionList}
\SetKwData{TorusTranslates}{TorusTranslates}
\SetKwFunction{OrthProjection}{OrthProjection}
\SetKwFunction{OrthComplement}{OrthComplement}
\SetKwFunction{AddTo}{AddTo}
\SetKwFunction{Extract}{Extract}
\KwData{List of polynomials $h_1,\ldots,h_s \in \overline{\IQ}[z_1,\ldots,z_n]$, 
a matrix $M \in \IZ^{r \times n}$}
\KwResult{Pairs $(N,V)$ of matrices $N$ and subvarieties $V \subset \IG_m^n$
}
\BlankLine
\tcc{Build the set subspaces $N$ that may occur recursively}
$\Subspaces=\{M\}$\;
\For{$S \in \Subspaces$}{
\For{$i=1$ \KwTo $s$}{
\ProjectionList $\leftarrow$ \OrthProjection{$\supp{h_i},S$}\; 
\For{$(j,k) \in \{1,\ldots,|\ProjectionList|$\}}{
$v \leftarrow \ProjectionList[j] - \ProjectionList[k]$\;
\AddTo(\Subspaces, \OrthComplement(v,S))\;
}
}
}
\tcc{Check subspaces for the existence of a torus translate}
\TorusTranslates = \{\}\;
\For{$N \in \Subspaces$}{
$I_N \leftarrow \langle \rangle$\; 
\For{$i=1$ \KwTo $s$}{
\ForEach{$j \in \OrthProjection{\supp{$h_i$},$N$}$}{
$I_N$ = $I_N$ + \Extract($h_i$,\,\,$j$) \tcc*[r]{Extracts $p_j$ \\ $P_{i,j}$ 
as in \eqref{eq:def_of_Pij}}
}
}
\If{$I_N \neq \langle 1\rangle$ \tcc*[r]{Discard if $I_N$ is the unit ideal}}{
\AddTo(\TorusTranslates,($N$,$V(I_N)$)) }
}
\Return \TorusTranslates
\caption{Torus containment} \label{cap:algo}
\end{algorithm}
\end{figure}
\par
\smallskip


%% file: minimalstrata.tex
\section{Finiteness in the minimal strata}

The aim of this section is to prove the finiteness
result Theorem~\ref{thm:intromainfin} in the cases
when the torsion condition is void (i.e.\ in 
the strata $\omoduli[3](4)^\hyp$ and $\omoduli[3](4)^\odd$) or of limited use, 
as in the hyperelliptic locus of the stratum $\omoduli[3](2,2)^\odd$ where it is
automatically satisfied when both the points are Weierstrass
points.
\par
In all three cases we study the degenerate fibers of the universal family $f\colon \cX \to C$ over
(a cover of) an algebraically primitive \Teichmuller curve.  We first make explicit the conditions
arising from real multiplication and the Harder-Narasimhan filtration, which puts us in the
situation considered in Theorem~\ref{thm:heightmain}.  We then
check that the  hypothesis of Theorem~\ref{thm:heightmain} are met.
\par
The case $\omoduli[3](4)^\hyp$ is quickly dealt with and shows all the essential features.
The case  $\omoduli[3](4)^\odd$ requires a long detour to check the
hypothesis (i) and (ii) of  Theorem~\ref{thm:heightmain}.  Note for comparison that 
in \cite{MatWri13} the odd stratum is of no more complexity than the hyperelliptic.

The case of the hyperelliptic locus in $\omoduli[3](2,2)^\odd$ is quite different, as the
Harder-Narasimhan filtration has a rank-two piece.  The information this yields seems less useful,
and unfortunately our methods fail completely for this locus.  To handle this case, we instead
appeal to \cite{MatWri13}.

\subsection{The stratum  $\omoduli[3](4)^\hyp$}

Let $X_\infty$ be a degenerate fiber of $f\colon\mathcal{X}\to C$.  It is necessarily irreducible, since the
generating form $\omega$ has a single zero.  It is geometric genus zero by real multiplication.
Hence $X_\infty$ is a trinodal curve.   We
let $\PP^1$ with coordinate $z$ be the normalization of $X_\infty$. We may choose the coordinate $z$
such that the hyperelliptic involution is $z \mapsto -z$ and that the zero section $D$
specializes to $z=0$. The three pairs of points on the normalization of the trinodal curve are
thus $x_i$ and $y_i = -x_i$, $i=1,2,3$. This normalization still leaves one parameter for scaling
the $x_i$ multiplicatively.  Due to this choice of $z$ the generating eigenform specializes to
\begin{equation}
  \label{eq:35}
  \omega_\infty = \sum_{i=1}^3 \left(\frac{r_i}{z-x_i} - 
\frac{r_i}{z+x_i} \right)dz =
  \frac{Cz^4}{\prod_{i=1}^3(z^2-x_i^2)} dz.
\end{equation}
The condition that $\omega$ has a four-fold zero amounts to the equations
\begin{gather}
  \label{eq:1stra4hyp}
  \sum_{i=1}^3 r_i x_{i+1}x_{i+2} =0 \\
  \label{eq:2stra4hyp} \sum_{i=1}^3 r_i x_i (x_{i+1}^2 + x_{i+2}^2) =0,
\end{gather} where indices
are to be read mod $3$.
\par
Let $\omega^\sigma$ be one of the two Galois conjugate eigenforms, 
the one generating the eigenform bundle of second largest degree.
From Proposition~\ref{prop:HNFisEigsplit} we deduce the following information.
\par
\begin{cor}
The form $\omega^\sigma$ has a double zero along $D$. In particular
\begin{equation}
  \label{eq:35sigma}
  \omega^\sigma_\infty = \sum_{i=1}^3 \left(\frac{r_i^\sigma}{z-x_i} - 
\frac{r_i^\sigma}{z+x_i} \right)dz = 
  \frac{C_1 z^4+C_2z^2}{\prod_{i=1}^3(z^2-x_i^2)} dz,
\end{equation}
which can also be expressed by the condition
\begin{equation} \label{eq:1stra4hypsigma}
  \sum_{i=1}^3 r_i^\sigma x_{i+1}x_{i+2} =0.
\end{equation}
\end{cor}
\par
Note that for the third embedding, denoted by $\tau$, the analogous condition
$\sum_{i=1}^3 r_i^\tau x_{i+1}x_{i+2} =0$ should not hold since 
$f_* {{\omega_{\overline{{\mathcal X}}/\overline{{\mathcal{C}}}}}}(-D)$
has just rank two. Indeed this relation does not hold for the degenerate
fiber of the $7$-gon. The values are given in \cite[Example~14.4]{BaMo12}.
\par
We may normalize to $r_3=1$ and to $x_3=1$. 
\par
\begin{cor}
  \label{cor:xi_dual}
The points $x_i$ scaled such that $x_3=1$ lie in the Galois
closure of trace field $F$. 
Moreover, the tuple $1/x_i^\sigma$ is proportional to the dual basis $(s_1,s_2,s_3)$
of $(r_1,r_2,r_3)$.  In particular, the $x_i$ are real.
\end{cor}
\par
\begin{proof}
We can rewrite \eqref{eq:1stra4hyp} and \eqref{eq:1stra4hypsigma}
as $\sum_{i=1}^3 r_i/x_i = 0$ and $\sum_{i=1}^3 r_i^\sigma / x_i = 0$. This implies
the second statement and the first follows.

Alternatively, we can deduce
from \eqref{eq:1stra4hyp} and \eqref{eq:1stra4hypsigma} the equality 
\begin{equation} \label{eq:x1x2rel}
(r_1 - r_1^\sigma) x_2 + (r_2 - r_2^\sigma) x_1 =0,
\end{equation}
i.e.\ the ratio $x_1/x_2 \in F$. Plugging this information back
into \eqref{eq:1stra4hyp} implies the first statement.
\end{proof}
\par
\begin{prop} \label{prop:setupin4hyp}
In $\IA^2$ with coordinates $(x_1,x_2)$ we define 
$$ \cY = \IA^2\ssm\{(x_1,x_2): x_1x_2 (x_1\pm 1)(x_2 \pm 1)(x_1\pm x_2)=0\}.$$
If we define $R \colon \cY \to \IG_m^3$  by
\begin{equation*}
  R(x_1,x_2) = 
\left(\frac{x_2-1}{x_2+1},
\frac{1-x_1}{1+x_1},
\frac{x_1-x_2}{x_1+x_2}\right).
\end{equation*}
and $\ell(x_1,x_2) = (x_1,x_2,1)$, then the boundary points of Teichm\"uller curves in the normalization
of \eqref{eq:35} and $x_3 = 1$ are points $y = (x_1,x_2)  \in \cY$ with $[\IQ(y): \IQ] \leq 3$ 
which satisfy \eqref{eq:relation} for some  $b = (b_1,b_2,b_3) \in (\IZ\ssm\{0\})^3$.
\end{prop}
\par
\begin{proof} By definition of a stable curve no two poles coincide. Moreover on the degenerate
  fibers of a \Teichmuller curve zeros of the generating one-form are disjoint from the poles (see
  e.g.\ \cite{moelPCMI}). This implies that the boundary point has coordinates in $\cY$.

  The $x_i$ lie in $F^\sigma$ by Corollary~\ref{cor:xi_dual}, and the degree bound follows.
\par
Since $y_i = -x_i$ the cross-ratios can be simplified to 
\begin{equation}
  \label{eq:CR}
  R_{ij} = \left(\frac{x_i+x_j}{x_i - x_j}\right)^2.  
\end{equation}
so that $R^{-2} = (R_{23},R_{13},R_{12})$. Since the $x_i$ are real, the root
of unity on the right of \eqref{eq:CREq} is $\pm 1$.  Possibly multiplying the $b_i$ by $2$, we can
take it to be $1$.
\par
By the argument preceding the proposition, the tuples  $(1/x_1^\sigma, 1/x_2^\sigma, 1/x_3^\sigma)$
and $(s_1,s_2,s_3)$ are proportional. Let $a = (a_1,a_2,a_3)$ be the cross-ratio
exponents, as defined along with \eqref{eq:CREq}. 
Since the $a_i$ are integers, \eqref{eq:defCREXP} is
restated as $\sum_{i=1}^3 a_i x_i =0$. If we define $b_i = 2a_i$, then both conditions
in \eqref{eq:relation} hold as a consequence of \eqref{eq:CREq}
and \eqref{eq:defCREXP}.
\end{proof}
\par
We next check that $\cY$, $R$ and $\ell$ match the hypothesis of (i) and (ii) of
Theorem~\ref{thm:heightmain}. 
The map $R$ is in fact injective and the closure of its image $\cS = \overline{R(Y)} \subset \IG_m^3$ 
is defined by the cubic
\begin{equation*}
  r_1 r_2 r_3 + r_1 + r_2 + r_3 \in \IQ[r_1,r_2,r_3]. 
\end{equation*}
\par
The following lemma is the first instance of the problem  discussed
in Section~\ref{sec:torusalgo}.  In this case we give the complete
discussion without computer assistance.
\par
We recall that $\cS\ssm \oa{\cS}$ is the union of all positive dimensional
cosets that are contained in $\cS$. 
\par
\begin{lemma}
\label{lem:So}
The complement $\mathcal{S}\ssm \oa{\mathcal{S}}$
is the union of the 6 lines obtained by permuting the coordinates of
\begin{equation*}
  \{(1,-1)\}\times\IG_m. 
\end{equation*}
\end{lemma}
\begin{proof}
  Let $(r_1,r_2,r_3)\in \mathcal S$ be in a positive dimensional coset
  contained entirely in $\mathcal S$. Then there exist
  $e_1,e_2,e_3\in\IZ$, not all zero, with 
  \begin{equation}
    \label{eq:coset}
    r_1 r_2 r_3 t^{e_1+e_2+e_3}+ r_1t^{e_1}+r_2 t^{e_2}+r_3 t^{e_3} = 0
  \end{equation}
  for all $t \in \IG_m$.  There cannot be any exponent in this equation which appears
  by itself, so without loss of generality (permuting the coordinates if necessary) $e_1+e_2+e_3 = e_1$ and $e_2 = e_3$, from which we obtain
  $e_2 = e_3 = 0$.  Then \eqref{eq:coset} implies $r_1 r_2 r_3 + r_1 = 0$ and $r_2 + r_3 = 0$, so
  $r_2 = \pm 1$ and $r_3 = \mp 1$, leading to the two lines inside $\mathcal S$:
\begin{equation*}
  \{(t,1,-1);\,\, t\in \IG_m\}\quad\text{and}\quad
\{(t,-1,1);\,\, t\in \IG_m\} \qedhere
\end{equation*}
\end{proof}
\par
The next lemma checks the hypothesis on the image of $\cC_b$, needed
for Theorem~\ref{thm:heightmain}~(ii).  Recall that $\cC_b\subset\mathcal{Y}$ is the curve cut out
by the linear equation $\sum b_i \ell_i = 0$.  In our case, it this is more explicitly $b_1 x_1 +
b_2 x_2 + b_3 = 0$.
\par
\begin{lemma}
  Suppose $b=(b_1,b_2,b_3) \in (\IC^*)^3$. Then $\overline
  {R(\mathcal{C}_b)}$ is not contained in the translate of a 
 proper algebraic subgroup of $\IG_m^3$.
\end{lemma}
\begin{proof}
  If the contrary holds then there exists
  $(a_1,a_2,a_3)\in\IZ^3\ssm\{0\}$ with 
  \begin{equation}
\label{eq:rat}
   \left(\frac{x_2-1}{x_2+1}\right)^{a_1}   
   \left(\frac{1-x_1}{1+x_1}\right)^{a_2}   
   \left(\frac{x_1-x_2}{x_1+x_2}\right)^{a_3},   
  \end{equation}
constant on $\mathcal{C}_b$.
We evaluate the order of this function at certain points of
$\mathcal{C}_b$ to derive a contradiction. Without loss of generality
we assume $a_1\ge 0$.

Suppose for the moment that $a_1>0$. The first factor of
(\ref{eq:rat}) has a pole when $x_2=-1$.
Therefore, $x_2=- 1$ must imply $x_1=\pm 1$ and we have
\begin{equation}
\label{eq:lin1}
  b_1 - b_2 + b_3=0\quad\text{or}\quad -b_1-b_2+b_3=0.
\end{equation}
Now (\ref{eq:rat}) vanishes at $x_2=1$, so $x_1=\pm 1$ as well. As
before we find
\begin{equation}
\label{eq:lin2}
  b_1+  b_2 + b_3=0\quad\text{or}\quad -b_1+b_2+b_3=0.
\end{equation}

We immediately observe that any pair  of linear equations, one coming
from (\ref{eq:lin1}) and the other from (\ref{eq:lin2}), is linearly
independent. Moreover, a common zero of any of these four systems
satisfies $b_1b_2b_3=0$. This contradicts our hypothesis. 

We must also treat the case $a_1=0$.  Then $a_2\not=0$, since
$(x_1-x_2)/(x_1+x_2)$ is non-constant. This time we consider the
points on $\mathcal{C}_b$ with $x_1=1$ and $x_1=-1$.
In either case we must have $x_2=1$ or $x_2=-1$. We find the same
linear equations as above, but paired up differently. 
Again, any solution of any pair must have a vanishing
coordinate.
\end{proof}
\par
\begin{proof}[{Proof of Theorem~\ref{thm:intromainfin}, case ${\omoduli[3]}(4)^\hyp$}]
 By Lemma~\ref{lem:So} the image of any $y \in \cY$ lies in $\oa{\mathcal{S}}$. All the
hypothesis of Theorem~\ref{thm:heightmain} are satisfied, so that the height of
any pair $(x_1,x_2)$ possibly arising from a cusp of an algebraically 
primitive \Teichmuller curve in this stratum is bounded. From the degree bound
in Proposition~\ref{prop:setupin4hyp} and Northcott's theorem, we deduce that the number of
such pairs is finite. By \cite[Proposition~13.10]{BaMo12},
there are only finitely many algebraically primitive \Teichmuller curves in the stratum $\omoduli[3](4)^\hyp$.
\end{proof}
\par

\subsection{The stratum $\omoduli[3](4)^\odd$}

Let now $X_\infty$ be a degenerate fiber of a family $f: \cX \to C$ over an 
algebraically primitive  \Teichmuller curve  generated by a 
Veech surface $(X,\omega)$ in $\omoduli[3](4)^\odd$. Let $\PP^1$ with
coordinate $z$ be the normalization of $X_\infty$. It turns out to be convenient to
normalize the zero to be at $z = \infty$ (as opposed to the previous section
where we had $z=0$), so that the stable form on the limit curve is
\begin{equation}
  \label{eq:norm4oddDIFF}
  \omega_\infty = \sum_{i=1}^3 \left(\frac{r_i}{z-x_i} - 
\frac{r_i}{z-y_i} \right)dz =
  \frac{C}{\prod_{i=1}^3(z-x_i)(z-y_i)} dz.
\end{equation}
We may suppose that $y_3 = -x_3$, leaving still
a global scalar multiplication as degree of freedom and further
down we will moreover let $x_3=1$, hence $y_3 = -1$.

\paragraph{\bf The surface $\cY$.} We parameterize stable forms by 
points $(x_1,y_1,x_2,y_2) \in \IA^4$. For the form defined by the fraction
on the right of \eqref{eq:norm4oddDIFF} to be stable, we need
$$ \frac{1}{(x_i - y_i)\prod_{i \neq j}(x_i - x_j)(x_i - y_j)} 
= -\frac{1}{(y_i - x_i)\prod_{i \neq j}(y_i - x_j)(y_i - y_j)} 
$$
for $i=1,2,3$. Checking it for all but one $i$ is sufficient by the
residue theorem. This conditions are equivalent to the vanishing of the two
polynomials 
\begin{alignat}1
\label{eq:stabil1}
  P_1= (y_1-x_2)(y_1-y_2)(y_1^2-1) - (x_1-x_2)(x_1-y_2)(x_1^2-1)  \\
\label{eq:stabil2}
P_2 =(y_2-x_1)(y_2-y_1)(y_2^2-1) - (x_2-x_1)(x_2-y_1)(x_2^2-1).
\end{alignat}
\par
The stability conditions also contain the hyperelliptic locus $x_i = -y_i$ dealt
with above and we want to get rid of this locus.\footnote{Starting here, several lemmas are based heavily on
  computations
made using {\tt sage}. 
The code required in the proof of Lemmas \ref{lem:cYgeomirred},
\ref{lem:SoODD}, \ref{lem:irred}, \ref{lem:cong1}, \ref{lem:cong2},
\ref{lem:projirred}, \ref{lem:3264} 
can be found in {\tt g3fin\_ch6.sage}.}
\par
\begin{lemma}
\label{lem:cYgeomirred}
The polynomials
\begin{eqnarray}
\label{eq:defineX}
\left.  \begin{array}{ll}
f_1&=    x_{1} x_{2} + y_{1} x_{2} -  x_{2}^{2} + x_{1} y_{2} + y_{1} y_{2} - 
y_{2}^{2} + 2 \\
f_2&=x_{1}^{2} + y_{1}^{2} -  x_{2}^{2} -  y_{2}^{2} 
  \end{array}
\right\}
\end{eqnarray}
generate a prime ideal $I$ of $\IQ[x_1,y_1,x_2,y_2]$.
The set of common zeros of $I$ is
 a  geometrically irreducible affine variety
$\overline{\cY} \subset \IA^4$ of dimension 2 on which 
$P_1$ and $P_2$ vanish.
Moreover, a point at which $P_1$ and $P_2$ vanish but at which
$(x_1-y_1)(y_1-y_2)(x_1-y_2)(x_2-y_2)(x_1+y_1)$ does not, lies on
$\overline{\cY}$.
\end{lemma}
\par
\begin{proof}
Assisted by a computer algebra system
one verifies that $f_1$ and $f_2$ generate
a prime ideal of $\IQ[x_1,y_1,x_2,y_2]$, i.e.\ that $\overline{\cY}$ is
irreducible over $\IQ$.
In a similar manner we check that $P_{1,2}\in I$
and, using
the Jacobian criterion,
we find that $$(x_1,y_1,x_2,y_2)=(1,-1,1,-1)$$ is
a  smooth point of $\overline{\cY}$. 
So it lies on precisely one geometric  component 
of $\overline\cY$. This component is defined over $\IQ$
as the said point is rational.
Therefore, $\overline \cY$ is geometrically irreducible. 
\end{proof}
\par
We define 
\be 
\cY = \overline{\cY} \ssm \left\{ \prod_{i,j=1; i \neq j}^3 (x_i - x_j)(y_i - y_j) \prod _{i,j=1}^3 (x_i - y_j) =0 \right\}. 
\ee
As a consequence of the definition of a stable form, the forms $\omega_\infty$ normalized as 
in \eqref{eq:norm4oddDIFF} have coordinates in $\cY$.

\paragraph{\bf Using the Harder-Narasimhan filtration and the function $\ell$.}

Let $\omega^\sigma$ be one of the two Galois conjugate eigenform, 
generating the eigenform bundle of second largest degree.
From Proposition~\ref{prop:HNFisEigsplit}
we deduce 
\begin{equation}
  \label{eq:oddomegasigmaDIFF}
  \omega^\sigma_\infty = \sum_{i=1}^3 \left(\frac{r_i^\sigma}{z-x_i} - 
\frac{r_i^\sigma}{z-y_i} \right)dz = 
  \frac{P_3(z)}{\prod_{i=1}^3(z-x_i)(z-y_i)} dz,
\end{equation}
where $P_3(z)$ is some polynomial of degree (less or equal to) three. This condition is
equivalent to
$$ \sum_{i=1}^3 r_i^\sigma \left(x_i - y_i \right) = 0.$$
\par
The same argument as above gives that the tuple of
$\left(x_i - y_i \right)^{- 1}$ is up to scale a 
dual basis to $(r_1^\tau,r_2^\tau,r_3^\tau)$. With the normalization
$x_3=1$ and $y_3=-1$, this implies that
the set $\left(x_i - y_i \right)$, $i=1,2,3$, lies in the Galois closure
of the trace field $F$,  in particular is its real.
\par
In this stratum the tuple $(b_1,b_2,b_3)$ is proportional to the
tuple of cross-ratio exponents if and only if
$$ \sum_{i=1}^3 b_i \,\frac1{x_i - y_i } = 0.$$
\par
Here we let $R = (R_1,R_2,R_3)$ be the three cross-ratios as defined in \eqref{CR6}.
For a given boundary point of an algebraically primitive \Teichmuller curve we take $c = {\rm ord}(\zeta_E)$, the multiplicative
order of the root of unity. If $a = (a_1,a_2,a_3)$ is the tuple of cross-ratio
exponents, we let $b = (b_1,b_2,b_3) = ca$. This discussion is then summarized
in the following statement.
\par
\begin{prop} \label{prop:degbound4odd}
With $\ell_i = 1/(x_i - y_i)$, boundary points on algebraically primitive \Teichmuller curve in the 
stratum $\omoduli[3](4)^\odd$ correspond, in the normalization of \eqref{eq:norm4oddDIFF}, 
to points in $y \in \cY$ with $[\IQ(y): \IQ] \leq 72$ that satisfy \eqref{eq:relation} 
for some  $b = (b_1,b_2,b_3) \in (\IZ\ssm\{0\})^3$.
\end{prop}
\par
\begin{proof} After the preceding discussion, we only need to justify
the field degree bound. For a given triple of $r_i \in F$, with the normalization
$x_3=1$ and $y_3=-1$ the numerator of \eqref{eq:norm4oddDIFF} gives $4$ equations for the unknowns
$x_1,x_2,y_2,y_3$ of total degree $1,2,3$ and $4$ respectively. Period coordinates imply that  
a stable form with a $4$-fold zero is locally uniquely determined by its residues.
Consequently, the set of solution to this system of equations is finite and
each solution is of degree at worst $24$ over $F$. 
\end{proof}
\par

\paragraph{\bf Determining $\oa{\mathcal{S}}$ for  $\cS = \overline{R(\cY)}$.}

We next start checking that $\cY$, $R$ and $\ell$ match the hypothesis (i) and (ii) of 
Theorem~\ref{thm:heightmain}. 
The map $R$ is in fact two-to-one on $\cY$ \cite[Corollary~8.4]{BaMo12}. With
the help of a computer algebra system we find that the closure of its 
image $\cS = \overline{R(\cY)} \subset \IG_m^3$ is cut out by the equation
{\tiny 
 \begin{align*}
   h &= X^{6} Y^{6} Z^{2} - 4 X^{6} Y^{5} Z^{3} - 4 X^{5} Y^{6} Z^{3} + 6
   X^{6} Y^{4} Z^{4} - 124 X^{5} Y^{5} Z^{4} + 6 X^{4} Y^{6} Z^{4} - 4
   X^{6} Y^{3} Z^{5} - 124 X^{5} Y^{4} Z^{5}   \\ &- 124 X^{4} Y^{5} Z^{5}
   - 4 X^{3} Y^{6} Z^{5} + X^{6} Y^{2} Z^{6} - 4 X^{5} Y^{3} Z^{6} + 6
   X^{4} Y^{4} Z^{6} - 4 X^{3} Y^{5} Z^{6} + X^{2} Y^{6} Z^{6} + 336
   X^{5} Y^{5} Z^{3} \\ &+ 864 X^{5} Y^{4} Z^{4} + 864 X^{4} Y^{5} Z^{4} +
   336 X^{5} Y^{3} Z^{5} + 864 X^{4} Y^{4} Z^{5} + 336 X^{3} Y^{5}
   Z^{5} - 2 X^{6} Y^{5} Z - 2 X^{5} Y^{6} Z \\&+ 8 X^{6} Y^{4} Z^{2} -
   246 X^{5} Y^{5} Z^{2} + 8 X^{4} Y^{6} Z^{2} - 12 X^{6} Y^{3} Z^{3}
   - 1672 X^{5} Y^{4} Z^{3} - 1672 X^{4} Y^{5} Z^{3} - 12 X^{3} Y^{6}
   Z^{3} \\&+ 8 X^{6} Y^{2} Z^{4} - 1672 X^{5} Y^{3} Z^{4} - 4800 X^{4}
   Y^{4} Z^{4} - 1672 X^{3} Y^{5} Z^{4} + 8 X^{2} Y^{6} Z^{4} - 2
   X^{6} Y Z^{5} - 246 X^{5} Y^{2} Z^{5} \\& - 1672 X^{4} Y^{3} Z^{5} -
   1672 X^{3} Y^{4} Z^{5} - 246 X^{2} Y^{5} Z^{5} - 2 X Y^{6} Z^{5} -
   2 X^{5} Y Z^{6} + 8 X^{4} Y^{2} Z^{6} - 12 X^{3} Y^{3} Z^{6} + 8
   X^{2} Y^{4} Z^{6} \\& - 2 X Y^{5} Z^{6} + 72 X^{5} Y^{5} Z + 1088 X^{5}
   Y^{4} Z^{2} + 1088 X^{4} Y^{5} Z^{2} + 2800 X^{5} Y^{3} Z^{3} +
   8288 X^{4} Y^{4} Z^{3} + 2800 X^{3} Y^{5} Z^{3} \\&+ 1088 X^{5} Y^{2}
   Z^{4} + 8288 X^{4} Y^{3} Z^{4} + 8288 X^{3} Y^{4} Z^{4} + 1088
   X^{2} Y^{5} Z^{4} + 72 X^{5} Y Z^{5} + 1088 X^{4} Y^{2} Z^{5} +
   2800 X^{3} Y^{3} Z^{5} \\& + 1088 X^{2} Y^{4} Z^{5} + 72 X Y^{5} Z^{5}
   + X^{6} Y^{4} - 2 X^{5} Y^{5} + X^{4} Y^{6} - 4 X^{6} Y^{3} Z - 246
   X^{5} Y^{4} Z - 246 X^{4} Y^{5} Z - 4 X^{3} Y^{6} Z \\&+ 6 X^{6} Y^{2}
   Z^{2} - 1672 X^{5} Y^{3} Z^{2} - 5229 X^{4} Y^{4} Z^{2} - 1672
   X^{3} Y^{5} Z^{2} + 6 X^{2} Y^{6} Z^{2} - 4 X^{6} Y Z^{3} - 1672
   X^{5} Y^{2} Z^{3} \\&- 13532 X^{4} Y^{3} Z^{3} - 13532 X^{3} Y^{4}
   Z^{3} - 1672 X^{2} Y^{5} Z^{3} - 4 X Y^{6} Z^{3} + X^{6} Z^{4} -
   246 X^{5} Y Z^{4} - 5229 X^{4} Y^{2} Z^{4} \\&- 13532 X^{3} Y^{3}
   Z^{4} - 5229 X^{2} Y^{4} Z^{4} - 246 X Y^{5} Z^{4} + Y^{6} Z^{4} -
   2 X^{5} Z^{5} - 246 X^{4} Y Z^{5} - 1672 X^{3} Y^{2} Z^{5} \\&- 1672 X^{2}
Y^{3} Z^{5} - 246 X Y^{4} Z^{5} - 2 Y^{5} Z^{5} + X^{4} Z^{6} - 4 X^{3}
Y Z^{6} + 6 X^{2} Y^{2} Z^{6} - 4 X Y^{3} Z^{6} + Y^{4} Z^{6} + 336
X^{5} Y^{3} Z \\& + 1088 X^{4} Y^{4} Z + 336 X^{3} Y^{5} Z + 864 X^{5} Y^{2}
Z^{2} + 8288 X^{4} Y^{3} Z^{2} + 8288 X^{3} Y^{4} Z^{2} + 864 X^{2}
Y^{5} Z^{2} + 336 X^{5} Y Z^{3} \\&+ 8288 X^{4} Y^{2} Z^{3} + 21888 X^{3}
Y^{3} Z^{3} + 8288 X^{2} Y^{4} Z^{3} + 336 X Y^{5} Z^{3} + 1088 X^{4}
Y Z^{4}
 + 8288 X^{3} Y^{2} Z^{4} + 8288 X^{2} Y^{3} Z^{4} \\&+ 1088 X Y^{4}
Z^{4} + 336 X^{3} Y Z^{5} + 864 X^{2} Y^{2} Z^{5} + 336 X Y^{3} Z^{5} -
4 X^{5} Y^{3} + 8 X^{4} Y^{4} - 4 X^{3} Y^{5} - 124 X^{5} Y^{2} Z \\&- 1672
X^{4} Y^{3} Z - 1672 X^{3} Y^{4} Z - 124 X^{2} Y^{5} Z - 124 X^{5} Y
Z^{2} - 4800 X^{4} Y^{2} Z^{2} - 13532 X^{3} Y^{3} Z^{2} - 4800 X^{2}
Y^{4} Z^{2} \\&- 124 X Y^{5} Z^{2} - 4 X^{5} Z^{3} - 1672 X^{4} Y Z^{3} -
13532 X^{3} Y^{2} Z^{3} - 13532 X^{2} Y^{3} Z^{3} - 1672 X Y^{4} Z^{3} -
4 Y^{5} Z^{3} + 8 X^{4} Z^{4} \\&- 1672 X^{3} Y Z^{4} - 4800 X^{2} Y^{2}
Z^{4} - 1672 X Y^{3} Z^{4} + 8 Y^{4} Z^{4} - 4 X^{3} Z^{5} - 124 X^{2} Y
Z^{5} - 124 X Y^{2} Z^{5} - 4 Y^{3} Z^{5} \\&+ 864 X^{4} Y^{2} Z + 2800
X^{3} Y^{3} Z + 864 X^{2} Y^{4} Z + 864 X^{4} Y Z^{2} + 8288 X^{3} Y^{2}
Z^{2} + 8288 X^{2} Y^{3} Z^{2} + 864 X Y^{4} Z^{2} \\&+ 2800 X^{3} Y Z^{3}
+ 8288 X^{2} Y^{2} Z^{3} + 2800 X Y^{3} Z^{3} + 864 X^{2} Y Z^{4} + 864
X Y^{2} Z^{4} + 6 X^{4} Y^{2} - 12 X^{3} Y^{3} + 6 X^{2} Y^{4} \\& - 124
X^{4} Y Z - 1672 X^{3} Y^{2} Z - 1672 X^{2} Y^{3} Z - 124 X Y^{4} Z +
6X^{4} Z^{2} - 1672 X^{3} Y Z^{2} - 5229 X^{2} Y^{2} Z^{2} \\&- 1672 X Y^{3}
Z^{2} + 6 Y^{4} Z^{2} - 12 X^{3} Z^{3} - 1672 X^{2} Y Z^{3} - 1672 X
Y^{2} Z^{3} - 12 Y^{3} Z^{3} + 6 X^{2} Z^{4} - 124 X Y Z^{4} + 6 Y^{2}
Z^{4} \\&+ 336 X^{3} Y Z + 1088 X^{2} Y^{2} Z + 336 X Y^{3} Z + 1088 X^{2}
Y Z^{2} + 1088 X Y^{2} Z^{2} + 336 X Y Z^{3} - 4 X^{3} Y + 8 X^{2} Y^{2}
- 4 X Y^{3} \\&- 4 X^{3} Z - 246 X^{2} Y Z - 246 X Y^{2} Z - 4 Y^{3} Z + 8
X^{2} Z^{2} - 246 X Y Z^{2} + 8 Y^{2} Z^{2} - 4 X Z^{3} - 4 Y Z^{3} + 72
X Y Z + X^{2} \\&- 2 X Y + Y^{2} - 2 X Z - 2 Y Z + Z^{2}.
 \end{align*}
}
The polynomial $h$ has total degree 14 and 199 non-zero terms. It is
symmetric under permutation of coordinates. 
\par
\begin{lemma}
\label{lem:SoODD}
For the vanishing locus $\cS$ of $h$, we have
\begin{equation*}
  \oa{\cS} = \cS\ssm \{(t,1,1),(1,t,1),(1,1,t);\,\, t\in \IG_m \}.
\end{equation*}
\end{lemma}
\par
\begin{proof}
We observe first that $\cS$ itself is no coset. Indeed, $h(1,1,1)=0$ and
a computation shows that $h$ is irreducible over $\IQ$. 
But a coset that is irreducible over the rationals and contains the
unit element is an absolutely irreducible algebraic subgroup of
$\IG_m^3$. If this were  true for the zero set of $h$, then 
$h$ would consist of $2$ monomials. This is obviously not the case.
\par
Any anomalous subvariety of $\cS$ must 
be a $1$-dimensional coset contained completely in $\cS$; 
we recall that anomalous subvarieties were defined on page
\pageref{def:anomalous}. 
It is thus the image of
\begin{equation*}
t\mapsto  (u_1t^{e_1},u_2t^{e_2},u_3 t^{e_3})
\end{equation*} 
where $u_1,u_2,u_3\in \IG_m$ and $E=(e_1,e_2,e_3)\in\IZ^3\ssm\{0\}$ are fixed. 
\par
This is the situation where the algorithm of Section~\ref{sec:torusalgo}
applies. The exponents here are unconstrained, $M$ is the identity matrix.
However, since we are interested in torus translates with 
$u_i \neq 0$ we may discard immediately subspaces defined by a matrix
$M'$ where a part
of the partition induced by projecting the support of $h$ onto $M'$ consists of a single element.
\par
Therefore, any $\lambda$ in $\supp{h}$ must have
a \emph{friend}, that is there has to exist $\lambda'\in \supp{h}$ with 
$\lambda\not=\lambda'$ and
\begin{equation*}
  \langle \lambda-\lambda',E\rangle = 0.
\end{equation*}
For convenience, we fix $\lambda_1 = (6,6,2)\in \supp{h}$ and the possible
$E$ contained in the subspaces $\langle\lambda_1-\lambda'_1 \rangle^\perp$
for $\lambda'_1 \in \supp{h}\ssm\{\lambda_1\}$.
\par
{\bf Tier 1:} The list of possible torus translates is reduced to
the consideration of one-dimensional subspaces by the following 
fact about $h$ that is proven by a (computer-assisted) check of all
possibilities.
\par
For any $\lambda'_1 \in \supp{h}\ssm\{\lambda_1\}$ there exists
$\lambda_2\in \supp{h}$ such that
\begin{equation*}
  \lambda_1-\lambda'_1\quad\text{and}\quad
\lambda_2-\lambda'_2
\end{equation*}
are linearly independent for all $\lambda'_2 \in \supp{h}\ssm\{\lambda_2\}$. 
In other words, given a  potential friend $\lambda'_1$ of $\lambda_1$, there is
some $\lambda_2$ in the support of $h$ whose potential friends set-up
a system of linear equations
\begin{equation*}
  \langle \lambda_1-\lambda_1',E\rangle = 
  \langle \lambda_2-\lambda_2',E\rangle = 0
\end{equation*}
that has a unique solution $E$ up-to scalar multiplication.

Any $E$ coming from an anomalous curve in $X$ must be a solution of
one of these systems. In particular,  only finitely many $E$ are
possible. But we can use {\tt sage} to create a list of possibilities
for $E$. In total there are 8796 and we will not reproduce them here.
This is well beyond the $3$ possibilities
that appear in the conclusion of this lemma. In the next tier we will 
reduce the number of possibilities
dramatically. 

{\bf Tier 2:} Given one  of the 8796 candidates $E$ from tier 1 we use
{\tt sage} to check that any element in the support of $f$ has a
friend with respect to $E$. As $f$ has 199 non-zero terms this seems
quite a strong restriction. However, 51 of candidates pass this
test. They are
\begin{alignat*}1
\Bigl\{&\left(1, 2, 1\right), \left(4, -1, -1\right), \left(1, 8,-1\right), \left(1, -8, 1\right), \left(1, -1, 8\right), \left(1, 1,
6\right), \\ 
&\left(2, 1, 1\right), 
\left(0, 1, 0\right), \left(8, -1,
-1\right), \left(6, 1, 1\right), \left(1, -6, 1\right), \left(1, 0,
0\right), \left(1, 1, 2\right), \\ &
\left(6, 1, -1\right), \left(6, -1,
-1\right), \left(1, 1, 4\right), \left(1, -6, -1\right), \left(1, -1,
4\right), \left(1, 2, -1\right), \\ & \left(1, -8, -1\right), 
\left(1, -1,
-6\right), \left(1, -1, -8\right), \left(1, -1, -4\right), \left(1, 8,
1\right), \left(4, -1, 1\right), \\ & \left(1, -4, -1\right),  \left(1, -1,
2\right), \left(1, 4, -1\right), \left(1, 1, -4\right), \left(1, -2,
-1\right), \left(2, -1, 1\right), \\ & \left(8, 1, 1\right),  \left(2, -1,
-1\right), \left(8, -1, 1\right), \left(1, 1, -6\right), \left(1, 1,
-2\right), \left(1, 6, -1\right), \left(4, 1, 1\right), \\ & \left(1, 1,
-8\right), \left(1, 4, 1\right), \left(1, -1, 6\right),  \left(0, 0,
1\right), \left(1, 6, 1\right), \left(2, 1, -1\right), \left(4, 1,
-1\right), \\ & \left(1, -4, 1\right), \left(1, -2, 1\right), \left(6, -1,
1\right), \left(8, 1, -1\right), \left(1, 1, 8\right), \left(1, -1,
-2\right)\Bigr \}.
\end{alignat*}
\par
{\bf Tier 3:} In the final tier we will reduce the $51$ 
candidates to $3$ using the second part of the algorithm
of Section~\ref{sec:torusalgo}. 
%
A candidate $(e_1,e_2,e_3)$ leads to a coset contained in the zero set
of $h$ only if the following property is true. 
The polynomial
$h(u_1t^{e_1},u_2t^{e_2},u_3t^{e_3})$,
where $u_1,u_2,u_3$ are  independents and the coefficient ring is
$\IC[t]$, 
vanishes at some complex point $\IG_m^3$.
This is a strong restriction because each power of $t$ yields a
polynomial in complex coefficients and all of these need to vanish at
the same point. 
In a matter of seconds, {\tt sage} eliminates 48 of the 51  candidates
above; indeed, the said polynomials yield the unit ideal in $\IC[u_1^{\pm 1},u_2^{\pm 1},u_3^{\pm 1}]$.

 The $3$ remaining  candidates are
\begin{equation*}
   (1,0,0), (0,1,0), (0,0,1).
\end{equation*}
Here, {\tt sage} tells us that we must have
\begin{equation*}
  u_2=u_3=1, \quad u_1=u_3=1, \quad\text{or}\quad u_1=u_2=1
\end{equation*}
respectively. Not only are these three curves anomalous but they are
even torsion anomalous. 
\par
Conversely, we easily find
\begin{equation*}
  h(t,1,1)=h(1,t,1)=h(1,1,t)=0
\end{equation*}
and so the candidates are indeed torsion anomalous curves. 
\end{proof}
\paragraph{The curve $\cC_c$.}

We now study the locus  cut out by the equation $\sum_i b_i
\ell_i =0$. 
Recall that the $b_i$ are non-zero. 
We 
divide by $b_3$ and let $c_1$ and $c_2$ be new independent variables that
take the role of $c_i = b_i/b_3$. Chasing denominators
and if  $c = (\spe{c_1},\spe{c_2})=(b_1/b_3,b_2/b_3)$
we write $\cC_c$ 
for the algebraic subset of $\IA^4$  cut out in $\cY$ by  
\begin{alignat}1
\label{eq:definef}
2  \spe{c_1} (x_2-y_2)  + 2 \spe{c_2} (x_1-y_1) +  (x_1-y_1)  (x_2-y_2)\in \IQ[x_1,y_1,x_2,y_2].
\end{alignat}
Indeed, it will be useful to consider $c_1$ and $c_2$ as independent
variables and let $\spe{c_1}$ and $\spe{c_2}$ denote their
specializations to the coordinates of a given $c\in\IQ^2$. 
We set
\begin{equation*}
f=  2  {c_1} (x_2-y_2)  + 2 {c_2} (x_1-y_1) +  (x_1-y_1)  (x_2-y_2)\in \IQ[c_1,c_2,x_1,y_1,x_2,y_2].
\end{equation*}

Although $\cu$ is defined by polynomials in rational coefficients it is
 sometimes useful to  think of it as an algebraic set over $\IC$. 
For any scheme $\cC$ over $\spec{\IQ}$ we write $\cC\otimes \IC$ for
its base change to $\spec{\IC}$. 

It is not difficult to show  that $\cu$ is an algebraic curve. We
state this result in the next lemma, which is proved further down and which justifies
the title of this subsection. 

\begin{lemma}
\label{lem:xiyiprops}
Say $c\in\IQ^2$ has non-zero coordinates. Then $\cu\not=\emptyset$. 
Let $\mathcal C$ be an
irreducible component of $\cu\otimes\IC$.
\begin{enumerate}
\item [(i)] Say $i\in\{1,2\}$.
The functions $x_i\pm 1$, $y_i\pm 1$,  $x_1-y_2$, $x_2-y_1$, $x_2-x_1$, and  $y_2-y_1$
are non-zero elements of the function field
of $\cC$. 
\item[(ii)] The component $\cC$ is a curve. 
  \end{enumerate}
\end{lemma}
\par

Much of this section deals with the more intricate question  of the irreducibility
of $\cC_c$. In fact, we believe that $\cC_c$ is irreducible for all
$c\in(\IQ\ssm\{0\})^2$. We are only
 able to prove a weaker statement which is ultimately sufficient for
 our needs.

The irreducibility of $\cC_c$ is merely an ingredient 
in the study of multiplicative relations
among  the $3$ cross-ratio maps
\begin{equation}
\label{eq:defR123}
  \begin{aligned}
R_1 &= R_{[23]} =  \frac{(x_2-1)(y_2+1)}{(x_2+1)(y_2-1)},  \\
R_2 &= R_{[13]} =  \frac{(x_1-1)(y_1+1)}{(x_1+1)(y_1-1)},  \\
R_3 &= R_{[12]} =  \frac{(x_2-x_1)(y_1-y_2)}{(x_1-y_2)(x_2-y_1)}
  \end{aligned}
\end{equation}
which are non-zero rational maps on irreducible components of $\cu\otimes\IC$ by the previous lemma.  
The following proposition is the main technical 
result of this section.  It will be crucial in verifying the hypothesis
of Theorem \ref{thm:heightmain} in order to obtain a height bound. 
\begin{prop} 
\label{prop:R123indep}
There is a finite subset $\Sigma \subset \IQ^2$ with the following
property. 
Suppose $c\in \IQ^2$ has non-zero coordinates and that $\mathcal{C}$
is an irreducible component of $\cC_c\otimes\IC$.
\begin{enumerate}
\item [(i)]
If $c\not\in \Sigma$ then $R_1,R_2,R_3$ are multiplicatively
independent on $\mathcal{C}$. 
\item[(ii)] If $c\in\Sigma$ and if $(b_1,b_2,b_3)\in\IZ^3\ssm\{0\}$ with 
$R_1^{b_1}R_2^{b_2}R_3^{b_3}$ a constant, then 
$(b_1,b_2,b_3)\not\in (\spe{c_1},\spe{c_2},1)\IQ$. 
\end{enumerate}
\end{prop}

The proofs of both the lemma and the proposition
 will require some preparation as well as 
computational support from
 {\tt sage}.

It will prove convenient to introduce
\begin{equation}
\label{eq:definet}
  t = \frac{y_1+1}{x_1+1}.
\end{equation}
Using this new variable and the  residue condition \eqref{eq:stabil1}
we have
  $\frac{y_1-y_2}{x_1-y_2} = 
\frac{x_2-x_1}{x_2-y_1}\frac{x_1^2-1}{y_1^2-1}=
\frac{x_2-x_1}{x_2-y_1}
\frac{x_1-1}{y_1-1}t^{-1}$.
In terms of $t$ we have
\begin{equation}
\label{eq:defR123t}
  \begin{aligned}
R_1 &=   \frac{(x_2-1)(y_2+1)}{(x_2+1)(y_2-1)},  \\
R_2 &=  \frac{x_1-1}{y_1-1} t, \\
R_3 &= 
\left(\frac{x_2-x_1}{x_2-y_1}\right)^2
\frac{x_1-1}{y_1-1}t^{-1}.
  \end{aligned}
\end{equation}  

For brevity we set $B=\IQ[c_1^{\pm 1},c_2^{\pm 1},x_1,y_1,x_2,y_2]$. 
If $c\in (\IQ\ssm\{0\})^2$   we specialize an element
$g\in B$ to $\spe{g}\in \IQ[x_1,y_1,x_2,y_2]$
by substituting $c_1,c_2$ by the coordinates of $c$.

\par
\begin{lemma}
\label{lem:irred}
Let $c\in\IQ^2$ satisfy
\begin{equation}
\label{eq:condc}
\spe{c_1}\spe{c_2}(\spe{c_1}+\spe{c_2}+1)(-\spe{c_1}+\spe{c_2}+1)(\spe{c_1}-\spe{c_2}+1)(-\spe{c_1}-\spe{c_2}+1) \not= 0
\end{equation}
and suppose that $\cu$ is non-empty.
Then $\cu\otimes\IC$ is smooth and irreducible.
In particular, $\cu$ is irreducible. 
\end{lemma}
\begin{proof}
We may consider $\IA^4$ as a Zariski open subset of $\IP^4$ using the
open immersion $(x_1,y_1,x_2,y_2)\mapsto [x_1:y_1:x_2:y_2:1]$. 
Recall that $\overline{\cY}\subset\IA^4$ is geometrically irreducible,
cf. Lemma \ref{lem:cYgeomirred}.
We let $\overline{\overline{\cY}}$ denote the Zariski closure of $\overline{\cY}$ in
$\IP^4$ and consider it as an irreducible projective
variety over $\spec{\IC}$. 
\par
Homogenizing the polynomial $\spe{f}$ given in (\ref{eq:definef}) yields a 
 hypersurface $\cH$ of $\IP^4$. 
A direct calculation by hand and using the fact that $c$ has non-zero coordinates 
shows that $f$ is absolutely irreducible.
Thus $\cH$ is geometrically irreducible. 
\par
Assisted by computer algebra 
one can show that (\ref{eq:condc}) is enough to
ensure that $\cu\otimes\IC$ is smooth. 
This is done by first homogenizing the defining equations (\ref{eq:defineX})
and ${f}$ and treating $c_1$ and $c_2$ as independent varieties. 
The Jacobian criterion in connection with elimination 
restricts the possibilities for $c$ in presence
of a non-smooth point. For 
$c$  satisfying (\ref{eq:condc}) the only restriction is  the vanishing
to a certain
degree $12$  polynomial $q$ in integer coefficients. The
homogenization
of this polynomial is 
\begin{alignat}1
\label{eq:poly12}
  &c_{1}^{12} - 3 c_{1}^{10} c_{2}^{2} + 6 c_{1}^{8} c_{2}^{4} - 7
c_{1}^{6} c_{2}^{6} + 6 c_{1}^{4} c_{2}^{8} - 3 c_{1}^{2} c_{2}^{10} +
c_{2}^{12} - 3 c_{1}^{10} c_{3}^{2} - 51 c_{1}^{8} c_{2}^{2} c_{3}^{2} \\
\nonumber
& +
78 c_{1}^{6} c_{2}^{4} c_{3}^{2} + 78 c_{1}^{4} c_{2}^{6} c_{3}^{2} - 51
c_{1}^{2} c_{2}^{8} c_{3}^{2} - 3 c_{2}^{10} c_{3}^{2} + 6 c_{1}^{8}
c_{3}^{4} + 78 c_{1}^{6} c_{2}^{2} c_{3}^{4} + 414 c_{1}^{4} c_{2}^{4}
c_{3}^{4} \\
\nonumber
&+ 78 c_{1}^{2} c_{2}^{6} c_{3}^{4} + 6 c_{2}^{8} c_{3}^{4} - 7
c_{1}^{6} c_{3}^{6} + 78 c_{1}^{4} c_{2}^{2} c_{3}^{6} + 78 c_{1}^{2}
c_{2}^{4} c_{3}^{6} - 7 c_{2}^{6} c_{3}^{6} + 6 c_{1}^{4} c_{3}^{8} - 51
c_{1}^{2} c_{2}^{2} c_{3}^{8}  \\
\nonumber
&+ 6 c_{2}^{4} c_{3}^{8}- 3 c_{1}^{2}
c_{3}^{10} - 3 c_{2}^{2} c_{3}^{10} + c_{3}^{12}.
\end{alignat}
If $q$ were to have a rational zero, then its homogenization
would have a non-trivial integral zero with coprime
coefficients. Even a human can check that (\ref{eq:poly12}) has no
non-trivial zeros modulo $2$. Therefore, $q$ does not
vanish at the given $c$.\footnote{
Over  $\IQ(\sqrt{-3})$ the polynomial (\ref{eq:poly12}) factors into
\begin{equation*}  
 (c_{1}^{6} - 3 c_{1}^{2} c_{2}^{4} + c_{2}^{6} - 3 c_{1}^{4} c_{3}^{2} -
21 c_{1}^{2} c_{2}^{2} c_{3}^{2} - 3 c_{2}^{2} c_{3}^{4} + c_{3}^{6})
+	3  (c_{2} -  c_{3})  (c_{2} + c_{3})  (-
c_{1} + c_{2})  (c_{1} -  c_{3})  (c_{1} + c_{3})  (c_{1}
+ c_{2})\omega
\end{equation*}
times its conjugate where
 $\omega = (\sqrt{-3}+1)/2$. Using this presentation it
 is not hard to see that
(\ref{eq:poly12})
has only finitely many
 zeros  in $\IP^3(\IR)$.}
\par
We have actually  verified that all points of
 $\overline{\overline{\cY}}\cap \mathcal H$ are smooth for $c$ as in the
hypothesis.  
 Fulton and Hansen's Corollary 1  \cite{FultonHansen} implies that
 $\overline{\overline{\cY}}\cap \mathcal H$ is
 connected.
A variety that is connected and smooth must be irreducible. 
Since $\cu\otimes\IC$ is Zariski open in $\overline{\overline{\cY}}\cap
\mathcal H$ it is either empty or irreducible. In the latter case
$\cu$ is irreducible. 
\end{proof}
\par

\par
Let us define
\begin{equation*}
  J = IB + fB = (f_1,f_2,f)
\end{equation*}
which is an ideal of $B$.

We observe  that the ring automorphisms
\begin{equation}
\label{eq:sym1}
  (c_1,c_2,x_1,y_1,x_2,y_2)\mapsto 
  (c_2,c_1,x_2,y_2,x_1,y_1)
\end{equation}
and
\begin{equation}
  \label{eq:sym2}
   (c_1,c_2,x_1,y_1,x_2,y_2)\mapsto 
   (-c_1,-c_2,y_1,x_1,y_2,x_2)
\end{equation}
map the ideal $J=(f_1,f_2,f)\subset B= \IQ[c_1^{\pm 1},c_2^{\pm 1},x_1,y_1,x_2,y_2]$
to itself. 
Indeed, (\ref{eq:sym1}) maps $f_1$ to $f_1-f_2$, $f_2$ to $-f_2$, and 
fixes $f$. Moreover, (\ref{eq:sym2}) leaves $f_1, f_2,$ and $f$
invariant. 

The following lemmas subsumes elimination-theoretic
properties of $J$ that are necessary for our application.
The elements can be found by computer algebra assisted elimination
of variables. The computer assisted computations are usually done in the polynomial ring
$\IQ[c_1,c_2,x_1,y_1,x_2,y_2]$, where $c_{1,2}$ are not units. For
example, we find that $B/J$ is an integral domain since $J$ is a prime
ideal. 

\begin{lemma}
\label{lem:cong1}
For $i,j\in \{1,2\}$ 
  there exist  $f_{x_i y_j} \in J\cap \IQ[c_1,c_2,x_i,y_j]$
 that satisfy
      \begin{alignat*}2
        \begin{array}{rlcl rlcl}
      f_{x_i y_i} &=  x_i^6 &+& O(x_i^5), 
&      f_{x_i y_i} &=  y_i^6 &+& O(y_i^5), \\
%
%
      f_{x_i y_i}(x_i,\pm 1) &=  x_i^6 &+& O(x_i^5),
 &     f_{x_i y_i}(\pm 1,y_i) &=  y_i^6 &+& O(y_i^5),\\
      f_{x_1 y_2}(x_1,x_1) &=  x_1^8 &+& O(x_1^7),
  &    f_{x_2 y_1}(x_2,x_2) &=  x_2^8 &+& O(x_2^7)
        \end{array}
      \end{alignat*}
and
\begin{alignat}1
\label{eq:fx2y111}
      f_{x_2 y_1}(1,1) &= -64 c_1 c_2 (-c_1+c_2+1), \\
\label{eq:fx2y211}
      f_{x_2 y_2}(1,1) &= 64c_2^2. 
\end{alignat}
\end{lemma}
\begin{proof}
We verify that  $f_{x_1 y_1}$ exists and satisfies the first two equalities.
The symmetry (\ref{eq:sym1})
implies the existence of $f_{x_2 y_2}$ and  the
first two equalities.

The same reasoning applies to the third and forth equalities.
Moreover, the sixth equality follows from the fifth one by the same
symmetry as before.

We verify final two inequalities directly. 
\end{proof}

\begin{lemma}
\label{lem:cong2}
There exist polynomials $f_{x_1 x_2} \in J\cap \IQ[c_1,c_2,x_1,x_2]$
and $f_{y_1 y_2} \in J\cap \IQ[c_1,c_2,y_1,y_2]$ that satisfy
      \begin{alignat}2
\nonumber
      f_{x_1 x_2} &= -2(x_2+c_2)^2 x_1^6 &+&O(x_1^5),\\
\nonumber
      f_{x_1 x_2} &= -2(x_1+c_1)^2 x_2^6 &+&O(x_2^5), \\
\nonumber
      f_{y_1 y_2} &= -2(y_2-c_2)^2 y_1^6 &+&O(y_1^5), \\
\nonumber
      f_{y_1 y_2} &= -2(y_1-c_1)^2 y_2^6 &+&O(y_2^5), \\
\nonumber
      f_{x_1 x_2}(x_1,x_1) &= x_1^8 &+ &O(x_1^7), \\
\nonumber
      f_{y_1 y_2}(y_1,y_1) &= y_1^8 &+& O(y_1^7), \\
\label{eq:fx1x211}
      f_{x_1 x_2}(1,1) &= 64 c_1 c_2 (c_1+c_2+1). &
      \end{alignat}
\end{lemma}
\begin{proof}
  As in the previous lemma we use symmetry, here   (\ref{eq:sym2}), to reduce the
  statements on $f_{y_1 y_2}$ to the corresponding statements on
  $f_{x_1 x_2}$. We then check the claims on $f_{x_1 x_2}$  
  directly.
\end{proof}

We recall the variable $t$ introduced near \eqref{eq:definet}.
 Using  the  ideal
\begin{equation*}
  K = JB[t] + ((x_1+1)t-(y_1+1))B[t] \subset B[t]
\end{equation*}
 we will eliminate  all  variables but 
$c_1,c_2,x_2,$ and $t$. 
We observe that $B[t]/K$ equals $(B/J)[(y_1+1)/(x_1+1)]$ inside the
field of fractions of $B/J$. Using computer algebra we can verify that
$c_1$ is a member of the ideal in $\IQ[c_1,c_2,x_1,y_1,x_2,y_2]$
generated by $f_1,f_2,f,x_1+1,y_1+1$. So we can write $1$ as a
$B$-linear combination of these generators. After dividing by $x_1+1$
we see that $B[t]/K$ equals the localization of $B/J$ at $x_1+1$. 

Note that the scheme $\cC_c$, whose
 irreducible components we want to show are curves, is just the specialization
of $B[t]/K$ to the corresponding value of $c$.

If $g\in B[t]$ and $c\in(\IQ\ssm\{0\})^2$ then as usual
 $\spe{g}$ denotes the specialization of $g$ in $\IQ[x_1,y_1,x_2,y_2,t]$.

\par
\begin{lemma}
\label{lem:projirred}
The following hold true. 
\begin{enumerate}
\item [(i)]
The intersection $K\cap \IQ[c_1^{\pm 1},c_2^{\pm 1},x_2,t]$
is 
generated as an ideal  by $f_{x_2 t} \in \IQ[c_1,c_2,x_2,t]$ which is an irreducible
element 
of $\IQ[c_1,c_2,x_2,t]$
 with $\deg_{t} f_{x_2 t} = 6$,
\begin{equation} \label{eq:fx2t} 
\begin{aligned}
      f_{x_2 t}(1,t) & = 32c_1^3 t^2
\left((-c_1+c_2+1)t^2 - (c_1 + c_2 + 1)\right), \quad \text{and}\\
f_{x_2 t}(-1,t) &=   32c_{1}  (t^{2} - (c_{1}  + c_{2}  +
1) t +  c_{2}) \cdot \\
& \phantom{=} \cdot  (c_{2} t^{2} + (c_{1}  -  c_{2}  -  1) t + 1) ((c_{1}  + c_{2}  -  1) t^{2} + c_{1} -  c_{2}
+ 1). \\
\end{aligned}
\end{equation}
\item[(ii)] Say $\varphi:\IQ[c_1,c_2,x_2,t]
\rightarrow \IQ[c_1,x_2,t]$ is the ring homomorphism that
  maps $c_2$ to $\pm c_1 + 1$
and $c_1,x_2,t$ to themselves. 
Then $\varphi(f_{x_2 t})$ is absolutely irreducible as an element
of $\IC(c_1)[x_2,t]$. 
\item[(iii)]
If $c\in\IQ^2$  has non-zero coordinates,
then $\cu$ is non-empty. 
If we assume further that $\cu$ is irreducible, the
 projection of $\cu$ to $\IA^2$ by taking the coordinates $x_2$ and $t$ is
Zariski dense in the curve cut out by $\spe{f_{x_2 t}}$.   
Moreover, $\spe{f_{x_2 t}}$ is
irreducible in $\IQ[x_2,t]$.
\end{enumerate}
\end{lemma}
\begin{proof}
The existence of $f_{x_2 t}$ and an explicit presentation follows from 
computer algebra assisted elimination.
We will not reproduce this polynomial here as it has $453$ non-zero
terms and total degree $12$. 
In this way we also check the properties \eqref{eq:fx2t}
and that $f_{x_2 t}$ is irreducible as an element of
$\IQ[c_1,c_2,x_2,t]$. This settles part (i) and we move to part (ii).

Using computer algebra it is possible to check that a given polynomial is irreducible
 over a given number field. We find that $\varphi(f_{x_2 t})$ 
is irreducible as an element of 
$\IQ(\sqrt{3})[c_1,x_2,t]$. It is irreducible as an element of $\IQ(\sqrt{3},c_1)[x_2,t]$
as it depends on $t$. To deduce geometric irreducibility, we apply
the following trick. Let $g\in \IQ(\sqrt{3},c_1)[x_2,t,u]$
 denote the homogenization of $\varphi(f_{x_2 t})$ with
respect to $x_2$ and $t$ and where $u$ is the new projective variable. Then $g$ is irreducible in
$\IQ(\sqrt{3},c_1)[c_2,t,u]$. 
Next we claim that it is irreducible in $L[c_2,t,u]$ where
$L$ is an algebraic closure of  $\IQ(c_1)$.
Indeed, we remark that 
 $g(1/\sqrt{3},1,1)=0$ and
$\frac{\partial}{\partial x_2} g(1/\sqrt{3},1,1)\not=0$. 
In other words, $[1/\sqrt{3}:1:1]$ is contained in precisely one 
geometric component of the vanishing locus in $\IP^2$ of $g$. 
If $h$ were reducible over $L$, then some element of the
Galois group of $L/\IQ(\sqrt{3},c_1)$ would map said component to
some other component. But this is impossible as the Galois group
leaves $[1/\sqrt{3}:1:1]$ invariant. So $g$ is irreducible in
$L[x_2,t,u]$. It follows that
 $\varphi(f_{x_2 t})$ is irreducible in  $L[x_2,t]$.  
So this polynomial must be irreducible over any base field containing
$\IQ(\sqrt{3},c_1)$ as $L$ is algebraically closed. 
This yields part (ii) and we now prove part (iii). 
%
%

 The natural ring homomorphism
$\IQ[c_1^{\pm 1},c_2^{\pm 1},x_2,t]/(f_{x_2 t}) \rightarrow B[t]/K$ is injective by
(i). Thus so is 
\begin{equation}
\label{eq:ringmap}
 \IQ[c_1^{\pm 1},c_2^{\pm 1},x_2,t,(t^2+1)^{-1}]/(f_{x_2 t})
 \rightarrow B[t,(t^2+1)^{-1}]/K
\end{equation}
as localizing is an exact functor. Therefore, the corresponding morphism of
affine schemes $\pi:\cZ\rightarrow\cW$ is dominant. 

Using computer algebra 
we  may verify that  the classes of $(t^2+1)x_1,(t^2+1)y_1,$ and $y_2$ in the quotient 
ring $B[t]/K$ are integral over $\IQ[c_1,c_2,x_2,t]$. 
Indeed, to check this for $(t^2+1)x_1$ we first produce a list of
generators
of the ideal $(f_1,f_2,f,(x_1+1)t-(y_1+1))\subset
\IQ[c_1,c_2,x_1,y_1,x_2,y_2,t]$ intersected with $\IQ[c_1,c_2,x_1,x_2,t]$. For each member of this
list we extract the leading term as a polynomial in the variable
$x_1$.
Next we show that these leading terms generate an ideal containing
$t^2+1$. A similar procedure yields our claims for $(t^2+1)y_1$ and
$y_2$. 
\par
Thus any element on right of
\eqref{eq:ringmap} 
is integral over the ring on the left. In other words,  
$\pi$ is a finite morphism. 
Finite morphisms are closed and since $\pi$ is
dominant we conclude that $\pi$  is surjective.
\par
After doing a base change we may view $\pi$ as a family of maps $\pi_c: \mathcal{Z}_c\rightarrow \cW_c$
parameterized by  $c\in (\IQ\ssm\{0\})^2$. Surjectivity  is stable under base change, 
so each $\pi_c$ is surjective. The  target $\cW_c$ of the map $\pi_c$ 
is the affine scheme attached to 
${\IQ[x_2,t,(t^2+1)^{-1}]/(\spe{f_{x_2 t}})}$. 
It is non-empty since $\spe{f_{x_2 t}} \in \IQ[x_2,t,(t^2+1)^{-1}]$ 
is not a unit; indeed, taking
$f_{x_2 t}(\pm 1,t)$ from  (i) into account we find that
 $x_2$ must appear in 
$\spe{f_{x_2 t}}$.
\par
Now $\cZ_c$ is homeomorphic to an open subspace of
$\cu$, obtained by specialization to $c$ of a localization at
$1/(x_1+1)$ and $t^2+1$. 
Surjectivity of $\pi_c:\cZ_c\rightarrow\cW_c$ implies $\cZ_c\not=\emptyset$.
We conclude that $\cu\not=\emptyset$, so $\cu$ is irreducible by
hypothesis
and therefore $\cZ_c$ is irreducible too. 
Thus $\cW_c$, being the continuous image of $\cZ_c$,
is also irreducible. This is equivalent to the fact that
the nilradical of $\IQ[x_2,t,(t^2+1)^{-1}]/(\spe{f_{x_2 t}})$ is a prime ideal. 
As $\IQ[x_2,t,(t^2+1)^{-1}]$ is a factorial domain we conclude that
$\spe{f_{x_2 t}}$  is the power of an irreducible
 element of $\IQ[x_2,t,(t^2+1)^{-1}]$. 
\par
Next we show that up-to association,
 $\spe{f_{x_2 t}}$  has at most one prime divisor 
 in $\IQ[x_2,t]$.
Otherwise it would be divisible by $t^2+1$.  Then the right-hand sides of
both equations in \eqref{eq:fx2t}  would be divisible by $t^2+1$
leaving us with the contradictory
\begin{equation*}
  -\spe{c_1}+\spe{c_2}+1 = -(\spe{c_1}+\spe{c_2}+1) \quad\text{and}\quad
  \spe{c_1}+\spe{c_2}-1 = \spe{c_1}-\spe{c_2}+1.
\end{equation*}
\par
So $\spe{f_{x_2 t}}$ is the $e$-th power of an  element in
$\IQ[x_2,t]$. We must now prove that $e=1$.
By specialization we see that $\spe{f_{x_2 t}}(\pm 1,t)$
are  $e$-th powers in $\IQ[t]$.
If $e>1$ then the first equation in 
the conclusion of part (i) implies
$(-\spe{c_1}+\spe{c_2}+1)(\spe{c_1}+\spe{c_2}+1)=0$. 
If $\spe{c_1} = \spe{c_2} + 1$, then the second one simplifies to 
\begin{equation*}
  64  \spe{c_1} (t^2 - 2\spe{c_1}t + \spe{c_1} - 1) ((\spe{c_1}-1)t^2 + 1)^2
\end{equation*}
by (\ref{eq:fx2t}). 
The only rational value $\spe{c_1}$ for which the displayed expression
is an $e$-th power with $e>1$ is $0$, which we excluded. 
If $\spe{c_1} = -\spe{c_2}-1$, then the second equation
in (\ref{eq:fx2t}) simplifies to 
\begin{equation*} 
64 \spe{c_1}  (-t^2 + \spe{c_1} + 1)^2  ((\spe{c_1}+1)t^2 - 2\spe{c_1}t - 1).
\end{equation*}
By a similar argument as before and using $\spe{c_1}\not=0$, this yields $e=1$, as desired. 
\end{proof}




%
%
%

Let $\cC$ be an irreducible component of $\cu\otimes\IC$. 
This is an irreducible variety over $\spec{\IC}$. 
By abuse of notation
$x_{1,2}$ and $y_{1,2}$ as elements of $\IC(\mathcal C)$,
the function field of $\mathcal C$.


\begin{lemma}
\label{lem:xiyiprops2}
Suppose $c\in\IQ^2$ has non-zero coordinates and let $\mathcal C$ be an
irreducible component of $\cu\otimes\IC$.
\begin{enumerate}
\item[(i)] 
The functions $x_1,y_1,x_2$ and $y_2$ are non-constant when considered
as elements of $\IC(\mathcal C)$.
\item[(ii)] Let $v$ be a valuation of $\IC(\cC)$ that is constant on $\IC$. 
Then $x_i$ is regular at $v$ if and only if $y_i$ is regular
at $v$.
  \end{enumerate}
\end{lemma}
\begin{proof}
Since $\mathcal C$ is a component
  of the intersection of a surface with a hypersurface we have
$\dim\mathcal C\ge 1$. So at least one among $x_1,y_1,x_2,$ and $y_2$
  is non-constant, as these elements generate the function field of $\cC$.
By the first equality in Lemma \ref{lem:cong1} we see
that $\spe{f_{x_i y_i}}$ is non-zero. So
 $x_i$ is constant if and only if $y_i$ is constant. 
Therefore, $x_1$ or $x_2$ is non-constant

If we are in the first case and if $x_2$ or $y_2$ are constant,
then $x_2$ and $y_2$ are constant. So $x_2 = -\spe{c_2}$ by the first
equality and $y_2 = \spe{c_2}$ by the third equality
of Lemma \ref{lem:cong2}. So $x_2-y_2 = -2\spe{c_2}$. 
Now $\spe{f}(x_1,y_1,x_2,y_2)=0$  as an element of $\IC(\cC)$
and (\ref{eq:definef}) implies $\spe{c_1}\spe{c_2}=0$, a contradiction. 
By using the second and fourth inequality of Lemma \ref{lem:cong2}
we arrive at a similar contradiction if $x_2$ is non-constant and
$x_1$ is constant. 

Thus all $4$ functions are non-constant and part (i) follows. 

Part (ii) follows from the first two equalities of Lemma
\ref{lem:cong1}. Indeed, $x_i$ is integral over $\IQ[y_i]$ and vice versa.
\end{proof}

\begin{proof}[Proof of Lemma \ref{lem:xiyiprops}]
Lemma \ref{lem:projirred}(iii) implies that $\cu$ is non-empty.  
By Lemma \ref{lem:xiyiprops2}(i) the coordinates $x_1,y_1,x_2,y_2$ are non-constant
on $\cC$. 
We have $x_i \pm 1 \not =0$ by the fourth equality of Lemma
\ref{lem:cong1} and $y_i\pm 1\not=0$ by the third equality. 
The statements $x_1-y_2\not=0$ and $x_2-y_1\not=0$ are a consequence
of the fifth and sixth equalities.
Finally, $x_2-x_1\not =0$ and $y_2-y_1\not=0$ by the
fifth and sixth equalities of
Lemma \ref{lem:cong2}. We conclude part (i). 
\par
 
We now compute the dimension of $\cC$.  By the first  equality in Lemma \ref{lem:cong1} we see
$\spe{f_{x_i y_i}}\not=0$.
So $x_i,y_i$ are algebraically dependent over $\IC$.
Now $x_1$ and $x_2$ are also algebraically dependent
due to the  first equality in Lemma \ref{lem:cong2}.
Hence $x_1,y_1,x_2,y_2$ are pairwise algebraically dependent. 
Therefore, $\IC(\cC)/\IC$ has transcendence degree $1$ and thus
$\mathcal C$ is a curve.
\end{proof}

\paragraph{The coset condition.} 

We write $\deg{f}$ for the degree of an element $f\not=0$  of the function
field of an irreducible component  of $\cu\otimes\IC$. 
We recall some facts, to be used further down, which we call basic degree properties. 
If $g$ lies in the same function field, then it is well-known that
\begin{equation*}
  \deg{(f+g)}\le \deg{f}+\deg{g}\quad\text{and}\quad
\deg{fg}\le \deg{f}+\deg{g}
\end{equation*}
assuming the relevant quantities are well-defined. 
If $n\in\IZ$, then 
\begin{equation}
\label{eq:deghom}
  \deg{f^n} = |n|\deg{f}. 
\end{equation}
Finally,
\begin{equation*}
  \deg{f} = 0 \quad\text{or}\quad\deg{f}\ge 1
\end{equation*} 
and the first case happens if and only if $f$ is constant.

In the  proof of the next lemma we use $\deg{\cZ}$ to denote the degree
 of the Zariski closure of the affine variety $\cZ$ in $\IP^n\supset\IA^n$.


\begin{lemma}
\label{lem:3264}
Say $c\in\IQ^2$ has non-zero coordinates and let $\mathcal C$ be an
irreducible component of $\cu\otimes\IC$. We consider $R_1,R_2,$ and $R_3$ 
presented in (\ref{eq:defR123}) as 
non-zero elements of the function field $\IC(\cC)$.
  \begin{enumerate}
  \item [(i)] We have
$\deg{R_{2}}\le 32$ and $\deg{R_3}\le 64$. 
\item[(ii)]
Say $\pm \spe{c_1} - \spe{c_2}+1=0$ with $|\spe{c_1}|\ge 1/2$ and 
suppose there is a multiple
 $(b_1,b_2,b_3)\in\IZ^3\ssm\{0\}$ of $(\spe{c_1},\spe{c_2},1)$ with 
$R_1^{b_1}R_2^{b_2}R_3^{b_3}$ a constant. Then
$\spe{f_{x_2 t}}$ is irreducible in $\IQ[x_2,t]$. 
  \end{enumerate}
\end{lemma}
\begin{proof}
We recall that $R$ is two-to-one on $\cY$. So at least one among
$R_1,R_2,R_3$ is non-constant.
The quadratic polynomials $f_1$ and $f_2$ from
(\ref{eq:defineX}) are irreducible. 
Their set of common zeros is $\overline{\cY}$ by Lemma \ref{lem:cYgeomirred}. 
By B\'ezout's Theorem 
we find $\deg \overline{\cY}\le 4$. 
The curve $\mathcal{C}$ is then an irreducible component of
the intersection of 
$\overline{\cY}$  with a hypersurface of degree $2$. 
Thus $\deg \mathcal{C}\le 8$.
Hence each coordinate function $x_1,y_1,x_2,y_2$ on $\mathcal{C}$ has
degree at most $8$. 

We apply the basic degree properties and (\ref{eq:defR123}) to bound
$\deg{R_{2}}\le 32$ and $\deg{R_3}\le 64$. This yields part (i).

To prove (ii)  let $(b_1,b_2,b_3)\in\IZ^3$ be a non-zero multiple of
$(\spe{c_1},\spe{c_2},1)$ with 
\begin{equation}
  \label{eq:constant1}
R_1^{b_1}R_2^{b_2}R_3^{b_3}\text{ a constant.}
\end{equation}
We observe that $b_1b_2b_3\not=0$. 
Let $Q> 1$ be an integer to be specified later on. 
By Lemma \ref{lem:dirichlet} with $n=1$ and $\theta=1/\spe{c_1}$ there are
integers $p$ and $q$ with $1\le q < Q$ and
\begin{equation}
\label{eq:dirichlet}
  \left|q\frac{1}{\spe{c_1}}-p\right| \le \frac{1}{Q}.
\end{equation}
Without loss of generality, $p$ and $q$ are coprime. 
The fact that $R_1^{b_1}R_2^{b_2}R_3^{b_3}$ is constant implies
\begin{equation*}
  |b_1| \deg{(R_1^q R_2^{\pm q + p}R_3^{p})} 
= \deg{(R_2^{b_1(\pm q+p)-b_2q}R_3^{b_1p-b_3q})}.
\end{equation*}
We use basic degree properties and the bounds from (i) 
to estimate
\begin{equation*}
  \deg{(R_1^q R_2^{\pm q + p} R_3^{p})}
\le 32 \left|\pm q+p - \frac{b_2}{b_1}q\right|
+64 \left|p-\frac{b_3}{b_1}q\right| = 96
\left|q\frac{1}{\spe{c_1}}-p\right| \le \frac{96}{Q}
\end{equation*}
since $\spe{c_1} = b_1/b_3$ 
and $\pm b_1-b_2+b_3=0$. 

We fix  $Q=97$, then 
\begin{equation}
  \label{eq:constant2}
R_1^{q}R_2^{\pm q +p} R_3^{p}\text{ must be a constant.}
\end{equation}
We recall $|\spe{c_1}|\ge 1/2$ and use  (\ref{eq:dirichlet}) again to
bound
$|p|\le 2q + 1/Q < 2q+1$. Hence $|p|\le 2q$.

We  know from Lemma \ref{lem:SoODD} that there are exactly $3$ cosets of dimension $1$
 in the
cross-ratio domain $\cS$. As $b_1b_2b_3\not=0$ 
there is up-to scalars at most one multiplicative relation among the $R_i$.
So (\ref{eq:constant2})  implies that the 
vectors $(\spe{c_1},\spe{c_2},1)$ and $(q,\pm q+p,p)$ are linearly 
dependent. 
In particular, $p\not=0$ and $p\not=\mp q$.

Using {\tt sage} we run over all coprime integers $p,q$ with 
$1\le q\le Q-1=96$, $1\le |p|\le 2q$, and $p\not=\mp q$
to verify that $f_{x_2 t}$ 
is irreducible as an element of $\IQ[x_2,t]$
when 
specializing  $c$ to $(q/p,1 \pm q/p)$ 
As our $\spe{f_{x_2 t}}$  is among these we conclude part (ii). 
\end{proof}

\begin{lemma}
\label{lem:multindep}
  Let $c\in\IQ^2$ have non-zero coefficients
 with $(\spe{c_1}+\spe{c_2}+1)(-\spe{c_1}+\spe{c_2}+1)\not=0$
 and suppose
that $\spe{f_{x_2 t}}$ is irreducible in $\IQ[x_2,t]$.
Let $\mathcal C$ be an irreducible component of $\cC_c\otimes\IC$ and 
$(b_1,b_2,b_3)\in\IZ^3$ with 
$R_1^{b_1}R_2^{b_3}R_3^{b_3}$ a constant in the function field of 
$\mathcal C$. 
If $|b_2|\le |b_1|$ then $b_1=b_2=b_3=0$. 
\end{lemma}
\begin{proof}
We will repeatedly use Lemma \ref{lem:xiyiprops} and write $K= \IC(\cC)$ for the function field of $\cC$.
First, let us fix an irreducible factor $p\in \IC[x_2,t]$
of $\spe{f_{x_2 t}}$ that vanishes when taken as a function on
$\mathcal C$. 
As $\spe{f_{x_2 t}}$ has rational coefficients
we may suppose that the coefficients of $p$ are in a
finite extension of $\IQ$.
Let us suppose that 
$(b_1,b_2,b_3)\in\IZ^3$ is as in the hypothesis. 

We consider $F$, the function field of the plane curve defined by the
vanishing locus of $p$,
 as a subfield of $K$ containing
$x_2$ and $t$. So $K/F$ is a finite extension of function fields.

By (\ref{eq:fx2t}) of  Lemma \ref{lem:projirred} and the hypothesis on
$c$
we see that $\spe{f_{x_2 t}}(1,t)$ has a non-zero root.
As $\spe{f_{x_2 t}}$ is a product of conjugates of $p$ over $\IQ$
up-to a factor in $\IQ^*$ we
conclude that $p(1,t)$ also has a non-zero root $t_0\in\IC$. 
So there is a valuation $v$ of $F$, constant on $\IC$, that
corresponds to the point $(1,t_0)$ on the vanishing locus of $f$. 
In other words, $v(x_2-1) > 0$ and  $v(t) =0$. We extend $v$ to
$K\supset F$ and
 remark
\begin{equation}
\label{eq:valuedrelation}
  b_1 v(R_1)+b_2 v(R_2)+b_3 v(R_3) =0
\end{equation}

We also have $v(x_2) = v(x_2+1)=0$ by the ultrametric
triangle inequality. 
Moreover, $y_2$ is  regular at $v$ by Lemma \ref{lem:xiyiprops2}(ii).
By (\ref{eq:fx2y211}) and $\spe{c_2}\not=0$ we must have $v(y_2-1)=0$. 
Using  these facts together
with (\ref{eq:defR123}) yields
\begin{equation*}
  v(R_1) = v(x_2-1)+v(y_2+1) - v(x_2+1) - v(y_2-1)
=  v(x_2-1)+v(y_2+1) > 0.
\end{equation*}

We proceed to show that $R_2$ and $R_3$ have valuation $0$. 
This   will imply $b_1=0$ and then $b_2=0$, since 
$|b_2|\le |b_1|$. 

If $v(x_1) < 0$ or $v(y_1)<0$ then both are negative according to Lemma
\ref{lem:xiyiprops2}(ii). 
In this case $v(R_2)=v(R_3)=0$ is an immediate consequence of the
ultrametric triangle inequality and (\ref{eq:defR123}). 
So we may assume that $x_1$ and $y_1$ are regular at $v$. 

Since $x_2$ specializes to $1$ at $v$ we may use 
 (\ref{eq:fx2y111}) and (\ref{eq:fx1x211})  to deduce $v(y_1-1)=0$ and $v(x_1-1)=0$,
respectively. 
The expression for $R_2$ in (\ref{eq:defR123t}) yields
\begin{equation}
\label{eq:vR20}
  v(R_2) = v(t).
\end{equation}
So $v(R_2)=0$ by the construction of $v$. 

Next we use
 $R_3$ as given in (\ref{eq:defR123t}).
As $x_2$ specializes to $1$ but $x_1$ and $y_1$  do not, we
have
$v((x_2-x_1)/(x_2-y_1))=0$.
Therefore,
\begin{equation}
\label{eq:R3reform}
  v(R_3) = v(t^{-1}) 
\end{equation}
and so $v(R_3)=0$. 

We have established $b_1=b_2=0$. To conclude that $b_3$
also vanishes we proceed  similarly. 
By (\ref{eq:fx2t}) we find $\spe{f_{x_2 t}}(1,0)=0$ 
and thus $p(1,0)=0$. 
We again fix a
valuation $w$ of $F$ with $w(x_2-1)>0$. But this time we impose
$w(t)>0$.

As above we first conclude that $y_2$
is regular with respect to $w$.

But what if $w(x_1)<0$ or $w(y_1)<0$?
The expression $f_2$ from \eqref{eq:defineX} is identically $0$ on 
$\cC$, so $w(x_1^2+y_1^2) = w(x_2^2+y_2^2)\ge 0$. 
The ultrametric triangle inequality implies
$w(x_1) = w(y_1) < 0$. We recall (\ref{eq:definet}) to find 
 $w(t) = w(y_1+1)-w(x_1+1)=0$ and this contradicts our choice of $t$. 
Hence $w(x_1)\ge 0$ and $w(y_1)\ge 0$. 

To complete the proof  we use again
 \eqref{eq:fx2y111} and \eqref{eq:fx1x211}
and conclude $w(y_1-1)=w(x_1-1)=0$. As above
we have
$w(R_3) = w(t^{-1}) \not=0$. 
So $b_3=0$ 
 because (\ref{eq:valuedrelation}) holds with $v$ replaced  by $w$. 
\end{proof}

\begin{proof}[Proof of Proposition \ref{prop:R123indep}]
Suppose $(b_1,b_2,b_3)\in\IZ^3$ such that $R_1^{b_1} R_2^{b_2} R_3^{b_3}$ is
constant 
 on an irreducible component 
$\mathcal{C}$ of $\cu\otimes\IC$. 


Swapping $\spe{c_1}$ with $\spe{c_2}$ corresponds to swapping
 $(x_1,y_1)$ with $(x_2,y_2)$ by (\ref{eq:sym1}). 
This has the effect of swapping $R_1$ with $R_2$
and thus swapping $b_1$ with $b_2$.
So without loss of generality we may suppose 
$|b_2|\le |b_1|$.

Multiplying $c$ by $-1$ corresponds to 
swapping  $(x_1,x_2)$ with $(y_1,y_2)$ by (\ref{eq:sym2}).
This induces  the transformation
$(R_1,R_2,R_3)\mapsto (R_1^{-1},R_2^{-1},R_3)$ and replaces
$(b_1,b_2,b_3)$ by $(-b_1,-b_2,b_3)$. We also observe that
\begin{equation*}
  (c_1+c_2+1)(-c_1+c_2+1) - (c_1-c_2+1)(-c_1-c_2+1)=4c_2.
\end{equation*}
So 
\begin{equation}
\label{eq:c12nonzero}
(\spe{c_1}+\spe{c_2}+1)(-\spe{c_1}+\spe{c_2}+1)\not=0
\end{equation}
 or 
$(\spe{c_1}-\spe{c_2}+1)(-\spe{c_1}-\spe{c_2}+1)\not=0$. Hence
after possibly
replacing $c$ by $-c$ we may suppose that
(\ref{eq:c12nonzero}) holds. 

If we suppose that $\spe{f_{x_2 t}}$ is irreducible in $\IQ[x_2,t]$,
then Lemma \ref{lem:multindep} applies and we conclude
$b_1=b_2=b_3=0$, as is desired in (i). 

So say $\spe{f_{x_2 t}}$ is not irreducible in $\IQ[x_2,t]$.
Then
$\pm \spe{c_1}  - \spe{c_2}+ 1 = 0$ 
by Lemmas \ref{lem:irred} and \ref{lem:projirred}(iii).
Now $\spe{f_{x_2 t}}$ equals the specialization  $\spe{\varphi(f_{x_2
    t})}$
with $\varphi$ as in Lemma \ref{lem:projirred}(ii).  
But this lemma implies that $\varphi(f_{x_2 t})$ is absolutely
irreducible when considered as a polynomial in $x_2$ and $t$ and
coefficients in $\IC(c_1)$.

We can thus apply  a variant of the Bertini-Noether Theorem, cf. Proposition VIII.7
\cite{LangDiophantine62}. It states that $\varphi(f_{x_2 t})$ remains
irreducible in $\IC[x_2,t]$ for all but at most finitely many complex specializations
of $c_1$. Therefore, our $\spe{c_1}$ is contained in a
finite set $\Sigma$ of exceptions which is accounted for by  part (ii) of the
proposition's conclusion. The reference in Lang's book provides
an effective way to determine this exceptional set provided one has
access to
an  effective version of the Nullstellensatz of which there are many
variants. 

Let us suppose, as in (ii), that $c$ is among such an exception
and that $(b_1,b_2,b_3)\in \IZ^3\ssm\{0\}$ is a multiple of
$(\spe{c_1},\spe{c_2},1)$.
Then $|\spe{c_2}|\le|\spe{c_1}|$ since $|b_2|\le |b_1|$ and so $\spe{c_2} = \pm \spe{c_1}+1$
implies $|\spe{c_1}|\ge 1/2$. 
The conclusion of Lemma \ref{lem:3264}(ii) contradicts the fact that
$\spe{f_{x_2 t}}$ is not irreducible. 
Therefore, $(\spe{c_1},\spe{c_2},1)$ and $(b_1,b_2,b_3)$ are linearly
independent.  This completes the proof. 
\end{proof}
\par
\begin{proof}[{Proof of Theorem~\ref{thm:intromainfin}, case
      ${\omoduli[3]}(4)^\odd$}]
By Lemma~\ref{lem:SoODD} the image of $\cY$ lies in $\oa{\mathcal{S}}$, 
hence condition (i) of Theorem~\ref{thm:heightmain} is met. 
By Proposition~\ref{prop:R123indep} either condition (ii) of Theorem~\ref{thm:heightmain} is met
(for $c \not \in \Sigma$) or we are led to one of finitely many curves
coming from the $c$ in $\Sigma$. 
In this case, let $(c_1,c_2,c_3)\in\IZ^3\ssm\{0\}$ be a primitive  multiple of 
 $(\spe{c_1},\spe{c_2},1)$;
it is not a multiple of a hypothetical vector
$(b_1,b_2,b_3)$ as in Proposition~\ref{prop:R123indep}(ii). 
Now $R_1(y)^{c_1}R_2(y)^{c_2}R_3(y)^{c_3}$ is a root of unity and this
relation is not constant on the whole curve. So $R(y)$ has bounded
height, use for example Theorem \ref{thm:silverman}. 
In any case, 
together with the degree bound in Proposition~\ref{prop:degbound4odd}, 
Northcott's theorem
shows that the number of points that could appear as cusps of an 
algebraically primitive \Teichmuller curve in the stratum ${\omoduli[3]}(4)^\odd$ is finite. By \cite[Proposition~13.10]{BaMo12}
there are only finitely many \Teichmuller curves in the stratum $\omoduli[3](4)^\hyp$.
\end{proof}

\subsection{The hyperelliptic locus in $\omoduli[3](2,2)^\odd$} \label{sec:hypin22odd} 

In this section we rely on the methods of \cite{MatWri13}.
\par
\begin{proof}[{Proof of Theorem~\ref{thm:intromainfin}}]
Suppose there was an infinite sequence of algebraically primitive \Teichmuller
curves in this locus. We claim that there is no linear manifold $M$
strictly contained in the hyperelliptic locus in $\omoduli[3](2,2)^\odd$
and containing an algebraically primitive \Teichmuller curve.
This is the analog of Theorem~1.5 in \cite{MatWri13}.
\par
In order to prove the claim we rely on \cite[Theorem~1.5]{Wright_field}. 
Since in this stratum there are no relative periods, the theorem 
reads $\dim(M) \cdot \deg_\IQ(k(M)) \leq 6$, where $k(M)$ 
is the affine field of definition. Since $M$ properly contains
a \Teichmuller curve, $\dim(M)>2$. Since $k(M)$ is contained in $F$, 
the only possibility is $k(M) = \IQ$. This can only happen if $M$
equal the hyperelliptic locus in $\omoduli[3](2,2)^\odd$ by the
argument given in \cite[Corollary~8.1]{Wright_field}. This contradiction
completes the proof.
\par
Next we claim that in this locus there exists a square-tiled surface, 
whose monodromy representation on the complement of $\langle \omega, \overline{\omega}
\rangle$ is Zariski-dense in the $4$-dimensional symplectic group. 
This is the analog of Theorem~1.3 in \cite{MatWri13}, which does to
directly apply, since we are in a codimension one subvariety of a stratum.  
Given this claim, we can apply Theorem~1.2 and Theorem~1.4
in loc.~cit to conclude.
\par
There are several methods to prove Zariski-density. This is shown using
Lie algebra calculations in \cite{MatWri13}. Here, alternatively, we invoke 
a criterion of Prasad and Rapinchuk (\cite{PrasadRapin}): If the representation contains two
matrices $M_1$ and $M_2$ that do not commute, with $M_2$ of infinite order
and the Galois group of the characteristic polynomial of $M_1$ is as
large as possible for a symplectic matrix, i.e.\ here the dihedral group
with $8$ elements, then the representation is Zariski dense or a product of 
$\SL_2(\IC) \times \SL_2(\IC)$. Consequently, if we give three pairwise non-commuting
elements of infinite order, one having the required Galois group and such that
the common $1$-eigenspaces of their second exterior power representation is just
one-dimensional (generated by the symplectic form), then we have shown Zariski-density
in $\Sp_4(\IC)$.
\par
\tikzset{middlearrow/.style={line width=0.02cm,
        decoration={markings, mark= at position 0.6 with {\arrow{#1}} ,},
        postaction={decorate}}}

\begin{figure}
\begin{center}
\begin{tikzpicture}[scale=1]


\draw[middlearrow={to}, ProcessBlue] (0.5,0) -- (0.5,2) arc (180:0:0.35cm);
\draw[middlearrow={to}, ProcessBlue] (1.2,3) arc (0:-180:0.35cm);
\draw[middlearrow={to}, ProcessBlue] (2.5,2) arc (0:180:0.35cm);
\draw[middlearrow={to}, ProcessBlue] (1.8,3) arc (-180:0:0.35cm) -- (2.5,4);
\draw[middlearrow={to}, ProcessBlue] (3.5,2) to (3.5,3);

\draw[middlearrow={to}, Red] (0,1.5) to (1,1.5);
\draw[middlearrow={to}, Red] (0,2.5) to (4,2.5);
\draw[middlearrow={to}, Red] (2,3.5) to (3,3.5);

\node[ProcessBlue] (b3) at (0.5,-0.3) {$b_3$};
\node[ProcessBlue] (b1) at (2.5,1.7) {$b_1$};
\node[ProcessBlue] (b2) at (3.5,1.7) {$b_2$};

\node[Red] (a3) at (1.3,1.5) {$a_3$};
\node[Red] (a2) at (4.3,2.5) {$a_2$};
\node[Red] (a1) at (3.3,3.5) {$a_1$};

\node (1) at (0.25,0.2) {$1$};
\node (2) at (0.25,1.2) {$2$};
\node (3) at (0.25,2.2) {$3$};
\node (4) at (1.25,2.2) {$4$};
\node (5) at (2.25,2.2) {$5$};
\node (6) at (3.25,2.2) {$6$};
\node (7) at (2.25,3.2) {$7$};

\draw (0,0) -- (0,3);
\draw (1,0) -- (1,3);
\draw (2,2) -- (2,4);
\draw (3,2) -- (3,4);
\draw (4,2) -- (4,3);
\draw (0,0) -- (1,0);
\draw (0,0) -- (1,0);
\draw (0,1) -- (1,1);
\draw (0,2) -- (4,2);
\draw (0,3) -- (4,3);
\draw (2,4) -- (3,4);

\end{tikzpicture}
\caption{A square-tiled surface in the hyperelliptic locus of
$\omoduli[3](2,2)^\odd$} \label{fig:squaretiled}
\end{center}
\end{figure}
We use the square-tiled surface given in Figure~\ref{fig:squaretiled} with
side gluings horizontally by the permutation $(1)(2)(3456)(7)$ and vertically
by $(123)(4)(57)(6)$. The representation
of the Veech group elements $\left(\begin{smallmatrix} 
1 &4 \\ 0 & 1\end{smallmatrix}\right)$
resp.\ $\left(\begin{smallmatrix} 1 &0 \\ 6 & 1\end{smallmatrix}\right)$
resp.\ $\left(\begin{smallmatrix} 13 &-6 \\ 24 & -11\end{smallmatrix}\right)$
on the complement of $\langle \omega, \overline{\omega}
\rangle$ in $H^1(X,\IQ)$ is given in the basis 
$$\{4a_1-a_2,-a_2+4a_3,-b_1+b_2,-2b_2+b_3\}$$
by 
$$ A = \left(\begin{smallmatrix} 1 &0 &-1 & 0 \\ 0 & 1 & 0 & -1 \\ 0 & 0 & 1 & 0 \\
0 & 0 & 0 & 1 \end{smallmatrix}\right)
\quad \text{resp.} \quad  B = \left(\begin{smallmatrix} 
 1 &0 &0 &0 \\
 0 &1 &0 &0 \\
-9 &3 &1 &0 \\
-2 &6 &0 &1 \\
\end{smallmatrix}\right)
\quad \text{resp.} \quad  C = \left(\begin{smallmatrix}
11/2 & -3/2 & 3/2 & 0\\
-3/2 & 3/2 &-1/2 & 0\\
 -15 &   5 &  -4 & 0\\
 -6 &   2 &  -2& 1\\
\end{smallmatrix}\right)
.$$
Let $M_1 = AB$. Its characteristic polynomial is $x^4 - 25x^3 + 144x^2 - 25x + 1$
and has the required Galois group. With $M_2=B$ and $M_3=C$ the remaining conditions
are easily checked.
\end{proof}


%% file: torsion_and_moduli.tex
\section{Torsion and moduli}
\label{sec:torsion-moduli}

The following theorem from \cite{moeller} gives strong constraints on the possible algebraically
primitive Veech surfaces with multiple zeros.  Recall that the Abel-Jacobi map is a homomorphism
$\Div^0(X)\to \Jac(X)$.  A \emph{torsion divisor} on $X$ is one whose image under the Abel-Jacobi
map is a torsion point of $\Jac(X)$.

\begin{theorem}[\cite{moeller}]
  \label{thm:torsion_condition}
  If $(X, \omega)$ is an algebraically primitive Veech surface with zeros $p$ and $q$, then $p-q$ is
  a torsion divisor.
\end{theorem}

We say that $(X, \omega)$ has {\em torsion dividing} $N$, if for any pair of zeros $p$ and $q$ of
$\omega$, the order of $p-q$ divides $N$.

In this section, we show that this torsion condition gives strong control over the moduli of the
cylinders of $(X, \omega)$ in any periodic direction.  More precisely, consider a periodic direction
of $(X, \omega)$, and let $\Gamma$ be its dual graph.  The \emph{blocks} of $\Gamma$ are the maximal
subgraphs of $\Gamma$ which cannot be disconnected by removing a single vertex.  As the edges of $\Gamma$
correspond to cylinders in our direction, this gives a partition of this set of cylinders, which we
will also call blocks of cylinders.

\begin{theorem}
  \label{thm:moduli_bound}
  Let $(X, \omega)$ be an algebraically primitive Veech surface, with 
torsion dividing   $N$.  Then  for any block of cylinders $C_1, \ldots, C_n$ of some periodic direction of $(X, \omega)$, we
  have 
$$h[ \modulus(C_1): \ldots :\modulus(C_n)] \leq (n-1) \log N  + \log(n-1)!.$$

  In particular, there are only finitely many choices up to scale for the tuple of moduli in any
  block of cylinders.
\end{theorem}

\begin{rem}
  Note that if $\omega$ has only one zero, the torsion condition is trivial.  Likewise,  the
  conclusion of Theorem~\ref{thm:moduli_bound} is trivial,  as in this case the dual graph has only one
  vertex, so each block consists of a single cylinder.
\end{rem}

\paragraph{Notation and definitions.}

We establish some basic notation and definitions which we will use throughout this section.  

Given a hyperbolic Riemann surface $X$, we write $\rho_X$ for its \Poincare metric, the unique
conformal metric of constant curvature $-1$.  We write $\ell_X(\gamma)$
for the length of a closed curve $\gamma$, and $\ell_X([\gamma])$ for the length of the shortest
curve in the homotopy class $[\gamma]$.

Recall that a Riemann surface $A$ is an \emph{annulus} if $\pi_1(A) \isom \zed$.  Every hyperbolic annulus is
conformally equivalent to a unique (up to scale) \emph{round annulus}, a planar domain bounded by
concentric circles.  
The \emph{modulus} of $A$ is $\modulus(A)  = \frac{1}{2\pi} \log(R)$,
where $R>1$ is determined by $A\isom \{z : 1 < |z| < R\}$.  If $A$ has a flat, conformal metric with
geodesic boundary, then $\modulus(A) = h/w$, where $h$ and $w$ are its height and width
respectively.

Suppose $\gamma\subset X$ is a simple closed geodesic, and let $\tilde{X}\to X$ be the corresponding
annular cover.  The hyperbolic length of $\gamma$ is related to the modulus of $\tilde{X}$ by
\begin{equation}
  \label{eq:2}
  \ell_X(\gamma) = \frac{\pi}{\modulus(\tilde{X})}.
\end{equation}
We will use the notion of a flat family of stable curves (in the analytic category) also in the case that the fibers are of finite type.  That is, 
a \emph{family of stable curves} over a Riemann surface $\mathcal{C}$ is a
two-dimensional analytic space $\mathcal{X}$ together with a holomorphic function $f\colon \mathcal{X}\to
\mathcal{C}$ whose fibers are stable curves, that is connected one-dimensional analytic spaces
whose only singularities are nodes, and with each component of the complement of the nodes a hyperbolic Riemann surface.
%
A model of a family of nonsingular curves degenerating to a node is given by the family $\pi_k\colon
V_k \to \Delta$, where
\begin{equation*}
  V_k = \{(x,y,t)\in \Delta^3 : xy = t^k\},
\end{equation*}
and $\pi_k(x, y, t) = t$.  Roughly speaking, the family $f\colon \mathcal{X}\to\mathcal{C}$ is
\emph{flat} if near every singularity of a fiber of $f$, there is a change of coordinates where the
family is $\pi_k\colon V_k \to \Delta$. We refer to \cite{hubbardkoch} for a more precise definition.

Given a family of stable curves $f\colon\mathcal{X}\to\mathcal{C}$, for $t\in \mathcal{C}$, we will
write $X_t$ for the fiber $f^{-1}(t)$.  A subscript $t$ will denote the restriction of various
objects to $X_t$, for example if $\omega$ is a section of $\omega_{\mathcal{X}/\mathcal{C}}$, then
$\omega_t$ is the restriction to $X_t$.

If $p$ is a node of a fiber $X_{t_0}$ for every $t$ close to $t_0$, there is a homotopy class of a
simple closed curve  $[\gamma_t]$ which degenerate to $p$ as $t \to t_0$, and such that the
monodromy around $t_0$ preserves each homotopy class $[\gamma_t]$.  We call these curves $\gamma_t$
the \emph{vanishing curves} of $p$.

\paragraph{Tall cylinders.}

Consider a family of flat surfaces $(X_t, \omega_t)$ which is degenerating to a stable curve
as $t\to 0$, where the period of $\omega_t$ around a vanishing curve $\gamma_t$ is real and independent
of $t$.  The following theorem makes precise the intuition that as $t\to 0$, the surfaces $(X_t,
\omega_t)$ are developing cylinders of large modulus in this homotopy class.

\begin{theorem}
  \label{thm:pinching_cylinders}
  Consider  a proper flat family of stable curves $f\colon \mathcal{X}\to\mathcal{C}$ with $p$ a 
  node of a singular fiber $X_{t_0}$.  Let $\omega$ be a meromorphic section of
  $\omega_{\mathcal{X}/\mathcal{C}}$ such that
  \begin{enumerate}
  \item $p$ is not contained in any zero or polar divisor of $(\omega)$, and
  \item the periods $\int_{\gamma_t} \omega_t$ around the vanishing curves of $p$ are a real constant.
  \end{enumerate}
  Then for $t$ sufficiently close to $t_0$, there is on $(X_t, \omega_t)$ a unique maximal,
  horizontal, flat cylinder $C_t$ homotopic to $\gamma_t$.

  Moreover, we have as $t\to t_0$,
  \begin{equation*}
    \modulus(C_t) \sim \frac{\pi}{\ell_{X_t}([\gamma_t])}.
  \end{equation*}
\end{theorem}

\begin{rem}
  Similar statements have appeared in the literature, see for example \cite{masur75}
  or \cite{bainbridge07}.
\end{rem}

If $\omega$ is a one-form defined on a neighborhood of $0$ in $\cx$, with a simple pole at $0$ of
residue $1/2 \pi i$, it is well-known that there is a change of coordinates $\phi$ such that
$\phi^*\omega = \frac{1}{2\pi i} \frac{dz}{z}$.  It follows that in the flat metric $\omega$, a
neighborhood of $0$ is an infinite cylinder circumference $1$.  The proof of
Theorem~\ref{thm:pinching_cylinders} will be based on the following relative version of this change
of coordinates, giving a standard form for a section of $\omega_{\mathcal{X}/\mathcal{C}}$ near an
isolated node, where the family is modeled by $\pi_k\colon V_k \to \Delta$.

\begin{lemma}
  \label{lem:normal_form}
  Let $\omega$ be a holomorphic section of $\omega_{V_k / \Delta}$ such that for each vanishing
  curve $\gamma_t$, we have
  \begin{equation*}
    \int_{\gamma_t} \omega_t = 1.
  \end{equation*}
  Then there is a holomorphic change of coordinates $\Phi$ of $V_k$ which fixes $0$ and commutes with
  the projection $\pi_k$, such that
  \begin{equation}
    \label{eq:8}
    \Phi^*\left( \frac{1}{2\pi i}\frac{dx}{x}\right) = \omega.
  \end{equation}
\end{lemma}

\begin{proof}
  The form $\omega$ may be
  written as the restriction to the fibers of
  \begin{equation*}
    \omega = \frac{1}{2\pi i} \left(\frac{dx}{x} + f \, dx + g \, dy\right).
  \end{equation*}
  for some holomorphic functions $f$ and $g$ on $\Delta^3$.  

  Consider now the first order PDE
  \begin{equation}
    \label{eq:5}
    x u_x - y u_y = xf - yg + h(xy, t),
  \end{equation}
  where $h(s,t)$ is a to-be-determined holomorphic function defined near $0$.  We claim that there
  is a unique choice of $h$ so that \eqref{eq:5} has a holomorphic solution $u(x,y,t)$ defined near
  $0$.  To see this, consider the Taylor series
  representations,
  \begin{equation*}
    f = \sum a_{ijl} x^i y^j t^l, \quad g = \sum b_{ijl} x^i y^j t^l, \qtq{and} u  = \sum c_{ijl} x^i y^j t^l.
  \end{equation*}
  We define $h$ by
  \begin{equation*}
    h(s, t) = -\sum_{i,k}(a_{i-1, i, l} - b_{i, i-1, l}) s^i t^l,
  \end{equation*}
  so that $xf - yg + h(xy, t)$ has no $x^i y^i t^l$ terms.  We then obtain a solution to
  \eqref{eq:5} by taking for $i\neq j$,
  \begin{equation*}
    c_{i,j,l} = \frac{a_{i-1, j, l} - b_{i, j-1, l}}{i-j},
  \end{equation*}
  and making an arbitrary choice if $i=j$.
  
  Now suppose $u(x, y, t)$ is a solution to \eqref{eq:5}, and define
  \begin{equation*}
    \alpha = x e^u, \quad \beta = y e^{-u}, \qtq{and} \Phi = (\alpha, \beta).
  \end{equation*}
  If $xy = t^k$, then $\alpha \beta = xy = t^k$, whence $\Phi$ preserves the variety
  $V(xy-t^k)$.  
  
  We now compute,
  \begin{equation}
    \label{eq:6}
    \Phi^*\left(\frac{dx}{x}\right) = d(\log \alpha) = \frac{dx}{x} + u_x \, dx + u_y\, dy.
  \end{equation}
  The equality $xy=t^k$ implies
  \begin{equation}
    \label{eq:7}
    y\, dx + x \, dy = 0.
  \end{equation}
  (interpreting equality of forms as equality of the restrictions to the fibers).  Combining
  \eqref{eq:5} with \eqref{eq:7} implies $u_x \, dx + u_y\, dy = f\, dx + g\, dy + h(xy,t)$, which
  when combined with \eqref{eq:6} yields
  \begin{equation*}
    \frac{1}{2\pi i} \Phi^*\left(\frac{dx}{x}\right) = \omega + \frac{1}{2\pi i } h(t^k, t)\frac{dx}{x}.
  \end{equation*}
  Since $\omega$ has by assumption unit periods on the vanishing curves, as does $\frac{1}{2\pi i}
  \frac{dx}{x}$, we must have $h(t^k, t)\equiv 0$, so $\Phi$ is the desired change of coordinates.
\end{proof}

The previous lemma tells us that our flat family of curves is developing flat cylinders $C_t$ of growing
modulus.  We now need a way to show that the hyperbolic length of the core curve of $C_t$ in its
intrinsic hyperbolic metric is nearly as small as its length as a curve on $X$.  We will now show
that it is enough for the boundary of $C_t$ to lie in the ``thick part'' of $X$.

Recall that the injectivity radius of $X$ at $x$ is the length of the shortest essential loop through $x$.
Given a curve $\gamma$ on $X$, the $\gamma$-injectivity radius of $X$ at $x$ is the length of the
shortest essential loop through $x$ which is homotopic to $\gamma$.  The
\emph{$\epsilon$-thick part} (resp.\ $\gamma, \epsilon$-thick part) of $X$ is the locus of points
with injectivity radius (resp.\ $\gamma$-injectivity radius) at least
$\epsilon$.  We denote by $\Thick_\epsilon(X)$ and $\Thick_{\gamma,\epsilon}(X)$ the
$\epsilon$-thick and $\gamma, \epsilon$-thick parts respectively.

\begin{lemma}
  \label{lem:short_curve_asymptotic}
  Let $X_n$ be a sequence of hyperbolic Riemann surfaces, each containing an essential annulus
  $A_n$, with $\alpha_n\subset A_n$ the unique primitive
  $\rho_{A_n}$-geodesic, and $\gamma_n\subset X_n$ the homotopic simple closed
  $\rho_{X_n}$-geodesic.
     Suppose that $\bdry A_n\subset
  \Thick_{\gamma_n, \epsilon}(X)$, and $\ell_{A_n}(\alpha_n)\to 0$.  Then $\ell_{X_n}(\gamma_n) /
  \ell_{A_n}(\alpha_n)\to 1$.
\end{lemma}

\begin{proof}
  Consider a single essential annulus $A\subset X$.  Passing to the annular cover determined by the
  homotopy class of $A$, we may take $X$ to be a hyperbolic annulus with $\bdry A \subset
  \Thick_\epsilon(X)$.  We write $\ell_A$ and $\ell_X$ for the lengths of the lengths of the
  corresponding simple closed geodesics in their respective \Poincare metrics, which are related to
  their respective moduli by \eqref{eq:2}.

  If $\modulus(X)$ is sufficiently large compared to $1/\epsilon$, then 
  $\Thick_\epsilon(X)$ is the
  union of two round annuli (that is, bounded by concentric circles) $T_1, T_2$ of modulus
  \begin{equation*}
    \modulus(T_i) = \frac{1}{\ell_X} \sin^{-1}\left(\frac{\ell_X}{\epsilon}\right).
  \end{equation*}
  (This is  a straightforward calculation in the band model of the hyperbolic plane.)  If $\ell_A$
  is sufficiently small, then $A$ is not contained in the thick part, and it follows that $X =
  A \cup T_1 \cup T_2$.  Let $B = A \setminus (T_1 \cup T_2)$.  
  We then have
  \begin{equation*}
    \modulus(A) \leq \modulus(X) = \modulus(B) + \modulus(T_1) + \modulus(T_2) \leq \modulus(A) + 2 \modulus(T_1),
  \end{equation*}
  where the inequalities follow from monotonicity of moduli. That is, $A\subset B$ implies
  $\modulus(A) \leq \modulus(B)$ (see for example \cite{mcmullen94}).  In terms of lengths,
  \begin{equation*}
    \frac{1}{\ell_A} \leq \frac{1}{\ell_X} \leq
    \frac{2}{\ell_X}\sin^{-1}\left(\frac{\ell_X}{\epsilon}\right) + \frac{1}{\ell_A}.
  \end{equation*}
  Letting $\ell_A \to 0$ with $\epsilon$ fixed, the claim follows.
\end{proof}

\begin{proof}[Proof of Theorem~\ref{thm:pinching_cylinders}]
  Passing to an open subset of $\mathcal{C}$, we may take our family of curves to be of the form
  $f\colon\mathcal{X}\to\Delta$, with $f_0$ the only singular fiber.  The homotopy class
  $[\gamma_t]$ of the vanishing curve is then well-defined for every $t\neq 0$.  We define a
  function $\iota\colon\mathcal{X}\to\reals$, by taking $\iota(x)$ to be the $\gamma_t$-injectivity
  radius of $X_t$ at $x$.  We may define $\iota$ even on the singular fiber by taking it to be $0$
  at the node $p$, and $\iota\equiv \infty$ on an irreducible component of $X_0$ which does not
  contain $p$.

  Let $\mathcal{X}'$ denote be the complement of the nodes of $\mathcal{X}$.  The vertical
  hyperbolic metric $\rho_{X_t}$ is continuous as a function on $T^*_{\mathcal{X'}/\Delta}$, as is
  shown in \cite{hubbardkoch} or \cite{wolpert90}.  It follows that $\iota$ is continuous on
  $\mathcal{X}$.

  Applying Lemma~\ref{lem:normal_form}, we may take a compact neighborhood $S$ of $p$ whose
  intersection with each fiber $X_t$ is either empty or a flat annulus $C_t'$ which is contained in
  a unique maximal flat annulus $C_t$.  As $\bdry S$ is compact, $\iota$ is bounded below on $\bdry
  S$, so each $\bdry C_t'$ is contained in $\Thick_{\gamma_t, \epsilon}(X_t)$ for some
  $\epsilon>0$.  Applying Lemma~\ref{lem:short_curve_asymptotic}, we obtain $\modulus(C')\sim
  \pi/\ell_{X_t}([\gamma_t])$ as $t\to 0$.  As
  $$\modulus(C_t') \leq \modulus(C_t) \leq
  \frac{\pi}{\ell_{X_t}([\gamma_t])},$$
  the same holds for $\modulus(C_t)$.
\end{proof}

\paragraph{The Abel metric.}

Let $X$ be a Riemann surface and $D= \sum n_i z_i$ a degree zero divisor whose Abel-Jacobi image in
$\Jac(X)$ is $0$.  By Abel's theorem, there is a meromorphic function $h_D\colon X \to \proj^1$ with
$(h_D)=D$.  Let $\tau_D = h_D^*(dz/z)$, a meromorphic one-form having integral periods and a simple
pole at each $z_j$ with residue $n_j / 2 \pi i$.  We will often simply write $\tau$ when we do not
wish to emphasize the divisor $D$.

We call the associated flat metric $|\tau|$ the \emph{Abel metric} associated to $D$.  The
horizontal direction of $\tau$ is periodic, as it comes from pulling back the flat metric on a
cylinder.  In this metric, a neighborhood of each $z_i$ is a half-infinite cylinder of width $n_i$.

Now  fix an algebraically
primitive Veech surface with periodic horizontal direction.  We identify the punctured unit disk
$\Delta^*$ with the quotient of
the hyperbolic plane by the corresponding parabolic element of the Veech group.  Given $t\in
\Delta^*$, we write $(X_t, \omega_t)$ for the corresponding flat surface with $\omega_t$ normalized
so that the horizontal direction is periodic with periods independent of $t$.  Let $f\colon
\mathcal{X}\to \Delta^*$ be the associated universal curve, and $f\colon\bcX\to\Delta$ the proper
flat family of stable curves.  The forms $\omega_t$ yield a section $\omega$ of $\omega_{\bcX/\Delta}$.

The dual graphs of the $(X_t, \omega_t)$ may be canonically identified as the monodromy is composed
of Dehn twists in the cylinders.  We denote the dual graph of each by $\Gamma$, and write $E$ and
$V$ for the set of vertices and edges.  Given an edge $e$ of $\Gamma$, we denote by $C_e(\omega_t)$
the corresponding cylinder of $\omega_t$ and let $[\gamma_e]\subset X_t \setminus Z(\omega_t)$
denote the homotopy class of a core curve of $C_e(\omega_t)$, oriented so that its period is
positive.

We give each edge of $\Gamma$ an orientation as follows.  The bottom and top boundary components of
$C_e(\omega_t)$ correspond to vertices $v_1$ and $v_2$ of $\Gamma$.  If $v_1\neq v_2$, give $e$ the
orientation pointing from $v_1$ to $v_2$.  Otherwise choose an arbitrary orientation.

Let $(m_e)_{e\in E}$ be the tuple of moduli of the cylinders of some $(X_t, \omega_t)$.  As the
$m_e$ have rational ratios, we may scale them uniquely so that they are relatively prime positive
integers.  We regard the tuple $(m_e)$ as weights assigned to the edges of the graph $\Gamma$.

Possibly passing to a cover of $\Delta^*$, we make take the zeros of $\omega$ to be sections $z_i$
of the universal curve over $\Delta$.  Let $Z\subset \mathcal{X}$ denote the divisor of zeros of
$\omega$.

We denote by $\Div^0(Z)$ the group of divisors of degree $0$ supported on $D$.  Let $K$ be the
kernel of the Abel-Jacobi map $\Div^0(Z) \to \Jac(\mathcal{X}/\Delta^*)$, a finite index subgroup
by the torsion condition.

A divisor $D= \sum n_i z_i \in K$ defines a meromorphic section $\tau_D$ of
$\omega_{\mathcal{X}/\Delta}$ which has a simple pole along each $z_j$ with residue $n_j/2 \pi  i$.
The restriction to each $X_t$ is the Abel metric defined above, which we denote $\tau_{D, t}$ or
just $\tau_t$.  

Given $e\in \Gamma$ and $D\in K$, let
\begin{equation*}
  w_e(D) = \int_{\gamma_e} \tau_{D, t}.
\end{equation*}
Equivalently, $w_e(D)$ is the winding number of $h_D(\gamma_e)$ around $0$.

\begin{prop}
  \label{prop:w_e_not_zero_implies_cylinder}
  If $w_e\neq 0$, then for $t$ sufficiently large, $(X_t\setminus Z_t, \tau_t)$ has a unique maximal
  horizontal cylinder $C_t(\tau_t)$ homotopic to $\gamma_e$.  Moreover, as $t\to0$, we have
  \begin{equation}
    \label{eq:11}
    \modulus(C_t(\omega_t)) \sim \modulus(C_t(\tau_t)).
  \end{equation}
  If $w_e=0$, then there is no such cylinder
  for any $t>0$.
\end{prop}

\begin{proof}
  If $w_e \neq 0$, then the cylinder is provided by Theorem~\ref{thm:pinching_cylinders}, and the
  moduli are asymptotic as both are asymptotic to the hyperbolic length of $\gamma_t$.

  If $w_e=0$, then there can be no such cylinder as a closed geodesic on a translation surface
  never has zero period.
\end{proof}

\begin{cor}
  \label{cor:wbounded}
  Given $z_1, z_2$ distinct zeros of $\omega$, let $D = N(z_1 - z_2)\in K.$  Then for each edge $e$ of
  $\Gamma$, we have $w_e(D) \in [-N, N]\cap \zed$.
\end{cor}

\begin{proof}
  If $w_e\neq 0$, then for large $t$ 
  the surface $(X_t \setminus Z_t, \tau_t)$ has  a flat cylinder $C$ in the homotopy class
  $\gamma_e$.  Under the meromorphic function $f_D\colon X_t \to \proj^1$, the cylinder $C$ is a
  degree $w_e(D)$ covering of a flat cylinder $C'\subset \cx^*$.  As the degree of $f$ is $N$, we
  must have $|w_e(D)| \leq N$. 
\end{proof}

Since the cylinders of $\tau_t$ have nearly the same modulus as the corresponding cylinders of
$\omega_t$, we are able to use the geometry of $\tau_t$ to obtain the following strong restriction
on the moduli.

\begin{prop}
  \label{prop:circuit_relation}
  Suppose that $\gamma$ is a closed circuit of $\Gamma$ which crosses the edges $e_{i_1}, \ldots,
  e_{i_n}$.  Let $\sigma_i = \pm 1$ depending on whether $\gamma$ crosses $e_i$ respecting its
  orientation.  Then
  \begin{equation}
    \label{eq:10}
    \sum_i \sigma_i w_{e_i} m_{e_i} = 0.
  \end{equation}
\end{prop}

\begin{proof}
  Take a closed disk $0\in\overline{\Delta}' \subset \Delta$ such that for each edge $e$ with
  $w_e(D)\neq 0$, the cylinder $C_e(\tau_t)$ exists, and let $\bcX'\to\overline{\Delta}'$ be the
  restriction of our family of curves to $\overline{\Delta}'$.  If $w_e(D)\neq 0$, let
  $\mathcal{C}_e\subset\bcX'$ be the open set consisting of the cylinders $C_e(\tau_t)$.  We let
  $J$ be the complement of these $\mathcal{C}_e$.

  Since $J$ is compact, the vertical metric $|\tau|$ is continuous away from the nodes, and
  $\tau$ has no poles at the nodes remaining in $J$,
  there is a uniform constant $C$ such that for any distinct $e$ and $f$ and any $t\in
  \overline{\Delta}'$, if the boundaries of $C_e(\tau_t)$ and $C_f(\tau_t)$ lie in the same
  component of $J$, then any pair of points in these boundaries may be joined by a curve of
  length at most $C$.

  For each $t\in \overline{\Delta}'$, we choose a lift of $\gamma$ to a closed curve
  $\tilde{\gamma}$ on $X_t$ as follows.  For each $e_{i_k}$ with nonzero weight, let
  $\tilde{\gamma}_k$ be a segment joining the boundary circles of $C_e(\tau_t)$ with the orientation
  indicated by $\gamma$.  For $e_{i_k}$ with zero weight, we let $\tilde{\gamma}_k$ be the constant
  curve at the corresponding node.  We close the curve by choosing for each $k$ a curve in $J$ of
  length at most $C$ which joins the endpoints of $\gamma_k$ and $\gamma_{k+1}$.

  For each $e$ with nonzero weight, we define $h_e(t)$ to be the height of the corresponding
  cylinder in the $\tau$-metric, taken with the same sign as $w_e(D)$.  In the $\omega$-metric, the
  corresponding cylinder has modulus $c m_e \log|t|$ for some non-zero 
constant $c$, so 
  by \eqref{eq:11}, we have
  \begin{equation*}
    h_e(t) \sim c \log|t| w_e(D) m_e.
  \end{equation*}
  
 Now since $\tau_t$ has integral periods, the form $\Im \tau$ is exact, so all of its periods are
  $0$.  In particular, for some contribution $D_t$ bounded by $|D_t|\leq n C$ 
stemming from the part of the path in $J$, 
we obtain in the limit $t \to 0$   
\begin{equation*}
     0 = (\log|t|)^{-1}\int_{\tilde{\gamma}}\Im \tau_t = (\log|t|)^{-1}\sum_i \sigma_i h_{e_i}(t) + (\log|t|)^{-1}D_t \to c\sum_i
     \sigma_i w_{e_i} m_{e_i},
  \end{equation*}
  implying \eqref{eq:10}.
\end{proof}

The vertices of $\Gamma$ determine a partition of the zeros of each $\omega_t$, assigning to each
vertex the set $S_v$ of zeros which lie on the corresponding component of the spine of $\omega_t$.  We
define for each vertex $v$ and divisor $D = \sum n_i z_i$ the weight
$$c_v(D) = \sum_{z_i \in S_v} n_i.$$
For each $v$, let $v_{\rm in}$ and $v_{\rm out}$ be the set of incoming and outgoing edges.

\begin{prop}
  \label{prop:w_relation}
  For each vertex $v$ of $\Gamma$,
  \begin{equation}
    \label{eq:12}
    c_v(D) = \sum_{w\in v_{\rm in}} w_e(D) - \sum_{w\in v_{\rm out}} w_e(D) 
  \end{equation}
\end{prop}

\begin{proof}
  On the component of $X_0$ corresponding to $v$, the form $\tau_0$  has simple poles at each of the
  $z_i$ of residue $n_i$.  The nodes of this component correspond to the edges of $\Gamma$ adjacent
  to $v$, and the residues of $\tau_0$ there are the weights $w_e(D)$, with the sign determined by
  the orientation of $e$.  Since the sum of residues is $0$, \eqref{eq:12} follows. 
\end{proof}

\paragraph{Electrical networks.}

Propositions~\ref{prop:circuit_relation} and   \ref{prop:w_relation} have a physical interpretation
in terms of electrical networks.  We regard the graph $\Gamma$ as an electrical network whose edges
have resistances $m_e$, with incoming or outgoing current $c_v$ at each vertex $v$, and with current
$w_e$ across any edge.  
Then \eqref{eq:12} reflects Kirchhoff's current
law ``what goes in must go out.'' , and  \eqref{eq:11} says, using
Ohm's law, that the potential drop around any closed loop is zero. So our
$w_e$ and $m_e$ indeed satisfy all the axioms of an electrical network.
\par
In this language, our task in the remainder of the section is the following:
Given a natural number $N$ and an electrical network with the property that passing a current of $N$
through any two vertices results in integral currents bounded by $N$ along
any edge (Corollary~\ref{cor:wbounded}), show that there is a finite number (depending on N)
of possibilities for the resistances.

The idea of studying Riemann surfaces through electrical networks has appeared elsewhere.  See for
example \cite{diller}.

\paragraph{Torsion determines moduli.}

Equation \eqref{eq:10} and Corollary~\ref{cor:wbounded} together put strong constraints on the moduli $m_e$, which we will now show
determines these moduli up to finitely many choices.

For each $D\in K$, the weights $w_e(D)$ define a relative homology class
$$\tilde{w}_e(D) = \sum_e w_e(D) [e]\in H_1(\Gamma, V),$$
with $V$ the vertex set of $\Gamma$, using the previously chosen orientation of the
edges of $\Gamma$.

Let $\tilde{H}_0(V)$ denote the kernel of the natural homomorphism $H_0(V) \to H_0(\Gamma)$.  The
assignment $D\mapsto c_D(v)$ is a group homomorphism $\rho\colon \Div^0(Z) \to \tilde{H}_0(V)$.

The tuple of moduli $\bm = (m_e)_{e\in E}$  determines a group homomorphism $\phi_\bm\colon H_1(\Gamma, V)\to
H^1(\Gamma, V)$ defined by $\phi_\bm([e]) = m_e [e^*]$, where $[e^*]$ is the corresponding cocycle killing
the other edges.

\begin{prop}
  \label{prop:w_homomorphism}
  The map $\tilde{w}\colon K\to H_1(\Gamma, Q)$ is a group homomorphism.
\end{prop}

\begin{proof}
  For any two divisors, we have $\tau_{D_1+ D_2} = \tau_{D_1} + \tau_{D_2}$.  As the $w_e(D)$ are
  periods of $\tau_D$, it follows that $\tilde{w}$ is a homomorphism.
\end{proof}

As $K\subset\Div^0(Z)$ is finite index, there is an induced homomorphism (abusing notation)
$\tilde{w}\colon \Div^0(Z) \otimes \ratls \to H_1(\Gamma, V; \ratls)$.

Below we write $\bdry\colon H_1(\Gamma, V)\to \tilde{H}_0(V)$ and $j\colon H^1(\Gamma, P)\to
H^1(\Gamma)$  for the homomorphisms of from the respective long exact sequences.

\begin{prop}
  \label{prop:homology_splitting}
  The homomorphism $\tilde{w}$ satisfies $\bdry \circ \tilde{w} = \rho$ and descends to a splitting
  $w\colon \tilde{H}_0(V; \ratls)\to H_1(\Gamma, V; \ratls)$ of the homology exact sequence of the
  pair $(\Gamma, V)$.

  The composition $j\circ \phi_\bm \circ w = 0$ is trivial, and moreover
  \begin{equation}
    \label{eq:3}
    \Im(\phi_\bm\circ w) = \Ker(j).
  \end{equation}
\end{prop}

To summarize this discussion, these maps fit into the following commutative diagram with the
composed map $\tilde{H}_0(V) \to H^1(\Gamma)$ trivial and all (co)homology having
$\ratls$-coefficients.

\begin{tikzpicture}[description/.style={fill=white,inner sep=2pt}]
\matrix (m) [matrix of math nodes, row sep=3em,
column sep=2.5em, text height=1.5ex, text depth=0.25ex]
{ & & & {\Div^0(Z) \otimes \ratls} & \\
   0  &   H_1(\Gamma)  & H_1(\Gamma,V)  & \widetilde{H}_0(V)  & 0 \\
   0  & H^0(V)/H^0(\Gamma)   & H^1(\Gamma,V) 
        & H^1(\Gamma)  & 0 \\  };
\path[->,font=\scriptsize] (m-1-4) edge node[auto] {$\tilde{w}$}  (m-2-3);
\path[->,font=\scriptsize] (m-1-4) edge node[auto] {$\rho$}  (m-2-4);
\path[->,font=\scriptsize] (m-2-1) edge  (m-2-2);
\path[->,font=\scriptsize] (m-2-2) edge  (m-2-3);
\path[->,font=\scriptsize] (m-2-3) edge  node[auto] {$\bdry$}  (m-2-4);
\path[->,font=\scriptsize] (m-2-3) edge  node[auto] {${\phi_\bm}$}  (m-3-3);
\path[->,font=\scriptsize] (m-2-4) edge  (m-2-5);
\path[->,font=\scriptsize] (m-2-4) edge  (m-3-4);
\path[->,font=\scriptsize] (m-3-1) edge  (m-3-2);
\path[->,font=\scriptsize] (m-3-2) edge node[auto] {${\delta}$}  (m-3-3);
\path[->,font=\scriptsize] (m-3-3) edge  node[auto] {${j}$} (m-3-4);
\path[->,font=\scriptsize] (m-3-4) edge  (m-3-5);
\end{tikzpicture}

\begin{proof}
  That $\bdry\circ\tilde{w} = \rho$ is exactly the content of Proposition~\ref{prop:w_relation}.

  That $j\circ \phi_\bm\circ w = 0$ means that $\phi_\bm \circ w$ kills every cycle of $\Gamma$.
  This is exactly the content of Proposition~\ref{prop:circuit_relation}.
 
  We now claim that
  \begin{equation}
    \label{eq:4}
    \Im(\phi_\bm\circ\tilde{w}) = \Ker( j).
  \end{equation}
  Since $\bdry\circ \tilde{w} =
  \rho$, we have $\Ker(\tilde{w}) \subset \Ker(\rho)$, so $\dim \Ker(\tilde{w}) \leq |Z| - |V|$, and
  $\dim \Im(\tilde{w}) \geq |V| - 1$.  Then since $\phi_\bm$ is injective, $\dim
  \Im(\phi_\bm\circ\tilde{w}) \geq |V|-1= \dim \Ker(j)$.  Since $\Im(\phi_\bm\circ\tilde{w})
  \subset \Ker( j)$, they must be equal.

  Now \eqref{eq:4} implies $\dim \Ker(\tilde{w}) = |Z| - |V|$, so $\Ker(\tilde{w}) = \Ker(\rho)$.
  Thus $\tilde{w}$ descends to the desired splitting $w$.  The remaining claims about $w$ then
  follow from the corresponding properties of $\tilde{w}$ which we have already proved.
\end{proof}

Consider now the group $\Gm(\ratls)^{|E|}$, acting diagonally on $H^1(\Gamma, V; \ratls)$ by
$(q_e)_{e\in E}\cdot [e_0^*] = q_{e_0} [e_0^*]$.  Let $B\subset \Gm(\ratls)^{|E|}$ be the subgroup
of $(q_e)_{e\in E}$ for which $q_e = q_f$ whenever the edges $e$ and $f$ lie in the same block of
$\Gamma$.

\begin{lemma}
  \label{lem:stabilizer}
  $B$ is exactly the subgroup of $\Gm(\ratls)^{|E|}$ which stabilizes $\Im(\delta)$.
\end{lemma}

\begin{proof}
  That $B$ stabilizes $\Ker(j)$ is clear, since any circuit $\gamma$ can be written as a sum of
  circuits, each contained in only one block.
  
  For the converse, suppose $\bq = (q_e)_{e\in E}$ stabilizes $\Im(\delta)$.  Given a vertex $v$, we have
  $\delta v = \sum_{e\in E(v)} [e^*]$, where the sum is over the edges incident to $v$, with the
  outward-pointing orientation.  As $\bq$ stabilizes $\Im(\delta)$, the cocycle
  $$\bq \cdot \delta v= \sum_{e\in E(V)} q_e [e^*]$$
  lies in the kernel of $j$.  Suppose that $e_1$ and $e_2$ are two
  edges incident to $v$ which lie in the same block $B$.  As $B$ is $2$-connected, by Menger's
  theorem (see \cite{bondymurty}) there is a
  circuit $\gamma\subset B$ which passes through $e_1$ and $e_2$ but through no other edge incident
  to $v$.    As $(\bq \cdot \delta v)(\gamma) = 0$, we obtain $q_{e_1} = q_{e_2}$.  Applying this
  repeatedly, it follows that
  $q_{e_1} = q_{e_2}$ whenever $e_1$ and $e_2$ lie in the same block, so $\bq \in B$.
\end{proof}

\begin{proof}[Proof of Theorem~\ref{thm:moduli_bound}]
  Consider a divisor $D = N(z_1 - z_2) \in K$.  By Corollary~\ref{cor:wbounded}, we have $w_e(D) \in
  [-N, N]\cap \zed$ for each edge $e$.  In particular, given $N$ there are only finitely many
  possibilities for the splitting $w$, so we may take $w$ to be known.

  Now any possible tuple of moduli $\bm$ satisfies $\Im (\phi_\bm\circ w) = \Ker(\delta)$ by
  Proposition~\ref{prop:homology_splitting}.  By Lemma~\ref{lem:stabilizer}, if $\bm_1$ and $\bm_2$
  are two such tuples, we must have $\phi_{\bm_1} \circ \phi_{\bm_2}^{-1} \in B$, which means
  exactly that in each block the $m_i$ are determined up to scaling.

  To make this effective,  consider a tuple $(m_1, \ldots, m_n)$ of moduli from a single block $B$
  of $\Gamma$.   Given $D$ as above together with a circuit $\gamma\subset B$, we obtain a linear
  function $f$ in the $m_i$ expressing the condition that $\phi_\bm(w(D)) (\gamma) = 0$.  Each such
  $f$ has integral coefficients of absolute value at most $N$, so $h(f) \leq \log(N)$.  Among such
  linear equations, we may choose $f_1, \ldots, f_{n-1}$ which determine $(m_i)$ up to scale.  From
  Cramer's rule, one may deduce that
  \begin{equation*}
    h(m_1, \ldots, m_n) \leq (n-1) \log N + \log(n-1)!\qedhere
  \end{equation*}
\end{proof}


%% file: general_finiteness.tex
\section{A general finiteness criterion} \label{sec:genfinthm}

We say that two stable forms are \emph{pantsless-equivalent} if the underlying stable curves are
isomorphic by a map which identifies the one-forms on any irreducible component which is not a pair
of pants.  A collection of \Teichmuller curves in a stratum ${\cal S}$ is \emph{pantsless-finite} if among
their cusps there are only finitely many pantsless-equivalence classes of stable forms.

In this section, we prove: 

\begin{theorem}
  \label{thm:pantsless_finiteness}
  In a fixed stratum ${\cal S}$, any pantsless-finite collection of algebraically primitive
\Teichmuller curves is finite.
\end{theorem}

Theorem~\ref{thm:pantsless_finiteness} is an easy consequence of the methods in \cite{BaMo12}
if the \Teichmuller curves in the collection are generated by Veech surfaces with
just one zero (and thus have only irreducible degenerations).  All the results on
torsion orders are of course void in this case -- and the final proof works
without them.

\paragraph{The cusps control the torsion orders.}

We first observe that a pantsless-finite collection of \APTCs has uniform torsion order bounds,
allowing us to apply Theorem~\ref{thm:moduli_bound} to control the moduli.

\begin{prop} \label{prop:pf_torsionbound}
  For any pantsless-finite collection of algebraically primitive \Teichmuller curves, there is a uniform bound on the
  torsion orders of $D_1- D_2$ for any two zero-sections $D_i$.
\end{prop}

Half of the argument is isolated in the following lemma for later use. The converse
statement implicitly appears also in \cite{Mo08}
at the end of Section~2. We refer to \cite{harrismorrison} for the notion of admissible coverings.
\par
\begin{lemma} \label{le:conseqoftorsion}
Let $f\colon \ol{\cX} \to \ol{C}$ be a family of curves, smooth over $C \subset \ol{C}$. Let
$s_1, s_2\colon \ol{C} \to \ol{\cX}$ be two sections with image $D_1, D_2$, whose difference $D_1 - D_2$ is 
$N$-torsion in the family of Jacobians $\Jac(\cX/C)$. Then for each fiber $X$ of $f$ there 
exists an admissible cover $h$ from $X$ to a tree of projective lines of degree $N$.  
\par
Let $X_\infty$ be a singular fiber of $f$ and suppose that
both sections pass through the same component $Y$ of $X_\infty$. Then there exists a partition 
$S = \{S_j, j \in J\}$ of the nodes $p_1,\ldots, p_n$ of $Y$ and $h\colon  Y \to \PP^1$ of degree $N$ 
with $h^{-1}(\{0\}) = D_1 \cap X_\infty$, with $h^{-1}(\{\infty\}) = D_2 \cap X_\infty$ and
such that $h$ is constant on each part $S_j$ of the partition.
\par
Conversely, suppose that $s_1,s_2$ are two sections, whose difference is torsion and suppose that
there is a  component $Y$ and a map $h$ of degree $N$, as above. Then $s_1-s_2$ is $N$-torsion.   
\end{lemma}

\begin{proof} For smooth fibers, an admissible cover is just a morphism $h\colon  X \to \PP^1$,
and the existence of such a morphism follows from the definition of the Jacobian and Abel's theorem.
The space of admissible covers is compact, so such an admissible cover exists also for an appropriate
semistable model of the singular fiber. 
\par 
Remove the vertex corresponding to $Y$ from the dual graph of the semistable curve $X_\infty$. 
The remaining graph decomposes into connected components and we partition the nodes of
$Y$ according the the component they are adjacent to. Let $z_1$ and $z_2$ be two nodes
in the same partition $S_j$. Since the image of $h$ is a tree $T$, the components
of $X_\infty$ corresponding to $S_j$ map to a branch of $T \setminus h(Y)$. 
This branch intersects $h(Y)$ in a single point $b$ and $h(z_i) = b$ for all $z_i \in S_j$.
\par
For the converse statement we claim that in a neighborhood $U$ of a singular fiber with $D_1$ 
and $D_2$ limiting to the same component, $D_1 - D_2$ defines a section of the relative
Picard scheme $\Pic^0(\ol{\cX}/U)$. This scheme parameterizes line bundles on $f^{-1}(\cX)$
that are of total degree zero, when restricted to any fiber, up to bundles pulled back from 
the base $0$. In fact, any section can be extended to the \Neron model of the family of
Jacobians, by the universal property of the \Neron model (\cite[Theorem~9.5.4~b)]{BLR}). 
Since $s_1$ and $s_2$ pass through
the same component $Y$ of $X_\infty$, their difference maps to the connected component
of the identity of the \Neron model, which is $\Pic^0(\ol{\cX}/U)$ (\cite[Definition~1.2.1]{BLR}).
\par
The existence of the map $h$ shows that the bundle $\cO_{X_\infty}(N(D_1-D_2)|_{X_\infty})$ is isomorphic 
to the trivial bundle, i.e.\ $(D_1-D_2)|_{X_\infty}$ is of order $N$ in $\Pic^0(\ol{\cX}/U)$.
Since the torsion order is constant in families, this implies the claim.
\end{proof}

\begin{proof}[Proof of Proposition~\ref{prop:pf_torsionbound}]
  Suppose $s_1$ and $s_2$ are zero sections of the universal curve of some \Teichmuller curve in a
  pantsless-finite collection of \Teichmuller curves.  If in some smooth fiber there is a saddle connection
  joining $s_1$ to $s_2$, then rotating the form so that this direction is horizontal and following
  the \Teichmuller geodesic flow to the boundary, we obtain a singular fiber where $s_1$ and $s_2$
  intersect the same irreducible component $(Y, \eta)$.  Since $\eta$ has more than one zero, $Y$ is
  not pants, so by assumption, there are only finitely many possibilities for $(Y, \eta)$.  By
  Lemma~\ref{le:conseqoftorsion}, there are only finitely many possibilities for the order of $s_1 -
  s_2$, since it is determined by $Y$.  In particular, $N(s_1 - s_2) =0$ for some $N$ depending only on $\mathcal{S}$.

  It is possible that two zeros $z, z'$ of a translation surface are not joined by a saddle
  connection; however, taking any curve joining $z$ to $z$, the shortest curve in its homotopy
  class is a union of saddle connections (see \cite{flp}).  Taking $z=z_1, z_2, \ldots, z_n = z'$,
  to be the successive endpoints of these saddle connections, we obtain $z - z' = \sum (z_i -
  z_{i+1})$ with $N(z_i-z_{i+1})= 0 $ for each $i$, so $N(z-z')=0$.
\end{proof}

\paragraph{Geometry of cylinder widths.}

Given an algebraically primitive Veech surface $(X, \omega)$ with periodic horizontal direction, the
widths $(r_1,\ldots, r_n)$ and heights $(s_1, \ldots, s_n)$ of its cylinders may be regarded as
vectors in $\reals^g$ via the $g$ embeddings of the trace field $F$.  We now show that the geometry
of this set of vectors nearly determines -- and is determined by -- the tuple of moduli $(m_i)$ of
the cylinders.  This will give us more control over the moduli then provided by
Theorem~\ref{thm:moduli_bound}.

Applying the \Teichmuller flow to $(X, \omega)$ scales the widths $(r_i)$ by a constant multiple, so
generally they should be regarded as a point of $\proj^{g-1}(F)$; however, there is
an almost canonical choice for their scale.

We say that $(X, \omega)$ is \emph{normalized} if its cylinder heights and widths satisfy the
following properties:
\begin{itemize}
\item each $r_i$ and $s_j$ either belong to $F$, or for some real quadratic extension $K =
  F(\alpha)$, they belong to $W = \Ker \Tr_{K/F}$ (using in either case an implicit real embedding);
\item $s_i/r_i\in\ratls^+$ for each $i$; and
\item for every pair of oriented closed curves $\alpha, \beta$ on $X$, with $\alpha$ horizontal, we
  have
  \begin{equation}
    \label{eq:13}
    \alpha \cdot \beta = \left\langle \int_\alpha \omega , \Im \int_\beta \omega \right\rangle,
  \end{equation}
  where $\alpha\cdot\beta$ is the intersection pairing on homology, and
  $$\langle x, y\rangle = \Tr_{F/\ratls}(xy).$$
  (Note that $xy\in F$ whenever $x,y\in W$.)
\end{itemize}

\begin{prop}
  \label{prop:normalization_exists}
  There is a diagonal matrix $D\in \GLtwoRplus$ such that $D \cdot (X, \omega)$ is normalized.  Such
  a $D$ is unique up to multiplication by $\left(
    \begin{smallmatrix}
      c & 0\\
      0 & c^{-1}
    \end{smallmatrix}
    \right)$ for any $c$ with $c^2\in \ratls$. 
\end{prop}

\begin{proof}
  Since $(X, \omega)$ is an eigenform for real multiplication, $H_1(X; \ratls)$ is equipped with the
  structure of an $F$ vector space.  We let $M\subset H_1(X; \ratls)$ be the subspace spanned by the
  horizontal cylinders of $(X, \omega)$, and $N = H_1(X; \ratls) / M$.  As real multiplication
  preserves the span of short curves, $M$ and $N$ are  $F$ vector spaces (see for example
  \cite[Proposition~5.5]{BaMo12}).  As $\omega\colon M \to \reals$ and $\Im \omega\colon N \to
  \reals$ are $F$-linear and spanned by the $r_i$ and $s_j$ respectively, we may suppose they belong
  to $F$ after applying an appropriate diagonal element.

  We now claim we may choose uniquely a diagonal matrix $\left(
  \begin{smallmatrix}
    \mu & 0 \\
    0 & 1
  \end{smallmatrix}
  \right)$ so that \eqref{eq:13} holds.  To see this, note that both pairings in \eqref{eq:13}
  are bilinear pairings $M \times N \to \ratls$ compatible with the $F$ vector space structure in
  the sense that $(\lambda \cdot x, y) = (x, \lambda y)$.  The space of such pairings is a
  one-dimensional vector space over $F$, so a unique such $\mu$ exists.

  Making this normalization, we still have the freedom to apply a diagonal matrix $\left(
  \begin{smallmatrix}
    \alpha & 0 \\
    0 & \alpha^{-1}
  \end{smallmatrix}
  \right)$.  Even if $\alpha\not\in F$, the trace pairing in \eqref{eq:13} is still defined and
  independent of $\alpha$ as
  $(\alpha x) (\alpha^{-1}y) = xy\in F$.  We take $\alpha = \sqrt{s_1/r_1}$.  Note that as $(X,
  \omega)$ is Veech, the moduli of its cylinders have rational ratios, so
  \begin{equation*}
    \frac{\alpha^{-1}s_i}{\alpha r_i} = \frac{r_1 s_i}{r_i s_1} =
    \frac{\modulus(C_i)}{\modulus(C_1)} \in \ratls^+,
  \end{equation*}
  which is the second required property of our normalization.  Finally note that the $\alpha r_i$
  and $\alpha^{-1} s_i$ either lie in $F$ or $W\subset K=F(\alpha)$ depending on whether or not $s_1/r_1$
  has a square root in $F$.

  Uniqueness follows, since applying $\left(
  \begin{smallmatrix}
    c & 0 \\
    0 & c^{-1}
  \end{smallmatrix}
  \right)$ to $(X, \omega)$ preserves the rationality condition if and only if $c^2\in \ratls$.  
\end{proof}

\begin{rem}
  Proposition~\ref{prop:normalization_exists} is a generalization of some basic observations from
  \cite{BaMo12} in the case where $(X, \omega)$ has $g$ horizontal cylinders, namely
  that $(X, \omega)$  may be normalized so that the heights of the cylinders are dual to the
  widths with respect to the trace pairing, and the vectors $\theta(r_i)$ are orthogonal in
  $\reals^g$. 
\end{rem}

Taking such a normalization, let $\iota_1, \ldots, \iota_g$ be the $g$ real embeddings of $F$, and
if we have taken an extension $K$, extend each $\iota_i$ to the real embedding of $K$ for which
$\iota_i(\alpha) > 0$.  For $x\in F$ or $K$, we write $x^{(i)}$ for $\iota_i(x)$ and $\theta$ for
the associated embedding of $F$ or $W$ in $\reals^g$.  The vectors $\theta(r_i)$ are then a basis of
$\reals^g$.  While our normalization was not unique, a different choice of normalization only
changes this basis by a constant multiple, so the shape of this basis of $\reals^g$ does not depend
on any choices.

In the sequel, it will not matter whether we have passed for a quadratic extension of $F$ or not, so
we will simply write $W$ for either $F$ or $\Ker \Tr_{K/F}$.

Now suppose that $(X, \omega)$ is normalized as above.  Let $\Gamma$ be the corresponding dual
graph.  We orient the edges of $\Gamma$ as in \S\ref{sec:torsion-moduli}.  Assign to each edge $e$
of $\Gamma$ the weight $r_e\in W$ to be the width of the corresponding cylinder of $(X, \omega)$,
as well as the weight $m_e = s_e/r_e \in \ratls$, the modulus of this cylinder.
At each vertex of $\Gamma$, the sum of the incoming weights equals the sum of the outgoing weights
(we called such an object a $W$-weighted graph in \cite{BaMo12}), so we may define
\begin{equation}
  \label{eq:graph_homology}
  \rho_\br = \sum_e r_e [e] \in H_1(\Gamma; W) \qtq{and} \sigma_\br = \sum_e m_e r_e [e^*] \in H^1(\Gamma; W).
\end{equation}

Since the $r_e$ span $W$, $\rho$ defines an isomorphism $\rho_\br^* \colon H^1(\Gamma; \ratls) \to W$.
From this point of view, we may interpret \eqref{eq:13} as saying for all $x\in H_1(\Gamma; \ratls)$
and $y \in H^1(\Gamma; \ratls)$, 
\begin{equation}
  \label{eq:14}
  (x, y) = \langle \sigma_\br^*(x), \rho_\br^*(y)\rangle,
\end{equation}
where the right pairing is the trace pairing on $W$ defined above and the left is the canonical
pairing between homology and cohomology.  Since $\rho_\br^*$ is an
isomorphism, it follows that $\sigma_\br^*$ is an isomorphism as well.

Now let $\Gamma = \Gamma_1 \cup \ldots\cup \Gamma_k$ be its decomposition into blocks, and define
$B_i\subset W$ to be the span of the vectors $r_e$ with $e$ in $\Gamma_i$.

\begin{prop}
  \label{prop:blocks_orthogonal}
  The subspaces $B_1, \ldots, B_k$ span $W$ and are orthogonal.
\end{prop}

\begin{proof}
  First, since $H^1(\Gamma; \ratls) = \oplus H^1(\Gamma_i; \ratls)$, $\rho_\br^*$ is an isomorphism, and
  $B_i$ is the image of $\rho_\br^*$ restricted to $H^1(\Gamma_i; \ratls)$, we see that the $B_i$ span
  $W$.

  Now let $e$ be an edge in $\Gamma_i$ and $s\in B_j$ for $i\neq j$.  We claim that $s =
  \sigma_\br^*(\gamma)$ for some $\gamma \in H_1(\Gamma_j; \ratls)$.  Assuming this claim, we obtain
  \begin{equation*}
    \langle r_e, s\rangle = ( [e^*], \gamma) = 0,
  \end{equation*}
  since $\gamma_i$ and $e$ are in different blocks.

  To see the claim, note that $\sigma_\br^* H_1(\Gamma_j; \ratls) \subset B_j$, since the $m_e$ are
  rational.  As they have the same dimension, these spaces are in fact equal, so we may find the
  desired $\gamma$.
\end{proof}

Consider now a single block $\Gamma_0$ in $\Gamma$.  We label its edges $e_1, \ldots, e_n$, with
edge $e_i$ having weight $r_i$ and modulus $m_i$.  We define $Q$ to be its matrix of traces
$$(Q_{ij}) = ( \langle r_i, r_j \rangle ),$$
which determines the collection of vectors $\theta(r_i)$ up to orthogonal rotation of $\reals^g$.
\begin{prop}
  \label{prop:moduli_determine_vectors}
  The matrix of traces $Q$ is determined by the moduli $m_1, \ldots, m_n$ and the graph $\gamma_0$.
  Moreover, $Q$ is inversely proportional to the $m_i$ in that if the $m_i$ are multiplied by some
  $q\in \ratls$, the matrix $Q$ is divided by $q$.

  Moreover, $h(Q) \leq C_1 h(m_1, \ldots, m_n) + C_2$ for constants $C_1$ and $C_2$ which depend
  only on the graph $\Gamma_0$.   
\end{prop}

\begin{proof}
  Choose a spanning tree $T\subset \Gamma_0$, and reorder the edges so that $e_1, \ldots, e_m$
  belong to $\Gamma_0 \setminus T$ and $e_{m+1}, \ldots, e_n$ belong to $T$.  

  Suppose we are given $r_1, \ldots, r_m\in W$.  There is then a unique choice of $r_{m+1},
  \ldots. r_n$ such that $\bdry \sum r_i [e_i] = 0$.  Namely, let $\bdry \colon H_1(\Gamma_0; V) \to
  \tilde{H}_0(\Gamma_0; V)$ (where $V$ is the set of vertices of $\Gamma_0$) be given by the block
  matrix $(A, B)$ (using the basis $[e_i]$ of $H_1(\Gamma_0; W)$).  Since $T$ is a spanning tree,
  the $n-m \times n-m$ matrix $B$ is invertible.  Let $K = (k_{ij}) = -B^{-1}A$ and $L =\left(
  \begin{smallmatrix}
    I \\
    K
  \end{smallmatrix}\right).$  Then given  $r_1, \ldots, r_m$ in $W$, we set
  \begin{equation*}
    r_{i} = \sum_{j=1}^n l_{ij}r_j.
  \end{equation*}
  With $\br = (r_1, \ldots, r_n)$, this is the unique vector extending $(r_1, \ldots, r_m)$ such
  that $\bdry \rho_\br = 0$.
  
  Let $M$ be the diagonal matrix with entries $m_1, \ldots, m_n$ on the diagonal.

  For $1\leq i \leq m$, let $\gamma_i \in H_1(\Gamma_0)$ be the unique circuit which crosses $e_i$
  once, positively oriented, then completes the circuit by the unique path in $T$ joining the
  endpoints.  We write $ \gamma_i = \sum_i n_{ij} e_j$, and let $N = (n_{ij})$.   Finally, let $P = NML$.
  This matrix $P$ has been constructed so that if $s_i = \sigma_\br(\gamma_i)$ are the periods of
  the $\gamma_i$,  $R$ is the $m\times g$ matrix defined by
  $(R_{ij}) = (r_i^{(j)})$, and $S$ is the  matrix defined by $(S_{ij}) = (s_i^{(j)})$, then  $S = PR$.

  Now, for $1 \leq i,j \leq m$, each  $([e_i^*], \gamma_j) = \delta_{ij}$.  Then by \eqref{eq:14},
  we have $\langle r_i, s_j\rangle = \delta_{ij}$, which means that $SR^t = I$, so $RR^t = P^{-1}$.
  The full matrix of traces is then
  \begin{equation*}
    Q =
    \begin{pmatrix}
      P^{-1} & P^{-1} L^t \\
      L P^{-1} & LP^{-1} L^t
    \end{pmatrix},
  \end{equation*}
  which depends only on $\Gamma_0$ and the $m_i$, as desired.

  To prove the height bound, we use the inequalities,
  \begin{align*}
    h(MN) &\leq h(M) + h(N) + \log n, \\
    h(M^{-1}) &\leq (n-1) h(M) + \log(n-1)!,
  \end{align*}
  where in either case $n$ is the number of columns of $M$.  Since, $A, B,$ and $N$ consists only of
  zeros and ones, they have height $0$, and the bound for $h(Q)$ follows.
\end{proof}

\begin{rem}
  Keeping careful track of constants, one obtains the explicit bound,
  \begin{equation*}
    h(Q) \leq (e-v) h(m_1, \ldots, m_e) + (e-v) \log e^2(v-e)! + \log(e-v)!(v-1)!^2(e-v+1)^2, 
  \end{equation*}
  where $v$ and $e$ are the number of vertices and edges of $\Gamma_0$.
\end{rem}

Taken together, Theorem~\ref{thm:moduli_bound} and Propositions~\ref{prop:blocks_orthogonal} and
\ref{prop:moduli_determine_vectors} gives strong constraints for the geometry of the cylinder widths
$v_i=\theta(r_i)$ in a normalized periodic direction in terms of a given $N$ which is a multiple of the
torsion orders of $(X, \omega)$.  By Proposition~\ref{prop:blocks_orthogonal}, the $v_i$ lie
in orthogonal blocks $B_i \subset\reals^g$ corresponding to the partition of the dual graph $\Gamma$ into
blocks.  For each block $\Gamma_i$ of $\Gamma$, the torsion order $N$ determines its vector of
moduli $[m_1: \ldots: m_{k_i}]$ up to finitely many choices by Theorem~\ref{thm:moduli_bound}.
Proposition~\ref{prop:moduli_determine_vectors} then tells us that this vector of moduli determines
the vectors $v_j$ in $B_i$ up to scaling and orthogonal rotation. 

While Theorem~\ref{thm:moduli_bound} tells that a bound for torsion orders gives the tuple of moduli of a
periodic direction up to scaling in each block, these torsion orders don't give any information
about the scale of these tuples of moduli in different blocks.  Proposition~\ref{prop:moduli_determine_vectors}
allows us to control this scale using knowledge of the residues of the stable form where two blocks intersect.

\begin{prop}
  \label{prop:real_ratios_bounded}
  Given a pantsless-finite collection of algebraically primitive \Teichmuller curves, there is a
  constant $M$ such that for any periodic direction of a surface in this collection, the ratio of
  moduli of any two cylinders is at most $M$.
\end{prop}

\begin{proof}
  By Proposition~\ref{prop:pf_torsionbound}, the torsion orders of any curve in this collections
  divide some constant $N$.

  Let $(X, \omega)$ be a Veech surface with periodic horizontal direction which lies on one of the
  \Teichmuller curves in our collection, let $\Gamma$ be its dual graph, and let $B,
  B'\subset\Gamma$ be adjacent blocks containing edges $e_1, \ldots, e_k$ and $e'_1, \ldots, e'_l$,
  ordered so that $e_1$ and $e_1'$ meet at the common separating vertex $v$.

  We take $(X, \omega)$ to be normalized as above, which means in particular that the horizontal
  cylinders have rational moduli.  We write $qm_1, \ldots, qm_k$ and $q' m_1', \ldots, q' m_l'$ for
  the moduli of the cylinders corresponding to $e_i$ and $e_i'$, with $q, q'\in\ratls$ chosen so
  that the tuples of $m_i$ and $m_i'$ are both relatively prime integers.  We let $r_i$ and $r_i'$
  denote the corresponding cylinder widths.

  By Theorem~\ref{thm:moduli_bound}, the $m_i$ and $m_i'$ are determined up to finitely many choices
  by $N$, so we may take them to be fixed.  It then suffices to bound $q'/q$.

  Since $\Gamma$ has no separating edges, there must be at least four edges incident to $v$, so the
  corresponding component of the limiting stable form is not pants, so it is known up to finitely
  many choices.  In particular, we may take $\lambda = r_1/r_1'$  to be fixed.  It follows that
  \begin{equation}
    \label{eq:1}
    \Tr(r_1^2) = \Tr(\lambda^2 {r_1'}^2) \leq \|\lambda^2\|_\infty \Tr({r_1'}^2),
  \end{equation}
  Where $\|\lambda\|_\infty = \max_i |\lambda^{(i)}|$.

  By Proposition~\ref{prop:moduli_determine_vectors}, we have
  \begin{equation}
    \label{eq:15}
    \Tr(r_1^2) = \frac{t}{q} \qtq{and} \Tr({r_1'}^2) = \frac{t'}{q'},
  \end{equation}
  where $t \in\ratls$ is determined by the corresponding moduli $m_i$ and graph $B$, and likewise
  for $t'$.  Combining \eqref{eq:1} and \eqref{eq:15}, we obtain
  \begin{equation*}
    \frac{q'}{q} \leq \|\lambda^2\|_\infty \frac{t'}{t}.
    \qedhere
  \end{equation*}
\end{proof}

\begin{prop}
  \label{prop:widths_bounded}
  Given a pantsless-finite collection of algebraically primitive \Teichmuller curves, there is a
  constant $K$ such that for any periodic direction of a surface in this collection, the ratio of
  widths of any two cylinders or saddle connections is at most $K$.
\end{prop}

\begin{proof}
  Consider a surface $(X, \omega)$ on one of these \Teichmuller curves, and let $C_1$ and $C_2$ be
  adjacent horizontal cylinders, in the sense that their boundaries share a saddle connection.
  Suppose that one of the cylinders, say $C_1$, has distinct zeros of $\omega$ on its upper and lower boundaries (which is true if
  $C_1$ does not correspond to a loop in the dual graph $\Gamma$).  We claim that there is a
  universal constant $L_1$ such that $w(C_2) \leq L_1 w(C_1)$.

  To see this, first shear the surface so that a vertical saddle connection $\gamma$ joins the zeros
  in $\bdry C_1$.  Let $\delta$ be a parallel geodesic which crosses $C_1$ and $C_2$.  Since the
  surface is Veech, $\delta$ is closed.  These curves limit to curves on a stable curve at infinity
  which is not pants, since it contains more than one zero.  By pantsless finiteness, $\ell(\delta)
  \leq C \ell(\gamma)$ for a universal constant $C$.  We then have,
  \begin{equation*}
    \modulus(C_2) w(C_2) = h(C_2) \leq \ell(\delta) \leq C \ell(\gamma) = C h(C_1) = C \modulus(C_1) w(C_1).
  \end{equation*}
  Since the ratios of moduli are bounded by $M$ from Proposition~\ref{prop:real_ratios_bounded}, we
  obtain $w(C_2) \leq MC w(C_1)$.

  If one of the cylinders, say $C_1$, actually corresponds to a loop, then its adjacent vertex 
  is a component of the limiting stable curve which is not a pair of pants.  Pantsless finiteness 
  then implies that there exists a universal constant $L_2$ such that 
  $w(C_2) \leq L_2 w(C_1)$ and $w(C_1) \leq L_2 w(C_2)$.

  Now consider the graph whose vertices are horizontal cylinders of  $(X, \omega)$ and
  whose edges are cylinders that share a saddle connection. Given any two vertices $C, C'$ 
  of this graph, we may find a path  $C=C_1, \ldots, C_m = C'$, since the surface $X$ is
  connected. We can find such a path of length at most the number of horizontal saddle
  connections, a constant $n$ depending only on the stratum. Using the previous bounds
  we conclude that $w(C_m) \leq \max\{L_1,L_2\}^{n-1} w(C_1)$.

  It then follows from pantsless finiteness that ratios of lengths of saddle connections are bounded
  as well, since any saddle connection has a nearby closed geodesic, and the ratio of their lengths
  is bounded above and below uniformly.
\end{proof}

Recall from \cite{smillieweiss10} that a \emph{triangle} in $(X, \omega)$ is the image of a triangle
in the plane under an affine map to $(X, \omega)$ which sends the vertices to zeros of $\omega$, and
which is injective except possibly at the vertices.  They showed that $(X, \omega)$ is Veech if the
set of areas of its triangles is bounded away from $0$.

Let $\mathcal{S}_1$ be the locus of unit area forms in our stratum, and
$\NST(\alpha)\subset\mathcal{S}_1$ to be the set of such surfaces which have no triangles of area
less than $\alpha$.  Our proof of finiteness will use the following theorem.

\begin{theorem}[\cite{smillieweiss10}]
  \label{thm:NST}
The set  $\NST(\alpha)$ consists of finitely many \Teichmuller curves.
\end{theorem}

\begin{rem}
  Smillie and Weiss actually proved a much stronger finiteness theorem that we do not need, namely that the union
  of these $\NST(\alpha)$ over all strata in any genus is finite.
\end{rem}

\begin{proof}[Proof of Theorem~\ref{thm:pantsless_finiteness}]
  Let $(X, \omega)$ be a surface in our collection of \Teichmuller curves, normalized to have unit
  area, and let $T\subset (X, \omega)$ be a triangle.  We may rotate $\omega$ so that the base of
  $T$ is horizontal, hence this horizontal direction is periodic.  Then apply a diagonal matrix so
  that $(X, \omega)$ has a horizontal cylinder of unit width.   By
  Propositions~\ref{prop:real_ratios_bounded} and \ref{prop:widths_bounded}, all horizontal
  cylinders have moduli and widths bounded above and below, as do the horizontal saddle
  connections.  Their heights are then bounded as well.  It follows that the area of $T$ is
  uniformly bounded below, and by Theorem~\ref{thm:NST}, our collection of \Teichmuller curves is finite. 
\end{proof}


%% file: principalstrata.tex
\section{Bounding torsion and the principal stratum}
\label{sec:pres-tors-princ}

In this section, we aim to apply the finiteness criterion, Theorem~\ref{thm:pantsless_finiteness}, 
to obtain finiteness for the principal stratum in genus three and prepare the grounds
for the discussion of the remaining strata in the next section.
In order to verify the pantsless-finite hypothesis, we will use the torsion condition together with
\emph{a priori} bounds on the torsion orders and height bounds from diophantine geometry to control
the possible stable forms arising as cusps of \Teichmuller curves generated by forms with many
zeros.

More precisely, given an irreducible component $(X_\infty, \omega_\infty)$ of the stable
form over a cusp of an \APTC, if $\omega_\infty$ has multiple zeros, then the torsion condition may
be interpreted as saying that certain cross-ratios involving these zeros and the cusps of $X_\infty$
are roots of unity.  Torsion order bounds, height bounds and information about the degrees of the
residues of $\omega_\infty$ often allow us to conclude that there are only finitely many
possibilities for $(X_\infty, \omega_\infty)$.  To formalize this, we introduce now the notion of a
form which is determined by torsion.
\par
We say that a meromorphic one-form on $\proj^1$ is {\em a form of type $(n; m_1,\ldots,m_k)$} if it has $n$
 poles, all of which are simple, and $k$ zeros of orders $m_1,\ldots,m_k$.  We denote the
poles $x_1, \ldots, x_n$ and the zeros $z_1, \ldots, z_k$.  For example, a pair of pants is a 
form of type $(3;1)$.
\par
\begin{definition}
  \label{def:determined_by_torsion}
  A one-form on $\proj^1$ of type $(n; m_1,\ldots,m_k)$ {\em is determined by torsion} if for every
  $N \in \IN$ and every $g \in \IN$ there are only a finite number of forms $\omega$ of this type,
  up to the action of $\Aut(\proj^1)$ and constant multiple, that satisfy the following condition.
  There exists a partition of $\{1,\ldots,n\}$ into subsets $S_j, j\in J$, none of which consists of
  a single element, such that
  \begin{compactitem} 
  \item[i)] for each part $S_j$ of the partition $\sum_{i \in S_j} \Resi_{x_i}(\omega) = 0$,
  \item[ii)] the ratios of residues $\Resi_{x_i}(\omega) / \Resi_{x_1}(\omega)$ for $i=1,\ldots,n$
    are elements of a totally real number field of degree $g$ whose $\ratls$-span has dimension $n-
    |J|$, and
  \item[iii)] for all $a \neq b$ the cross-ratio $[z_a,z_b, x_{i_1}, x_{i_2}]$ is a root of unity of
    order dividing $N$ whenever there exists a part $S_j$ containing both $i_1$ and $i_2$.
  \end{compactitem} 
\end{definition}

These conditions are motivated by the constraints imposed on a component of a limiting stable form
of a \Teichmuller curve.

\begin{prop}
  \label{prop:boundary_conditions}
  Suppose $(X_\infty, \omega_\infty)$ is a stable form lying over a boundary point of an
  algebraically primitive \Teichmuller curve, and $Y$ is an irreducible component of $X_\infty$.
  Then there is a partition of the nodes of $Y$ satisfying the above conditions.
\end{prop}

\begin{proof}
  Consider the dual graph $\Gamma$ of $X_\infty$ with edges labeled by the corresponding residues of
  $\omega_\infty$ as in \S\ref{sec:genfinthm}, and consider the vertex $v$ corresponding to $Y$.
  We partition the set of edges $E(v)$ incident to $v$ according to the component of $\Gamma
  \setminus \{v\}$ in which they lie.  As $X_\infty$ has no separating edges, each part of this
  partition has at least two elements.

  Condition i) is then a consequence of the residue theorem.  The condition iii) rephrases that the
  difference between any two zeros is a torsion section with torsion order divisible by $N$, as
  stated in Lemma~\ref{le:conseqoftorsion}.

  To obtain ii), consider the set $S$ of edges of $\Gamma$ obtained by deleting from $E(v)$ one edge
  from each component of $\Gamma \setminus \{v\}$.  We wish to show that the corresponding residues
  are linearly independent over $\ratls$.  Recall from \S\ref{sec:genfinthm} that assigning the
  residues of $\omega_\infty$ to the edges of $\Gamma$ determines a homology class $$\rho_\br =
  \sum_e r_e [e] \in H_1(\Gamma; F)$$ (after normalizing the form by dividing by
  $\Resi_{x_1}(\omega_\infty)$ so that the residues belong to the trace field $F$).  The induced map
  $\rho_\br^*\colon H^1(\Gamma; \ratls) \to F$ is an isomorphism, so it suffices to show that the
  cohomology classes $\{[e^*]: e\in S\}$ are linearly independent.  This follows from the fact that
  for each $e\in S$ there is a circuit of $\Gamma$ which crosses $e$ and no other edges in $S$.
\end{proof}

Forms with sufficiently many zeros are determined by torsion.  We state below the relevant results,
which we will prove in \S\ref{sec:prooffewzeros} and \S\ref{sec:proof4zeros}.
\par
\begin{prop} \label{prop:3detbytorsion}
A  form of type $(n; m_1,\ldots,m_k)$ having three or more 
zeros, i.e.\ $k \geq 3$, is determined by torsion.
\end{prop}
\par
\begin{prop} \label{prop:twozerosdetbytorsion}
A  form of type $(4; 1,1)$ is determined by torsion.
\end{prop}

\begin{proof}[Proof of Theorem~\ref{thm:torsion_implies_finite}]
  Suppose we have an \emph{a priori} bound for torsion orders of \APTCs in $\omoduli(1^{2g-2})$.
  Propositions~\ref{prop:3detbytorsion} and \ref{prop:twozerosdetbytorsion} imply that there are
  only finitely many possibilities for any non-pants component of a stable curve at a cusp of an
  \APTC in this stratum.  Theorem~\ref{thm:pantsless_finiteness} then implies finiteness.
\end{proof}
\par
To apply these results, we need \emph{a priori} bounds on torsion orders.  When an irreducible
component $(X_\infty, \omega_\infty)$ of the stable form over a cusp contains two zeros, one can
often bound the torsion order of the difference of these two zeros by appealing to Laurent's theorem
\cite{laurent} on torsion points in a subvariety of $\Gm$.  Whether this works depends on the
combinatorial type of $(X_\infty, \omega_\infty)$.  To formalize this, we introduce the notion of
prescribing torsion.

\begin{definition}
  \label{def:prescribes_torsion}
  A form of type $(n; m_1,\ldots,m_k)$ {\em prescribes torsion} if for every $g \in \IN$
  there exists $N_0 \in \IN$ such that for every partition of $\{1,\ldots,n\}$ into subsets $S_j,
  j\in J$, none of which consists of a single element, and for every form $\omega$ on $\PP^1$, such
  that
  \begin{compactitem}
  \item[i)] for each part $S_j$ of the partition $\sum_{i \in S_j} \Resi_{x_i}(\omega) = 0$,
  \item[ii)] the ratios of residues $\Resi_{x_i}(\omega) / \Resi_{x_1}(\omega)$ for $i=1,\ldots,n$
    are elements of a totally real number field of degree $g$ whose $\ratls$-span has dimension $n -
    |J|$, and
  \item[iii)] for all $a \neq b$ the cross-ratio $[z_a,z_b, x_{i_1}, x_{i_2}]$ is a root of unity
    whenever $i_1$ and $i_2$ belong to the same part of the partition,
  \end{compactitem} 
  then all of the roots of unity appearing in iii) have order dividing $N_0$.
\end{definition}
We will prove the following results on which types of forms prescribe torsion.
\par
\begin{prop} \label{prop:3presctorsion}
A  form of type $(6; 1,1,1,1)$ prescribes torsion.
\end{prop}
\par
It might well be true -- and would suffice to prove finiteness
in the principal stratum in all genera -- that this proposition
holds for all components of type $(2k; 1^{2k-2})$. The following
lemma is also part of checking the hypothesis for pantsless-finiteness
and works for all $k$.
\par
\begin{lemma} \label{lemma:3zerosprescres} If a form of type $(n; m_1,\ldots,m_k)$ with three or
  more zeros, i.e.\ $k \geq 3$, occurs as an irreducible component of a stable curve lying over a
  boundary point of an algebraically primitive \Teichmuller curve, then $n$ is even. Moreover, there
  is only a finite number (depending on $n$) of tuples of residues (up to constant multiple) that
  occur for such boundary points.
\end{lemma}
\par
\begin{prop} \label{prop:twozerospresctorsion}
A  form of type $(4; 1,1)$ prescribes torsion.
\end{prop}
\par
From these statements we will deduce the  main theorem in the case of the principal stratum. 
\par
\begin{proof}[Proof of Theorem~\ref{thm:intromainfin}, case ${\omoduli[3]}(1,1,1,1)$]
  Consider a Veech surface generating an algebraically primitive \Teichmuller curve in this stratum,
  and suppose that two of its zeros $z_i$ and $z_j$ can be joined by a saddle connection of slope
  $\theta$.  Consider the stable curve $X_\infty$ obtained by applying the \Teichmuller geodesic
  flow in the direction $\theta$.  The two zeros $z_i$ and $z_j$ will limit to the same component of
  $X_\infty$.  By Lemma~\ref{lemma:3zerosprescres}, there are two possibilities: either $X_\infty$
  is irreducible, or it has two components, each containing two zeros.  The
  Propositions~\ref{prop:3presctorsion} and~\ref{prop:twozerospresctorsion} and
  Lemma~\ref{le:conseqoftorsion} then bound in either case the torsion order of the corresponding difference of sections
  $s_i -s_j$ by a universal constant.
  \par
  More generally, any two zeros can be joined by a finite chain of saddle connections, where the
  length of the chain is at most the number of zeros.  Consequently, the torsion order of any
  difference of sections $s_i -s_j$ is bounded by a universal constant.
  \par
  Theorem~\ref{thm:torsion_implies_finite} then implies finiteness of \APTCs in this stratum.
\end{proof}

\subsection{Forms with few zeros} 
\label{sec:prooffewzeros}

\begin{proof}[Proof of Proposition~\ref{prop:twozerosdetbytorsion} and 
Proposition~\ref{prop:twozerospresctorsion}]
We consider a form $\omega_\infty$ of type $(4;1,1)$ and use $\Aut(\proj^1)$ to
normalize the  form $\omega_\infty$ to have its two zeros
at $0$ and $\infty$. The partition of $\{1,2,3,4\}$ may have one or two parts.
\par 
In the case of two parts, 
$$ \omega_\infty = \left(\frac{r_1}{z-x_1} - \frac{r_1}{z-x_2} + \frac{r_2}{z-u_1}
-\frac{r_2}{z-u_2}\right) dz = \frac{C z \,dz}{\prod_{i=1}^2(z-x_i)\prod_{i=1}^2(z-u_i)},$$
where $\{x_1,x_2\}, \{u_1,u_2\}$ is the partition of the nodes 
and where $C \in \cx$ is some constant.
\par
The torsion conditions imply that there exist two roots
of unity $\zeta_x$ and $\zeta_u$ such that
$$ x_2 = \zeta^2_x x_1, \quad u_2 = \zeta^2_u u_1.$$
The resulting equation for the $z^2$ and constant terms
of the numerator of $\omega_\infty$ to vanish imply, with the normalization
$x=1$ and $r_1 = 1$, that 
$$ r_2 = \frac{\zeta_u}{\zeta_x}
\frac{1-\zeta_x^2}{1-\zeta_u^2} = \frac
{\ol{\zeta_x}-\zeta_x}{\ol{\zeta_u}-\zeta_u} \quad \text{and}
\quad u_1 = - \frac{\zeta_x}{\zeta_u}. $$
This is the situation of where McMullen's theorem on ratio of sines \cite{mcmullentor} applies.
As a consequence, since the $r_i$ belong to a cubic field, there are only a finite number of choices 
for the torsion orders of the roots of unity and a finite number of possibilities for the
$r_i$. This concludes the proof of both propositions in this case.
\par
In the case of one part 
\begin{equation} \label{eq:om1111case2}
 \omega_\infty = \left(\sum_{i=1}^4 \frac{r_i}{z-u_i} \right) dz = \frac{C  z\, dz}
{\prod_{i=1}^4 (z-u_i) },
\end{equation}
where $r_4 = -(r_1+r_2+r_3)$, and $r_1, r_2, r_3$ are linearly independent over $\ratls$. The
torsion condition now implies that the pairwise ratios of the $u_i$ are roots of unity. We may
normalize $u_1 = 1$ and $u_i = \zeta_i$ for some $N$th roots of unity $\zeta_i$.  There are rational
functions $f_i$ such that $r_i = f_i(u_1, \ldots, u_4)$, so if the roots of unity $\zeta_i$ are
known, then so is $\omega_\infty$.  This proves Proposition~\ref{prop:twozerosdetbytorsion}.
\par
To prove Proposition~\ref{prop:twozerospresctorsion}, we need to show that in the same
situation, there are only finitely many choices for the roots of unity $\zeta_i$ for which the
residues $r_i$ belong to a field of bounded degree.

The degrees of the $f_i$ are independent of $g$. The image of a tuple of roots of unity,
which are of height zero, is thus of bounded height by \eqref{eq:htupperbd}.  Since the ratios of
residues moreover lie in a field of degree bounded by $g$, Northcott's theorem implies that for
each $g$ there are only a finite number of possible residue tuples up to scale.
\par
For each such tuple, we have to check that there are only finitely many roots of unity that give
rise to this residue.  For fixed residues $r_i$, the possible $\zeta_i$
belong to the curve cut out by the equations,
\begin{equation} \label{eq:solveforzeta}
 \sum_{i=1}^4 r_i \zeta_i = 0 \quad \text{and} \quad 
\sum_{i=1}^4 r_i \zeta^{-1}_i  = 0,
\end{equation}
with the normalization $\zeta_1=1$.  If there were an infinite 
number of solutions in roots of unity, they would lie in a translate
of a torus by a torsion point. That is there exist $a_2,a_3,a_4 \in \IZ,$
not all zero, and roots of unity $\eta_i$, such that
$\zeta_i = \eta_i t^{a_i}$ for $i=2,3,4$  is a solution to
\eqref{eq:solveforzeta} for all $t$. Considering the limit $t \to 0$
and $t \to \infty$ implies that the highest and lowest exponent
must not be isolated. We may thus renumber the terms such that $a_2=0$
and $a_3 = a_4\neq 0$. Substitution into \eqref{eq:solveforzeta} implies that
$r_3 \eta_3 + r_4 \eta_4 =0$, hence $r_3/r_4 = \pm 1$. This contradicts
$\IQ$-linear independence.
\end{proof}

\subsection{Forms with many zeros} 
\label{sec:proof4zeros}

For forms with more zeros, the basic idea of the proofs is similar, but the extra dimensions
involved create significant difficulties.

\begin{lemma}
  Let $\omega_\infty$ be a form of type $(n;m_1,\ldots,m_k)$ with $k \geq 3$.  Suppose
  that $\omega_\infty$ satisfies the conditions i), ii) and iii) in
  Definition~\ref{def:prescribes_torsion}. Then $n$ is even and $|S_j|=2$ for all parts of
  the partition.
  \par
  Moreover, if we normalize $z_1=\infty$ and $z_2 = 0$ and renumber the poles such that $S_j = \{j,
  \ell+j\}$, where $\ell = n/2$, then $x_j$ is the complex conjugate of $x_{\ell+j}$.
\end{lemma}
\par
We will write subsequently $y_j$ instead of $x_{\ell+j}$. We use the normalization $z_1=\infty$ and $z_2 = 0$ 
and $z_3 = 1$ and we call this {\em standard normalization} in the rest of this section. 
\par
\begin{proof} 
Fix a part $S_j$. By condition iii) all the points $x_i$ for $i \in S_j$
lie on the same circle around zero and they also lie on the same circle around one.
Since the two circles intersect in precisely two points that
are complex conjugate, all of the claims follow.
\end{proof}
\par
Such a form $\omega_\infty$ is determined, up to scale, 
by the location of its zeros and poles, i.e.\ by a point
$P=(z_1,\ldots,z_k,x_1,\ldots,x_\ell,y_1,\ldots,y_\ell)$  in $\moduli[0,k+n]$.
Coordinates on this moduli space are cross-ratios, among which we
select
$$R_{abj} = [z_a,z_b,y_j,x_j], \quad 1\leq a < b \leq k, \quad 1\leq j\leq \ell,$$
since these will be roots of unity for a form satisfying iii).

We define $\IA = \cx^{\ell k(k-1)/2}$ with coordinates $t_{abj}$ ($a,b,j$ as above), and we define
$\IA'\subset \IA$ to be the complement of the hyperplanes of the form
$V(t_{abj})$, $V(t_{abk} -1)$, $V(t_{abj} - t_{a'bj})$, and $V(t_{abj} - t_{ab'j}).$  The
cross-ratios $R_{abj}$ then define a morphism $\CR\colon  \moduli[0,k+n] \to \IA'$.

Since $R_{abj}R_{bcj} = R_{acj}$, we may without loss of information restrict our
attention to the cross-ratios where the first zero is fixed to be $z_1$, which we denote by
$$R_{ij} = [z_1,z_i,y_j,x_j], \quad 1 < i \leq k, \quad 1\leq j\leq \ell.$$
We define $\IA_\red = \cx^{\ell(k-1)}$ with coordinates $t_{ij}$ (with $i,j$ as above), and we
define $\IA_\red' \subset \IA_\red$ to be complement of the hyperplanes of the form $V(t_{ij})$,
$V(t_{ij} -1)$, and $V(t_{ij} - t_{i'j})$.  The cross-ratios $R_{ij}$ then define a  morphism $\CR^\red\colon  \moduli[0,k+n] \to \IA_\red'$.

Finally, we define $\IA_\min = \cx^{n+k-3}$ with coordinates $t_{ij}$ for $i=2,3$ and $1 \leq j \leq
\ell$ together with $t_{i1}$ for $4 \leq i \leq k$, and we define $\IA_\min'$ to be the complement
of the hyperplanes of the form $V(t_{ij})$, $V(t_{ij} -1)$, and $V(t_{ij} - t_{i'j})$.  The
cross-ratios $R_{ij}$ then define a morphism $\CR^\min\colon  \moduli[0,k+n] \to \IA_\min$.  There is a
canonical projection $p_\min \colon \IA_\red \to \IA_\min$, forgetting the indices which do not
appear for $\IA_\min$.
 
\begin{lemma} \label{le:CRproperties}
The morphism $\CR^\min$ is injective and dominant.
The morphism $\CR^\red$ is injective and dominant onto the subvariety $\cY$ of  
$\IA_\red'$ cut out by the equations 
\begin{equation*}
{\rm Eq}(i,j,j')\colon  \quad [1, t_{2j}, t_{3j}, t_{ij}] - [1, t_{2j'}, t_{3j' }, t_{ij'}] = 0
\end{equation*} 
for $i \geq 4$ and $1 \leq j < j' < \ell$.
\end{lemma}
\par
Obviously, the same subvariety is also cut out by all the equations
${\rm Eq}(i,1,j')$ for $i \geq 4$ and $2 \leq j' < \ell$.
\par
\begin{proof} Again, normalize so that $z_1=\infty$ and $z_2 = 0$ 
and $z_3 = 1$. Then $R_{2j} = x_j/y_j$ and  $R_{3j} = (1-x_j)/(1-y_j)$.
Since $R_{2j}\neq R_{3j}$, the knowledge of these two cross-ratios thus determines the location
of $x_j$ and $y_j$. Since $R_{i1} = (z_i-x_1)/(z_i-y_i)$, and $R_{i1}\neq 1$, this cross-ratio
determines the location of $z_i$. This proves injectivity. Dominance
of $\CR^\min$ follows since the dimensions agree.
\par
This argument states more precisely that for any fixed $j$ 
the knowledge of $R_{2j}, R_{3j}$ and $R_{ij}$ determines the location
of $z_i$. In fact, a straightforward calculation yields
$$z_i = [1,R_{2j}, R_{3j}, R_{ij}].$$
These of course have to agree for any pair of indices $j$ 
and $j'$, which is expressed by ${\rm Eq}(i,j,j')$. Consequently, the image of
$\CR^\red$ is contained in the subvariety cut out by all the ${\rm Eq}(i,j,j')$.
Given $R_{2j'}$, $R_{3j'}$, $R_{21}$,$R_{31}$, and $R_{j1}$, one can solve
${\rm Eq}(i,1,j')$ uniquely for $R_{1j'}$. Since $\CR^\min$ is dominant, 
it follows that $\CR^\red$ is dominant to $\cY$.
\end{proof}
\par
\begin{proof}[Proof of Proposition~\ref{prop:3detbytorsion} and 
Lemma~\ref{lemma:3zerosprescres}]
Consider a form $\omega_\infty$ satisfying the conditions i), ii) and iii) in
Definition~\,\ref{def:determined_by_torsion} or Definition\,\ref{def:prescribes_torsion}.  Its image under
$\CR$ and hence also $\CR^\min$ is then a torsion point. By Lemma~\ref{le:CRproperties} the map
$\CR^\min$ has an inverse rational map, hence $\omega_\infty$ is determined up to scale and finitely
many choices by a bound on its torsion orders.  This proves Proposition~\ref{prop:3detbytorsion}.

Consider now the rational map 
$$\Resi\colon  \moduli[0,k+n] \dashrightarrow \proj^n$$
that associates to a point in $\moduli[0,k+n]$
the projective tuple of residues of the corresponding one-form
$$\omega_\infty = \frac {\prod_{j=1}^k (z-z_j)^{m_j} \, dz}
{\prod_{i=1}^\ell(z-x_i)(z-y_i)}.$$
The
rational map $\Resi\circ(\CR^\min)^{-1}$ depends only on the type of the stratum. Consequently, by
\eqref{eq:htupperbd} the $\Resi\circ(\CR^\min)^{-1}$-image of the set of torsion points
has bounded height.
By Northcott's theorem, there are at most finitely many possible residue tuples of degree at most
$g$, 
proving Lemma~\ref{lemma:3zerosprescres}.
\end{proof}
\par
We prepare now for the proof of Proposition~\ref{prop:3presctorsion}.
Suppose that its statement was false for a stratum of some fixed type
$(n;m_1,\ldots,m_k)$ with $k \geq 3$, not specializing to $n=6$
and $m_i=1$ yet. By Laurent's Theorem \cite{laurent},
 this means that there exists a subvariety
$T \subset \IA'$ which is 
\begin{compactitem} 
\item[a)] the translate of a positive-dimensional torus 
by a torsion point, 
\item[b)] generically contained in the image of $\CR$, 
\item[c)] generically contained in the image of the locus of stable forms (meaning
  $\Resi_{x_i}\omega_\infty = -\Resi_{y_i} \omega_\infty$), and 
\item[d)] contained in a fiber of $\Resi'$ (since there are only finitely many possible
  residue tuples by Lemma~\ref{lemma:3zerosprescres}). 
\end{compactitem} 
Here $\Resi'$ is the morphism $\Resi'=\Resi\circ(\CR^\min)^{-1}\circ
p_\min$. 
\par
In a proof by contradiction we may restrict to
the case $\dim T= 1$, i.e.\ 
$$T=\{(c_{abj}t^{e_{abj}}),\,\, 1\leq a < b \leq k,\,\, 1\leq j \leq \ell, \quad  t \in \cx^*\}.$$
We next discuss the possibilities for the limit point $T(0)$ of
the $(\CR^\min)^{-1} \circ p_\min$-image of $T$ as $t \to 0$. This
is a well-defined point in the Deligne-Mumford
compacification $\barmoduli[0,k+n]$ and corresponds to a 
stable curve $Y_0$ in $\partial \barmoduli[0,k+n]$ together
with the limiting stable form $\eta = \lim_{t\to 0} \omega_\infty(t)$.
\par
We represent the dual graph of the stable curve $Y_0$ determined by $T$ 
by a tree $\cT_T$, 
whose vertices are decorated by the zeros and poles that limit
in the corresponding component.
\par
\begin{lemma} \label{le:treeconstraints}
The limiting stable form $(Y_0,\eta)$
associated with $T(0)$ has the following properties.
\begin{compactitem}
\item[i)] Whenever a curve $\gamma$ is pinched
as $t \to 0$, no pair of poles $x_i$ and $y_i$ lies on one side
of $\gamma$ such that two zeros lie on the other side. 
\item[ii)] None of the pinched curves has $\eta$-period equal to zero, in particular
$\eta$ has a pole at each of the nodes of $Y_0$. 
\item[iii)] Each component $Y$ of $Y_0$ has at least one zero.
\end{compactitem}
\end{lemma}
\par
\begin{proof}
Since $T$ is generically contained in $\CR(\moduli[k+n])$,
for each index $(a,b,j)$ either $e_{abi} \neq 0$ or $c_{abi} \neq 1$.
In particular the limit $t \to 0$ of $c_{abi}t^{e_{abi}}$ is not $1$.
This implies the first statement.

Since the residues  $\Resi_{x_i}\omega_\infty(t)$
are $\ratls$-linearly independent and constant, the period of a curve could only be zero if it does not separate
any pair of poles $\{x_i, y_i\}$.  As there are at least four zeros, such a curve must violate i).

Now, since $Y_0$ is a stable curve, each component must have at least three nodes and poles.  Since
each node is a simple pole of the stable form by ii), each component must have at least one zero as well.
\end{proof}
\par
\begin{cor} In the case of a stratum of type $(6;1,1,1,1)$
a complete list of decorated trees arising as $\cT_T$  
is given by the  three possibilities in Figure~\ref{fig:Degeneration graphs}
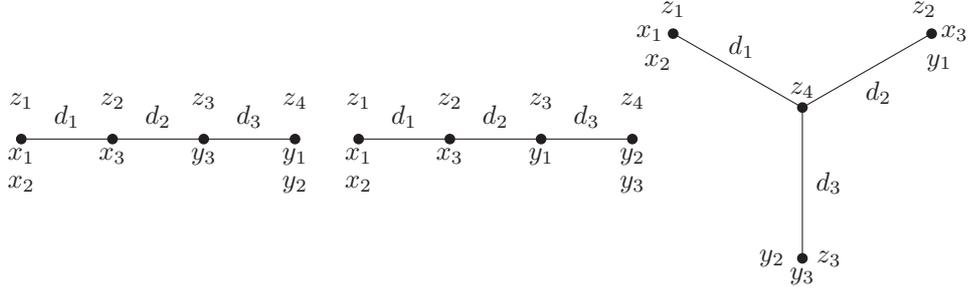
\begin{figure}
\begin{minipage}[hbt]{2.9cm}
	\centering
       \begin{tikzpicture}
	\coordinate[label=below:$x_1$] (x1) at (0,0);
	\coordinate[label=below:$x_3$] (x3) at (1.2,0);
        \coordinate[label=below:$y_3$] (y3) at (2.4,0);
        \coordinate[label=below:$y_1$] (y1) at (3.6,0);
        \node (x2) at (0,-0.6) {$x_2$};
        \node (y2) at (3.6,-0.6) {$y_2$};
        \node (z1) at (0,0.5) {$z_1$};
        \node (z2) at (1.2,0.5) {$z_2$};
        \node (z3) at (2.4,0.5) {$z_3$};
        \node (z4) at (3.6,0.5) {$z_4$};
        \node (d1) at (0.6,0.3) {$d_1$};
        \node (d2) at (1.8,0.3) {$d_2$};
        \node (d3) at (3,0.3) {$d_3$};       
	\draw (x1) -- (y1);
	\fill (x1) circle (2pt);
	\fill (x3) circle (2pt);
        \fill (y3) circle (2pt);
        \fill (y1) circle (2pt);
	\label{(A)}
       \end{tikzpicture}
\end{minipage}
\hspace{1.3cm}
\begin{minipage}[hbt]{2.9cm}
	\centering
       \begin{tikzpicture}
	\coordinate[label=below:$x_1$] (x1) at (0,0);
	\coordinate[label=below:$x_3$] (x3) at (1.2,0);
        \coordinate[label=below:$y_1$] (y1) at (2.4,0);
        \coordinate[label=below:$y_2$] (y2) at (3.6,0);
        \node (x2) at (0,-0.6) {$x_2$};
        \node (y3) at (3.6,-0.6) {$y_3$};
        \node (z1) at (0,0.5) {$z_1$};
        \node (z2) at (1.2,0.5) {$z_2$};
        \node (z3) at (2.4,0.5) {$z_3$};
        \node (z4) at (3.6,0.5) {$z_4$};
        \node (d1) at (0.6,0.3) {$d_1$};
        \node (d2) at (1.8,0.3) {$d_2$};
        \node (d3) at (3,0.3) {$d_3$};       
	\draw (x1) -- (y2);
	\fill (x1) circle (2pt);
	\fill (x3) circle (2pt);
        \fill (y1) circle (2pt);
        \fill (y2) circle (2pt);
	\label{(B)}
       \end{tikzpicture}
\end{minipage}
\hspace{0.7cm}
\begin{minipage}[hbt]{2.9cm}
	\centering
       \begin{tikzpicture}
	\coordinate[label=above:$z_4$] (z4) at (3,3);
        \coordinate[label=below:$y_3$] (y3) at (3,1);
        \coordinate[label=left:$x_1$] (x1) at (1.3,3.98);
        \coordinate[label=right:$x_3$] (x3) at (4.7,3.98);
        \fill (z4) circle (2pt);
	\fill (y3) circle (2pt);
        \fill (x1) circle (2pt);
        \fill (x3) circle (2pt);
        \node (x2) at (1.1,3.6) {$x_2$};
        \node (z1) at (1.3,4.3) {$z_1$};
        \node (z2) at (4.6,4.3) {$z_2$};
        \node (y1) at (4.8,3.6) {$y_1$};
        \node (y2) at (2.6,1) {$y_2$};
        \node (z3) at (3.35,1) {$z_3$};
        \node (d1) at (2.2,3.8) {$d_1$};
        \node (d2) at (4,3.2) {$d_2$};
        \node (d3) at (3.35,2) {$d_3$};
        \draw (3,1) -- (3,3) -- ++ (30:2cm);
        \draw (3,1) (3,3) -- ++ (150:2cm);
        \label{(C)}
       \end{tikzpicture}
\end{minipage}
\caption{Possible stable limits of $T$} \label{fig:Degeneration graphs}
\end{figure}
up to renumbering zeros and poles,  together with trees
obtained by collapsing one or more edges of these three trees.
\end{cor}
\par
\begin{proof}
The number of vertices is bounded by four by Lemma~\ref{le:treeconstraints}~iii)
and there are only two possible trees with four vertices, as listed
in Figure~\ref{fig:Degeneration graphs}, each with one zero that we label as in the figure.
Denote by $x_1$ one of the two poles on the component of $z_1$.
Then $y_1$ lies on the component of $z_3$ or of $z_4$ by 
Lemma~\ref{le:treeconstraints}~i). The remaining case distinction
is now easily completed.
\end{proof}
\par
We attach to every edge $e$ of  $\cT_T$,
equivalently to every curve that is pinched when degenerating to $Y_0$,
the number $d_e$ of right Dehn twists performed by the mododromy of a simple loop around $t=0$
in $T$.  As the monodromy must perform a nontrivial twist around each pinched curve, we interpret
$d_e=0$ as indicating that the edge $e$ has been deleted from $T$.
\par
\begin{prop} \label{prop:exptuple}
  Given two zeros $z_a$ and $z_b$ let $S$ be the oriented segment in $\cT_T$
  joining $z_a$ to $z_b$. For any $j \in \{1,\ldots,\ell\}$ let $S_j$ be the oriented subsegment of $S$
  joining the projection $\bar{y_j}$ of $y_j$ onto $S$ to the projection $\bar{x_j}$ of $x_j$ onto
  $S$.

  Then the exponent in $T$ of the coordinate $R_{abj}$ 
  is $e_{abj} = \pm\sum_{e \in S_j} d_e$, where the sign is positive if $S$ and $S_j$ have the same
  orientation, and negative otherwise.
\end{prop}
\par
\begin{proof}
We assume that $S$ and $S_j$ have the same orientation, as swapping $x_j$ and $y_j$ has the effect
of inverting $R_{abj}$.
  
Normalize the zeros and poles of $\omega_\infty(t)$ so that $z_a=0$, $z_b = \infty$ and $x_j = 1$.
Then $y_j = c_{abj}t^{e_{abj}}$ by definition of the torus.  Let $\gamma$ be a path joining $y_j$ to
$x_j$.  Then the monodromy around $t=0$
sends $\gamma$ to $\gamma + e_{abj} \delta$, where $\delta$ is a loop winding once around $0$.

On the other hand, the edges $S_j$ correspond the pinching loops  which
separates $z_a$ and $y_j$ from $z_b$ and $x_j$.  As
each loop intersects $\gamma$ once,  performing $d_e$ Dehn twists about each loop
sends $\gamma$ to $\gamma+ \sum_{e\in S_j} d_e$.  Comparing these two computations of the monodromy,
the formula for $e_{abj}$ follows.
\end{proof}
\par
The torus-translate $T$ defines via the projection $p_\red\colon  \IA \to \IA_\red$
a torus translate $T_\red$ in $\IA_\red$.
Conversely every translate of a torus  $T_\red \subset  \IA_\red$ determines a
torus-translate $T \subset \IA$, since $R_{abj} = R_{bj}/R_{aj}$. We will thus
subsequently work with tori in $\IA_\red$ only. Such a torus is parameterized as
$$T_\red=\{(c_{ij}t^{e_{ij}}) \colon  1 < i \leq k,\,\, 1\leq j \leq \ell, \quad  t \in \cx^*\}.$$
\par
From now on we restrict our attention to $n=6$ and $m_i=1$, i.e.\ to a
stratum of type $(6;1,1,1,1)$.
\par
\begin{lemma}
  \label{lem:three_matrices}
  Suppose that $T_\red = p_\red(T)$ for some torus translate $T$ satisfying conditions a), b), c),
  and d) above. Then the tuple of exponents $e_{ij}$ is in the row span of one of the following
  three matrices
  $$ M_1 = \left(\begin{smallmatrix}
      1 & 1 & 1 & 1 & 1 & 1 & 0 & 0 & 0 \\
      0 & 1 & 1 & 0 & 1 & 1 & 0 & 1 & 1 \\
      0 & 0 & 1 & 0 & 0 & 1 & 0 & 0 & 0 \\
    \end{smallmatrix}
  \right), \quad M_2 = \left(\begin{smallmatrix}
      1 & 1 & 1 & 1 & 1 & 1 & 0 & 0 & 0 \\
      0 & 1 & 1 & 0 & 1 & 1 & 0 & 1 & 1 \\
      0 & 0 & 0 & 0 & 0 & 1 & 0 & 0 & 1 \\
    \end{smallmatrix}
  \right), \quad M_3 = \left(\begin{smallmatrix}
      1 & 0 & 0 & -1 & 0 & 0 & 0 & 0 & 0 \\
      0 & -1 & 0 & 0 & 0 & 0 & 0 & 1 & 0 \\
      0 & 0 & 0 & 0 & 0 & 1 & 0 & 0 & -1 \\
    \end{smallmatrix}
  \right).
  $$
  The columns appear in the order
  $(R_{21},R_{31},R_{41},R_{22},R_{32},R_{42},R_{23},R_{33},R_{43})$.
\end{lemma}
\par
\begin{proof}
The rows correspond to the exponents associated with $$(d_1,d_2,d_3) \in\{(1,0,0),(0,1,0),(0,0,1)\},$$
calculated according to Proposition~\ref{prop:exptuple}, in each of the three cases in 
Figure~\ref{fig:Degeneration graphs}.
\end{proof}
\par

\begin{proof}[Proof of Proposition~\ref{prop:3presctorsion}]\footnote{This proof is heavily
    computer-based.  All of the assertions below were verified by {\tt sage} \cite{sage}.  The file {\tt
      g3fin_ch9.sage} contains all of the calculations, and is available 
on the authors' web pages.}
  By the above discussion and Lemma~\ref{lem:three_matrices}, we must show that there is no vector
  $N$ contained in the row-span of one of the matrices $M_i$ and corresponding torus-translate $\bba
  T_N\subset\IA_\red'$ satisfying these properties:
  \begin{itemize}
  \item $\bba T_N$ is generically contained in the image $\CR^\red(\moduli[0,10]) \subset
    \mathcal{Y}\subset\IA_\red'$.
  \item  The opposite-residue condition $\Resi_{x_i}\omega_\infty = -\Resi_{y_i}$  is
    satisfied along $\bba T_N$.
  \end{itemize}

  For computational convenience, we work in the full affine plane $\IA_\red\supset \IA_\red'$.  Let
  $h_1,h_2,h_3 \in \ratls[t_{ij}]$ be the numerators of the rational functions ${\rm Eq}(4,1,2)$,
  ${\rm Eq}(4,1,3)$, and ${\rm Eq}(4,2,3)$, which cut out $\mathcal{Y}$ by
  Lemma~\ref{le:CRproperties}.  The ideal $I = (h_1,h_2,h_3)\subset\ratls[t_{ij}]$ is prime, so it
  cuts out the $\ratls$-Zariski-closure $\barcY$ of $\mathcal{Y}$ in $\IA_\red$ (by the same
  argument as in Lemma~\ref{lem:cYgeomirred}, $\barcY$ is the Zariski closure of $\cY$ as $\cY$ has
  smooth rational points, though we do
  not need this).  The rational map $\CR^\red$ is a birational equivalence between $\moduli[0,10]$
  and $\barcY$.

  As we are only interested in torus-translates contained in the image of $\moduli[0,10]$, we define
  the \emph{peripheral divisor} $D_0 = \IA_\min \setminus \CR^\min(\moduli[0,10])$, which we compute
  as follows.  We identify $\moduli[0,10]$ with an open subset of $\cx^7$ with coordinates $z_4,
  x_i, y_i$ (with $i=1,2,3$) using the standard normalization of the zeros, $z_1 = \infty$, $z_2 =
  0$, and $z_3 =1$.  The divisor $D_1 = \cx^7 \setminus \moduli[0,10]$ consists of $21$ hyperplanes,
  each corresponding to the collision of two of the marked points.  Given a hyperplane $H\subset
  D_1$ cut out by an affine polynomial $h$, the numerator of the rational function $h\circ
  (\CR^\min)^{-1}$ cuts out a divisor in $\IA_\min\setminus\CR^\min(\moduli[0,10])$.  We collect the
  irreducible factors of the numerators of these $21$ rational functions, as well as the irreducible
  factors of the denominators of the coordinates of $(\CR^\min)^{-1}$ (which express the condition
  that a marked point is colliding with $z_1=\infty$).  Together these polynomials cut out
  $36$ $\ratls$-irreducible divisors in $\CR^\min$ whose union is the peripheral divisor $D_0$.

  We also define the peripheral divisor $D = p_\min^{-1}(D_0) \subset \IA_\red$, and
  $J\subset\ratls[t_{ij}]$ the corresponding ideal.  The rational map $\CR^\red$ then induces an
  isomorphism $\CR^\red \colon \moduli[0,10] \to \barcY \setminus D$.

  We now apply the torus-containment algorithm to the variety $\barcY$ with initial subspaces $M_1,
  M_2, M_3$.  The result is a list of $554$ subspaces of $\zed^9$ ($3$ of rank three, $97$ of rank
  two, and $454$ of rank one), each contained in the row span of one $M_i$.  For each subspace $N$,
  there is a subvariety $V_N\subset \cx^9$ cut out by an ideal $I_N\subset\ratls[a_{ij}]$
  (with indices $i=2,3,4$, $j=1,2,3$) such that the torus-translate $\bba T_N$ is contained in $\barcY$ if and
  only if $\bba\in V_N$.  We now check that none of these varieties yield a torus-translate
  satisfying all of our constraints.

  First, we rule out those subspaces $N$ which do not in fact yield any torus-translates, that is,
  when $V_N$ is contained in a coordinate hyperplane.  In terms of the torus-containment algorithm,
  this occurs when for one of the polynomials $h_i$ defining $\barcY$, the partition of a
  set of exponent-vectors induced by $N$ has a singleton.  This condition rules out all but $78$ of
  the $554$ subspaces.  

  We then rule out subspaces $N$ parameterizing only peripheral torus-translates, that is, contained
  in the peripheral divisor $D$.  There is a peripheral divisor $D_N\subset \cx^9$ parameterizing
  coefficients $\bba$ such that $\bba T_N \subset D$, cut out by the coefficient ideal (defined in
  \S\ref{sec:torusalgo}) $J_N\subset
  \ratls[a_{ij}]$.  Nonperipheral torus-translates are then
  parameterized by $V_N \setminus D_N$, whose Zariski-closure is cut out by the saturation ideal
  $K_N = \bigcup_{i=1}^\infty I_N:J_N^i$, which we compute for each remaining $N$ (see
  for example \cite{GreulPfister}[\S1.8.9] for information on saturation ideals).  In particular,
  if $K_N = (1)$, then $V_N\subset D_N$, and this $N$ yields only peripheral torus-translates.  This
  condition rules out all but $17$ subspaces, and the ones which remain are only rank-one.

  We now apply the opposite-residue condition to each of the remaining vectors $N$.  The numerators
  of the rational functions $\Resi_{x_i}\omega_\infty / \Resi_{y_i} \omega_\infty +1$ ($i=1,2,3$)
  generate an ideal $L\subset\ratls[t_{ij}]$ with coefficient ideal $L_N\subset\ratls[a_{ij}]$,
  which cuts out the variety parameterizing those $\bba$ such that $\bba T_N$ satisfies the
  opposite-residue condition.  We then compute the saturation ideals $\mathfrak{I}_N =
  \bigcup_{i=1}^\infty (K_N + L_N):J_N^i$.
  For all but one of the remaining vectors $N$, we have $\mathfrak{I}_N =
  (1)$, meaning there is no nonperipheral torus-translate corresponding to $N$ which satisfies the
  opposite-residue condition.

  For the final remaining vector $N$, we were not able to compute the saturation ideal
  $\mathfrak{I}_N$ directly.  In this case, we instead compute the primary decomposition of $K_N$.
  For each associated prime $P_i$, we compute that $\bigcup_{i=1}^\infty (P_i + L_N):J_N^i = (1)$,
  which implies that $\mathfrak{I}_N = (1)$.
\end{proof}

%% file: otherstrata.tex
\section{Intermediate strata: Using the torsion condition}
\label{sec:other-strata-genus}

In this section we prove Theorem~\ref{thm:intromainfin} for all remaining strata.  We maintain the
general hypothesis that $f\colon \cX \to C$ is the family over a \Teichmuller curve, generated by an
algebraically primitive Veech surface $(X,\omega)$.
\par
\subsection{The stratum $\omoduli[3](2,1,1)$}

Consider a stable form $(X_\infty, \omega_\infty)$ which is the limit of a cusp of an \APTC in
$\omoduli[3](2,1,1)$.  In order to apply Theorem~\ref{thm:pantsless_finiteness}, we need to check
that there are finitely many possibilities for the non-pants components of $(X_\infty,
\omega_\infty)$.  There are three cases to consider:
\begin{itemize}
\item $(X_\infty, \omega_\infty)$ has a component of type $(4; 2)$ and either two pants components
  or a component of type $(4; 1,1)$.
\item $(X_\infty, \omega_\infty)$ has a component of type $(5; 2,1)$ and one pants component.
\item $(X_\infty, \omega_\infty)$ is irreducible.
\end{itemize}

We now establish finiteness for each of these cases in turn.

\begin{prop}
  \label{prop:42}
  A limiting stable form of an \APTC in $\omoduli[3](2,1,1)$ has no irreducible components of type
  $(4; 2)$.
\end{prop}

\begin{proof}
  Consider an irreducible component $(Y, \eta)$ of a limiting stable form of type $(4; 2)$.  In
  either of the configurations described above, the four poles of $(Y, \eta)$ come in two pairs
  $(x_i, y_i)$, with $i=1,2$, such that $\Resi_{x_i}\eta = - \Resi_{y_i}\eta$ for each $i$.  The
  involution $J$ of $Y$ swapping each pair $(x_i, y_i)$ then satisfies $J^*\eta = -\eta$, in
  particular the unique zero $p$ of $\eta$ is fixed by $J$.

  By Proposition~\ref{prop:HNFisEigsplit}, the form $\eta^\sigma$ with Galois-conjugate residues
  also vanishes at $p$.  This must be a simple zero, since $\eta^\sigma$ is not a constant multiple
  of $\eta$.  But $\eta$ can only vanish to even order at a fixed point of $J$, as $J^*\eta =
  -\eta$, a contradiction.
\end{proof}

\begin{prop} \label{prop:521}
A stable curve of type $(5;2,1)$ prescribes torsion and is determined by
torsion.
\end{prop}
\par
\begin{proof} We compare with the case of type $(4;1,1)$. Here, too, 
there are at most two connected components of the complement of $Y$
in the dual graph. The case that there is only one connected component
follows exactly along the same lines as for  $(4;1,1)$. 
\par
In the case of two connected components we may suppose that
the one-form is 
\begin{equation}
\label{eq:om21zeros}
 \omega_\infty|_Y = \left(\sum_{i=1}^3 \frac{r_i}{z-u_i} + \frac{r_4}{z - x_1} - \frac{r_4}{z - x_2} 
\right) dz 
 = \frac{z^2 dz}{\prod_{i=1}^{3}(z-u_i)\prod_{i=1}^2 (z-x_i)}.
\end{equation}
We may normalize $u_1=1$ and by the torsion condition $u_i = \eta_i$
is a root of unity for $i=2,3$ and $x_2 =  \zeta x_1$  for some
root of unity $\zeta \neq 1$. The residues at $x_1$ and $x_2$ add
up to zero and this amounts to the condition
\begin{equation}
\label{eq:21residue} 
\zeta^2 \prod_{i=1}^3 (x_1-u_i) -  \prod_{i=1}^3(\zeta x_1 - u_i) = 0.
\end{equation}
Since  $\zeta \neq 1$, this polynomial (with coefficients
in $\ratls(\zeta,\eta_i)$) has degree exactly three in $x_1$.
For fixed torsion order, there is a finite number of choices for
the roots of unity and for each of them there are at most three
possibilities for $x_1$. This shows that forms of this type are determined
by torsion.
\par
We may view \eqref{eq:21residue} as defining a hypersurface $H$
in $\IA^4$ with coordinates $x_1,\zeta$,$u_2$ and $ u_3$. Over the open set 
 $\zeta \neq 1$ the projection $Q$ to $\IA^3$ forgetting $x_1$ is finite,
in fact degree three. Since we are interested in points on
$H$ whose image consists of roots of unity, all the possibilities
for $x_1$ have bounded height by Lemma~\ref{lem:htcompare}. Being residues, 
the $r_i$ are images of a rational map $\Resi$ on $H$, and consequently
their heights are bounded, too. Since the ratios $r_i/r_j$ lie in a field
of degree three, we conclude that there is only a finite number
of possibilities for the ratios $r_i/r_j$.
\par
We now normalize $r_1=1$ and fix one of the finitely many choices for the
other $r_i$. We still have to show that there is only a finite number of
roots of unity that give rise to such a tuple of residues. We follow
the argument of \cite[Proposition~13.9]{BaMo12}. Since there is only
one relative period, and since $Q$ is finite, $Q(\Resi^{-1}(r_1:r_2:r_3:r_4))$
is a curve inside $(\cx^*)^3$. If the claim was false, this curve has
to be a translate of a subtorus, in fact a subtorus as explained in loc.\ cit.
If this were true, we could find roots of unity $\eta_2, \eta_3$ and $\zeta$ 
and a function $x_1(a)$ such that (defining $\eta_1 =1$)
\begin{equation}
\label{eq:witha_om21}
 \omega_\infty|_Y = \left(\sum_{i=1}^3 \frac{r_i}{z-\eta_i^a} + \frac{r_4}{z - x_1(a)} - \frac{r_4}{z - \zeta^ax_1(a)} 
\right) dz 
\end{equation}
has a double zero at $z=0$ and a simple zero at $z=\infty$. Clearing denominators, 
we can either use the $z^3$-term or the linear term to solve for $x_1(a)$ and
take the limit $a \to 0$. We obtain
$$ x_1(0) = \frac{(q_2-q_3)r_2 - q_3r_1}{q_1r_4} \quad \text{and} \quad x_1(0) = 
\frac{q_1r_4}{(q_2-q_3)r_2 - q_3r_1}, $$
where $\eta_j = e^{2\pi i q_j}$ and $\zeta = e^{2\pi i q_1}$, which implies that
$$ ((q_2-q_3)r_2-q_3 r_1-q_1r_4)((q_2-q_3)r_2-q_3 r_1+ q_1r_4) = 0.$$
The vanishing of either factor contradicts the fact that $\{r_1,r_2,r_4\}$ is a
$\ratls$-basis of $F$ and this completes the proof that such a stratum prescribes
torsion.
\end{proof}
\par
\begin{prop} \label{prop:6:211}
A stable curve of type $(6;2,1,1)$ at the cusp of a \Teichmuller curve
generated by $(X,\omega) \in \omoduli[3](2,1,1)$ 
prescribes torsion and is determined by torsion.
\end{prop}
\par
The proof combines the fact that the residues are determined up to
finitely many choices by height bounds as in Section~\ref{sec:pres-tors-princ} 
and properties of the Harder-Narasimhan filtration.
\par
\begin{proof} 
The stable form is determined by the location of the 6 poles
and the three zeros, these we may assume to be at $z_1=\infty$,
$z_2 = 0$ and $z_3 = 1$ with $z_1$ 
corresponding to the double zero. As in the 
cases in the principal stratum we obtain a rational map 
$$\Resi\colon  \moduli[0,9] \dashrightarrow \proj^3$$
that associates to a point in $\moduli[0,9]$
the projective tuple of residues of the corresponding one-form
$$\omega_\infty = \frac {z(z-1) \, dz}
{\prod_{j=1}^3(z-x_j)(z-y_j)}.$$
The cross ratios $R_{2j} = y_j/x_j$ and $R_{3j} = (y_j-1)/(x_j-1)$
are roots of unity by the torsion condition.
By Lemma~\ref{lemma:3zerosprescres}, the fact that roots of unity
have bounded height implies there are at most finitely many
residue tuples $(r_1:r_2:r_3)$ lying in a field of degree three.
\par
Fix one of these residue tuples. If the statement of the proposition
was wrong, then there is a translate of a torus contained
in the fibers of $\Resi$ over such a tuple. We will rule out that
there is a one-dimensional such torus $T$, given by 
$$R_{2j} = c_jt^{e_j} \quad \text{and} \quad R_{3j} = d_jt^{f_j}.$$
We now use the conditions imposed by the Harder-Narasimhan filtration
to conclude that on the one hand 
\begin{equation*}
 \omega_\infty  = \left(\sum_{j=1}^3 \frac{r_j}{z-x_j} - \frac{r_j}{z - y_j} 
\right) dz 
\end{equation*}
has a double zero at $\infty$, and that on the other hand by Proposition~\ref{eq:HNfiltg3}
\begin{equation*}
 \omega_\infty^\sigma  = \left(\sum_{j=1}^3 \frac{r_j^\sigma}{z-x_j}
 - \frac{r_j^\sigma}{z - y_j} \right) dz 
\end{equation*}
also has a zero at $z_1=\infty$. Equating for the top degree terms in 
the numerator we find
that $\sum_{j=1}^3 r_j (x_j - y_j) = 0$ and  $\sum_{j=1}^3 r_j^\sigma (x_j - y_j) = 0$.
This implies that
$$(x_1-y_1: x_2-y_2 : x_3-y_3) = (s_1^\tau :  s_2^\tau:  s_3^\tau)$$
where the $s_i$ are the dual basis of $(r_1,r_2,r_3)$. From this we
deduce that the ratios $(x_i-y_i)/(x_j-y_j)$ have to be constant along $T$. Now
this differences are expressed in cross-ratios as 
$ x_j-y_j = (1-R_{2j})(1-R_{3j})/(R_{2j}-R_{3j})$, so that we have to rule out
that
$$ \frac{(1-c_1t^{e_1})(1-d_1t^{f_1}) \, (c_2t^{e_2}-d_2t^{f_2})}
{(1-c_2t^{e_2})(1-d_2t^{f_2}) \, (c_1t^{e_1}-d_1t^{f_1})}  = {\rm const}.$$ 
Using the valuation at $t=0$ of the numerator and denominator we deduce
that one of $e_1 = e_2$,  $e_1 = f_1$, $e_2 = f_2$ or $f_1 = f_2$ holds.
Switching the roles of $x_1$ and $y_1$ or $x_2$ and $y_2$ we may assume
that in fact $e_1 = e_2$ together with $e_1<f_1$ and $e_2<f_2$ hold. 
Degree considerations now imply that $f_1 = f_2$.
If $e_1 = 0$ then both $d_1 = d_2$ and $d_2/c_2 = d_1 /c_1$ hold or
$d_1 = d_1/c_1$ and $d_2 = d_2/c_2$. The first case is a contradiction since
$x_1 = y_1$ and the second case is a contradiction since $R_{21} = 1$, i.e.\ the 
pole of $\omega_\infty$ coincides with the zero along $T$. If $e_1>0$ then
$t^{f_i}$ is the largest $t$-power appearing in one of the linear factors.
We deduce that  $d_1=d_2$ and hence $R_{31} = R_{32}$ along $T$.
The case $e_1<0$ is ruled out the same way.
\end{proof}
\par
\begin{proof}[Proof of Theorem~\ref{thm:intromainfin}, case ${\omoduli[3]}(2,1,1)$]
  By Propositions~\ref{prop:42}, \ref{prop:521}, and \ref{prop:6:211}, the collection of \APTCs in
  this stratum is pantsless-finite, so finiteness follows by Theorem~\ref{thm:pantsless_finiteness}.
\end{proof}


\subsection{The stratum $\omoduli[3](2,2)^\hyp$}

\begin{proof}[Proof of Theorem~\ref{thm:intromainfin}, case ${\omoduli[3]}(2,2)^\hyp$]
See  \cite[Theorem~3.1]{Mo08}
\end{proof}
\par
In the language introduced above, the main ingredient of this paper (besides
a special case of Theorem~\ref{thm:moduli_bound} using N\'eron models) is 
that on a hyperelliptic
curve, a stable form of type $(2g; g-1,g-1)$ prescribes torsion and is
determined by torsion.

\subsection{The stratum $\omoduli[3](3,1)$}

\begin{proof}[Proof of Theorem~\ref{thm:intromainfin}, case ${\omoduli[3]}(3,1)$]
See \cite[Theorem~13.1]{BaMo12}
\end{proof}

Part of this proof is \cite[Lemma~13.5]{BaMo12} which states in the language
introduced above, that an irreducible stable curve of type $(6;3,1)$ is
{\em not} determined by torsion. This lemma states that this holds only
if we exclude three-torsion.
\par
We will see another instance of this phenomenon in the next case.
\par
\subsection{The stratum $\omoduli[3](2,2)^\odd$}

The hyperelliptic locus in this stratum has been dealt with in \S\ref{sec:hypin22odd} using
\cite{MatWri13}.  It remains to establish finiteness in the nonhyperelliptic locus.
The proof in this case parallels \cite[Section~13]{BaMo12}.  

We will see that
a component of type $(6;2,2)$ is determined by torsion only if we exclude $2$-torsion.
Using the hyperelliptic open-up of \cite{chenmoeller}, these excluded forms only arise as cusps of
\Teichmuller curves in the hyperelliptic locus.
\par
Let $\moduli[0,8]$ be the moduli space of $8$ distinct labeled points 
on $\PP^1$, corresponding to two points $z_1$ an $z_2$  (zeros) and three pairs
of points $x_i$,$y_i$, $i=1,2,3$ (poles). We associate to such a point the
one-form 
$$\omega_P = \frac{z^2dz}{\prod_{j=1}^3 (z-x_j)(z-y_j)}.$$
We normalize usually the two zeros to be at $z_1=0$ and $z_2=\infty$.
With this normalization, $\moduli[0,8]$ is naturally a subset of $\PP^5$.
Let $S(2,2) \subset \moduli[0,8]$  be the locus where $\omega_P$
satisfies the opposite residue condition $\Resi_{x_j} \omega_P = - \Resi_{y_j} \omega_P$
for $j=1,2,3$. The variety $S(2,2)$ is locally parameterized by the projective
4-tuple consisting of the three residues and one relative period, so $S(2,2)$ is 
three-dimensional.
\par
 Define the cross-ratio morphisms $Q_i\colon S(2,2)\to\Gm$ and $R_i\colon S(2,2)\to\Gm$ by
\begin{equation*}
  Q_i = [z_1,z_2,y_i,x_i] \qtq{and} R_i = [x_{i+1},y_{i+1}, x_{i+2}, y_{i+2}],
\end{equation*}
with indices taken mod $3$.  In the above normalization $Q_i = y_i/x_i$.  We
define $Q,\CR\colon S(2,2)\to\Gm^3$ by $Q = (Q_1,Q_2,Q_3)$ and $\CR = (R_1,R_2,R_3)$, and define
$\Resi\colon S(2,2)\to\proj^2$ by $\Resi(P) = (\Resi_{x_i}\omega_P)_{i=1}^3$.  Finally, 
given $\zeta = (\zeta_1, \zeta_2, \zeta_3)\in\Gm^3$, we define $S_\zeta(2,2)\subset S(2,2)$ 
to be the locus where $Q_i = \zeta_i$ for each $i$.
\par
\begin{lemma} \label{lem:22cusp_normal_form} Any irreducible stable form $(X,\omega) \in
  \omoduli[3](2,2)^\odd$ lying over a cusp of an algebraically primitive \Teichmuller curve $C$
  generated by $(X,\omega) \in \omoduli[3](2,2)^\odd$ is equal to $\omega_P$ for some $P \in
  S_{(\zeta_1,\zeta_2,\zeta_3)}(2,2) \in \CR^{-1}(T)$, where the $\zeta_i$ are non-identity roots of
  unity and where $T \subset \Gm^3$ is a proper algebraic subgroup. Moreover, if we normalize the
  components $(r_1:r_2:r_3)$ of $\Resi(P)$ such that $r_1 \in \ratls$, then $\{r_1,r_2,r_3\}$ is a
  basis of some totally real cubic number field.
\end{lemma}
\par
\begin{proof}
The proof of \cite[Lemma~13.4]{BaMo12}, using only the description of boundary
points and the torsion condition, applies verbatim.
\end{proof}
\par
\begin{lemma} \label{le:s22zerodim}
Let $\zeta_i$ be roots of unity, all different from one.  Unless $\zeta_i=-1$
for all $i$, the variety $S_{(\zeta_1,\zeta_2,\zeta_3)}(2,2)$ is zero-dimensional. If
$\zeta_i = -1$ for all $i$, then  $S_{(-1,-1,-1)}(2,2) = S(2,2)$ is two-dimensional.
\end{lemma}
\par
\begin{proof}
 $S_{(\zeta_1, \zeta_2, \zeta_3)}(2,2)$ is cut out by the equations $y_i = \zeta_i x_i$ and 
  \begin{equation}
    \label{eq:33}
    D_{i} = \zeta_i^2 \prod_{j\neq i}(x_i - x_j)(x_i - \zeta_j x_j) - \prod_{j\neq i} (\zeta_i x_i -
    x_j)(\zeta_i x_i - \zeta_j x_j),
  \end{equation}
  which expresses the opposite-residue condition.
  
  Suppose that $S_{(\zeta_1, \zeta_2, \zeta_3)}(2,2)$ has a positive dimensional component.  Then there is a homogeneous polynomial
  $P$ of some degree $d < 4$ which divides $D_k$ for all $k$.  Expanding $D_k$, we obtain
  \begin{equation*}
    D_k = x_k^4\zeta_k^2(1-\zeta_k)(1-\zeta_k^2) + \cdots +\zeta_{k+1}x_{k+1}^2 \zeta_{k+2} 
    x^2_{k+2}(1-\zeta_k^2)(1-\zeta_k),
  \end{equation*}
  with indices taken mod $3$.  Because each $D_k$ contains $x_k^4$ with non-zero coefficient, each monomial
  $x_k^d$ appears in $P$ with non-zero coefficient. We have
  $$P(0,x_2,x_3) = \alpha_2x_2^d + \alpha_3x_3^d + \ldots \quad
  \mid \quad D_1(0,x_2,x_3) = \zeta_2 x_2^2 \zeta_3 x_3^2 (1-\zeta_1^2)(1-\zeta_1).$$
  This is not possible since the
  $\alpha_i$ are nonzero since we may suppose $\zeta_1^2 \neq \pm 1$, possibly after swapping the 
indices. 
\par
\end{proof}

\begin{prop} \label{prop:fin22res}
  There is a finite number of projectivized triples of real cubic numbers $(r_1:r_2:r_3)$
  such that for any irreducible periodic direction on any $(X,\omega) \in
  \Omega\moduli[3](2,2)^\odd$ generating an algebraically primitive \Teichmuller curve,  the projectivized
  widths of the cylinders in that direction is one of the triples $(r_1:r_2:r_3)$.

  In particular, there are only a finite number of trace fields $F$ of algebraically
  primitive \Teichmuller curves in $\Omega\moduli[3](2,2)$.
\end{prop}

\begin{proof}
  By Northcott's Theorem, we need only to give a uniform bound for the heights of the triples
  $(r_1:r_2:r_3)$ of widths of cylinders, or equivalently of residues of limiting irreducible stable
  forms satisfying the conditions of Lemma~\ref{lem:22cusp_normal_form}.
\par
If $\zeta_i = - 1$ for $i=1,2,3$ the resulting stable form over any cusp of the
\Teichmuller curve is hyperelliptic. By the hyperelliptic open-up \cite{chenmoeller}
then the whole \Teichmuller curve parameterizes a family of hyperelliptic curves.
This case has been dealt with in Section~\ref{sec:hypin22odd}.
\par
In the remaining cases $S_{(\zeta_1,\zeta_2,\zeta_3)}(2,2)$ is zero-dimensional
by Lemma~\ref{le:s22zerodim} and the proof of \cite[Proposition~13.7]{BaMo12}
can be copied.
\end{proof}
\par
\begin{prop}
  \label{prop:22fixxy}
  Given a basis $(r_1,r_2,r_3)$ over $\ratls$ of a totally real cubic number
  field, there are only finitely many stable forms over cusps of algebraically primitive \Teichmuller curves in
  $\Omega\moduli[3](2,2)^\odd$ having residues $(r_1,r_2,r_3)$.
\end{prop}
\par
\begin{proof}

  Consider the
  variety $C = \Resi^{-1}(r_1:r_2:r_3)\subset S(2,2)$ of forms having residues $\pm r_i$ and two
  zeros of order $2$.  A dimension count shows that $C$ is at least one-dimensional.  In
  fact, $C$ is exactly one-dimensional, as $C$ is locally parameterized by the single relative
  period of the forms $\omega_P$.  Let $C_0$ be a component of $C$.  We suppose that $C_0$ contains
  infinitely many cusps of algebraically primitive \Teichmuller curves and derive a contradiction.
  Consider the image $Q(C_0)\subset(\cx^*)^3$.  We claim that $Q(C_0)$ is a curve.  If not, and
  $Q(C_0)= (\zeta_1,\zeta_2,\zeta_3)$, then $C_0$ is a component of $S_{(\zeta_1,\zeta_2,\zeta_3)}$, hence
$\zeta_i = -1$ for all $i$ and we are in the hyperelliptic case that has already been dealt with.
\par 
Now since $C_0$ contains infinitely many cusps
of \Teichmuller curves, $Q(C_0)$ must contain infinitely many torsion points of $(\cx^*)^3$ by
Lemma~\ref{lem:22cusp_normal_form}.  From this it follows that $Q(C_0)$ is a translate of a subtorus
of $(\cx^*)^3$ by a torsion point. As in \cite[Proposition~13.9]{BaMo12} one checks that
$Q(C_0)$ is in fact a subtorus of $(\cx^*)^3$, rather than a translate.  
\par
 It remains to show that $Q(C)$ is not a subtorus of $(\cx^*)^3$. If this
  were true, we could find roots of unity $\zeta_i$ and a projective triple $(x_1(a):x_2(a):x_3(a))$
  depending on a parameter $a$, such that for all $a \in \cx$ the differential $$ \omega_\infty =
  \left( \sum_{i=1}^3 \frac{r_i}{z-x_i(a)} - \frac{r_i}{z-\zeta_i^a x_i(a)} \right) dz =
  \frac{p(z)dz}{\prod_i(z-x_i(a))(z-\zeta_i^ax_i(a))}$$ has double zeros at $z=0$ and  
 at $z=\infty$. The vanishing of the $z^4$-term of $p(z)$ implies  
$$ \sum r_i x_i (1-\zeta_i^a) = 0$$ 
and the constant term (divided by $x_1x_2x_3$) also yields a linear equation.  Using the
  normalization $x_1=1$ we may solve the two linear equations for $x_2$ and $x_3$.  We then take the
  limit of $x_2$ and $x_3$ as $a\to 0$, applying l'H\^opital's rule twice. If we let $\zeta_i =
  e^{2\pi i q_i}$ for some $q_i \in \ratls$, we obtain
  \begin{equation}
    \label{eq:lHopi}
    x_2(0) =
    \frac{q_3r_3-q_1r_1}{q_2r_2-q_3r_3} \qtq{and} x_3(0) = \frac{q_2r_2 - q_1r_1}{q_3r_3 - q_2r_2}.
  \end{equation}

We normalize $r_1=1$ and write $\tir_i = q_i r_i$ as shorthand.
Taking the derivative of the $z^3$-term of $p(z)$ with respect to $a$ at $a=0$ and
making the substitution \eqref{eq:lHopi}, we obtain 
$$ \sum_{i=1}^3 \tir_i^3 + 3 \tir_1\tir_2\tir_3 - \sum_{i \neq j} \tir_i \tir_j^2
= 0.$$
and from the constant terms with this substitution the limit $a=0$ is
$$ \tir_3\, \Bigl(-6\tir_1\tir_2\tir_3 + \sum_{i \neq j} \tir_i \tir_j^2\Bigr) = 0$$
Taking the resultant with respect to $r_2$ we obtain
$$ q_1r_2^6q_3r_3\,(q_1 - q_3r_3)\,(q_1 + q_3r_3) = 0.$$
Since $\{r_1=1,r_2,r_3\}$ is a $\ratls$-basis of $F$, this gives the 
contradiction we are aiming for.
\end{proof}
\par
\begin{proof}[Proof of Theorem~\ref{thm:intromainfin}, case ${\omoduli[3]}(2,2)^\odd$]
This is now a consequence of Proposition~\ref{prop:22fixxy} 
and \cite[Proposition~13.10]{BaMo12}.
\end{proof}
